\documentclass[oneside,reqno]{amsart}

\usepackage{macros}
\usepackage[left=2.5cm, right=2.5cm, top=2cm, bottom=4cm]{geometry}
\usepackage[round]{natbib}

\usepackage[colorlinks=true, pdfstartview=FitV, linkcolor=blue,citecolor=blue, urlcolor=blue]{hyperref}

\usepackage{seqsplit}
\usepackage{xstring}

\usepackage{xcolor}

\definecolor{labelkey}{rgb}{0,0,1}
\definecolor{Red}{rgb}{0.7,0,0.1}
\definecolor{Green}{rgb}{0,0.7,0}
\definecolor{jgreen}{rgb}{0,0.7,0}

\numberwithin{equation}{section}

\usepackage{amsfonts, amssymb, amsmath, amsthm, mathrsfs, bbm, cjhebrew, gensymb, textcomp, mathtools, dsfont, calligra}

\usepackage[normalem]{ulem}

\usepackage{commath, setspace, subcaption, parcolumns, multirow, multicol, accents, comment, marginnote, verbatim, empheq, enumerate, stackrel, enumitem, float, physics}
\usepackage[capitalize,nameinlink,noabbrev]{cleveref}

\usepackage{graphicx, graphics, epsfig, psfrag, tikz, tikz-cd,svg}

\usepackage{todonotes}
\usepackage{multicol}

\usepackage{tikz}

\newtheorem{Thm}{Theorem}[section]
\newtheorem{Lem}[Thm]{Lemma}
\newtheorem{Prop}[Thm]{Proposition}
\newtheorem{Cor}[Thm]{Corollary}

\newtheorem{Def}[Thm]{Definition}
\newtheorem{Rmk}[Thm]{Remark}

\newtheorem*{Thm*}{Theorem}

\DeclareMathOperator{\esssup}{ess\ sup}

\newcommand{\T}{\mathbb{T}}

\newcommand{\bbD}{\mathbb{D}}

\newcommand{\OO}{\mathcal{O}}
\newcommand{\PP}{\mathcal{P}}

\newcommand{\Pe}{\mathrm{Pe}}

\newcommand{\goesto}{\rightarrow}

\DeclareFontFamily{U}{matha}{\hyphenchar\font45}
\DeclareFontShape{U}{matha}{m}{n}{
      <5> <6> <7> <8> <9> <10> gen * matha
      <10.95> matha10 <12> <14.4> <17.28> <20.74> <24.88> matha12
      }{}
\DeclareSymbolFont{matha}{U}{matha}{m}{n}
\DeclareFontSubstitution{U}{matha}{m}{n}

\DeclareFontFamily{U}{mathx}{\hyphenchar\font45}
\DeclareFontShape{U}{mathx}{m}{n}{
      <5> <6> <7> <8> <9> <10>
      <10.95> <12> <14.4> <17.28> <20.74> <24.88>
      mathx10
      }{}
\DeclareSymbolFont{mathx}{U}{mathx}{m}{n}
\DeclareFontSubstitution{U}{mathx}{m}{n}

\DeclareMathDelimiter{\vvvert}{0}{matha}{"7E}{mathx}{"17}

\newcommand{\al}{\alpha}
\newcommand{\be}{\beta}
\newcommand{\de}{\delta}
\newcommand{\De}{\Delta}
\newcommand{\eps}{\epsilon}

\newcommand{\gam}{\gamma}
\newcommand{\Gam}{\Gamma}
\newcommand{\kap}{\kappa}
\newcommand{\lam}{\lambda}

\newcommand{\ph}{\varphi}

\newcommand{\si}{\sigma}
\newcommand{\sig}{\varsigma}

\newcommand{\ze}{\zeta}

\newcommand{\tp}{\tilde{p}}

\newcommand{\mbl}{\mathbf{l}}
\newcommand{\mbf}{\mathbf{f}}

\newcommand{\mbg}{\mathbf{g}}

\newcommand{\bg}{\mathbf{g}}
\newcommand{\bh}{\mathbf{h}}
\newcommand{\mbe}{\mathbf{e}}

\newcommand{\bu}{\mathbf{u}}
\newcommand{\bv}{\mathbf{v}}
\newcommand{\bw}{\mathbf{w}}

\newcommand{\mbphi}{\boldsymbol{\phi}}

\newcommand{\Rone}{R}
\newcommand{\bdy}{\partial}
\newcommand{\lb}{(}
\newcommand{\rb}{)}

\newcommand{\til}[1]{{\tilde{#1}}}
\newcommand{\no}[2]{\|#1\|_{#2}}

\newcommand{\jadd}[1]%
{\color{jgreen}}
\newcommand{\jdel}[1]%
{\color{red}}
 \title[Reconstructing wide spectrum forcing in transport-diffusion and NSE]{Reconstruction of wide spectrum forcing in transport-diffusion and Navier--Stokes equations}
 \author{Jochen Br\"{o}cker$^1$, Giulia Carigi$^2$, Tobias Kuna$^3$, Vincent R. Martinez$^4$}

\begin{document}

\begin{abstract}
This article considers the problem of reconstructing unknown driving forces based on incomplete knowledge of the system and its state. This is studied in both a linear and nonlinear setting that is paradigmatic in geophysical fluid dynamics and various applications. Two algorithms are proposed to address this problem: one that iteratively reconstructs forcing and another that provides a continuous-time reconstruction. 
Convergence is shown to be guaranteed provided that observational resolution is sufficiently high and algorithmic parameters are properly tuned according to the prior information; these conditions are quantified precisely. The class of reconstructable forces identified here include those which are time-dependent and potentially inject energy at all length scales.
This significantly expands upon the class of forces in previous studies, which could only accommodate those with band-limited spectra. The second algorithm moreover provides a conceptually streamlined approach that allows for a more straightforward analysis and simplified practical implementation.
\end{abstract}

\maketitle
\date{\today}

\vspace{1em}
\setcounter{tocdepth}{2}

\tableofcontents
\section{Introduction}\label{sec:intro}
Data assimilation is a term used in the geophysical community and refers to the problem of reconstructing the underlying (and time dependent) state of a dynamical system, potentially stochastic, from measurements. 
The measurements are typically not in one--to--one relationship with the underlying state and may be corrupted by noise.
Data assimilation plays an important role in forecasting of dynamical processes (in particular of the atmosphere and oceans), but also in the identification of dynamical models. 
Indeed, data assimilation and model identification are strongly linked; the latter can often be considered as a special case of the former by treating unknown model parameters as part of the dynamical states which are then reconstructed by data assimilation.
Vice versa, approaches to model identification in general require estimates of the dynamical states and thus rely on data assimilation as part of their implementation.
In the present paper, we will analyse simple data assimilation schemes that reconstruct unknown components of the dynamics while simultaneously estimating the underlying states.
We focus on applications to infinite dimensional dissipative systems, namely the two--dimensional Navier--Stokes equations and transport--diffusion equations (for instance for atmospheric aerosols or tracer gases). 
In both cases, the unknown components of the dynamics are additive forcings.
In the case of transport--diffusion equations, these forcings can be interpreted as surface fluxes of the transported tracers. 
Throughout the manuscript, we will refer to all such external drivers of the dynamics simply as {\em forcings}.
The analysis of the transport--diffusion equations is technically easier since its dynamics are linear. We nevertheless develop an analytical study of both the transport--diffusion and Navier-Stokes equations in order to clarify how nonlinearity can influence the class of forcings that can be reconstructed from observations.
Tracer gases and aerosols play an important role in the dynamics of the atmosphere.
Aerosols act as condensation nuclei and thus have a major influence on precipitation.
Tracer gases such as ozone, methane, or CO$_2$ impact the radiative transfer and are thus linked to important atmospheric phenomena such as the ozone hole and the energy budget of the planet (``greenhouse effect''), respectively.
Gases as well as aerosols (especially in the lower troposphere) are common pollutants with strong and potentially adverse effects on the environment, human activity, and health.
Ozone in urban areas as a consequence of smog has health implications, as do the aerosols and dust particles comprising the smog itself.
Volcanic ash clouds can impede air travel, and the presence of smog might require to reduce or shut down road traffic in affected areas.
For these reasons, modelling and forecasting the atmospheric transport of tracers is of great scientific and socioeconomic importance.
We discuss two algorithms that both deal with the estimation of the dynamical state and the reconstruction of forcings simultaneously. 
Both algorithms apply in the context of the Navier--Stokes as well as the transport--diffusion equations.
The first algorithm (which we refer to as the {\em Sieve Algorithm}) works iteratively and permits the asymptotic reconstruction of time-dependent forcings. 
Each iteration leverages the data to estimate the unobserved portion of the state via a feedback control equation in order to produce an approximation of the true state of the system. This, in turn, is processed through the original equation to produce an approximation of the unknown forcing. In this way, the equation is iteratively applied as a ``sieve" to filter out state errors that pollute the estimation of the unknown force.
The second approach (which we refer to as the {\em Nudging Algorithm}) leverages the observations to simultaneously drive the state and model errors towards zero, but only on the observational subspace. One of the main advantages of this algorithm over the Sieve Algorithm is the ease of its implementation, as well as its analysis. However, it is only guaranteed to reconstruct \textit{time--independent} forcings.
Finally, we remark that both algorithms are applicable to stationary (i.e.\ time--independent) transport--diffusion problems.
This setting is relevant in situations where both the forcing as well as the transporting velocity field can be regarded as constant in time.

In what follows, we provide a brief description of the Sieve Algorithm and Nudging Algorithm for each equation, then informally state the convergence results that we obtain for them. One novelty of the class of source terms considered in this study is that the small-scale features of the source in space are functionally enslaved to its large-scale features. This class of forces encompasses those which are restricted entirely to large scales, but additionally contains those whose small scales obey a prescribed power law behavior with respect to length scale. The large scale features of a function, $f$, will be modeled by the low-pass filter, $P_Nf$, representing the low-mode Fourier projection of $f$ onto wavenumbers $|k|\leq N$, so that the small-scale features are represented by the complementary projection, $Q_Nf=(I-P_N)f$.
We refer to the class of sources for which $Q_Nf={F}(P_Nf)$, for some function ${F}$ and some $N$, as a class of {\em functions of quasi--finite rank $N$}.
The map ${F}$ is assumed Lipschitz in some appropriate norms.
\subsection{Sieve Algorithm}\label{sect:sieve:intro}
First we will describe the Sieve Algorithm for the transport-diffusion equation for time-dependent velocity field, then a version of the algorithm for time-independent velocity fields and, finally, we will describe the Sieve Algorithm for the two-dimensional (2D) Navier-Stokes equations (NSE).
\subsubsection{Sieve Algorithm for the Transport-Diffusion equation}\label{sect:def:sieve:td}

We first describe the Sieve Algorithm for reconstructing forcings in the transport--diffusion equations.
Our state space is the $d$--dimensional torus ${\bbT^d}=[0,L]^d$, where $d=2$ or $d=3$. 
For convenience, we will set $L=2\pi$.
Of course, the physical dimension can be recovered upon rescaling.

We consider the transport diffusion equation
\begin{align}\label{eq:td:intro}
\partial_t \phi + \bv \cdot \nabla \phi = \kappa \Delta \phi + g, \quad \phi(0,x)=\phi_0(x),
\end{align}
for some fixed, and possibly time-dependent velocity field, $\bv$, which we assume to be known; the function $\phi_0$ represents the initial distribution of the scalar field and the function $g$ represents forcings or surface fluxes; both $\phi_0$ and $g$ are assumed to be {\em unknown}.
We will work in situations (see \cref{sect:notation:td}) where $\eqref{eq:td:intro}$ has a unique solution $\phi(t,x)= \phi(t,x;\phi_0,g)$ and we frequently drop the $x$-argument when no confusion seems possible. 

Our aim is to reconstruct $g$ at future times, and subsequently $\phi$, from (a)~observations on the scalar field, represented by the time-series $\{P_N\phi(t)\}_{t\geq0}$, consisting of Fourier modes of $\phi$ up to wave-number $|k|\leq N$, and (b)~{\em knowledge of the model} \eqref{eq:td:intro} that gives rise to $\phi$.
We initialize the algorithm by a guess $\psi_0^{(0)}$ for the  initial condition $\phi_0$  and a guess $f^{(0)}$ for the forcing $g$. 
Our guess for the initial condition has the form ${\psi^{(0)}_0=P_N\phi_0+Q_N\psi^{(0)}_0}$.
%
Here $Q_N\psi^{(0)}_0$ is an arbitrary (square-integrable) function, while the low modes $P_N\psi^{(0)}_0$ are set equal to the observations $P_N\phi_0$ at time $t = 0$.
%
The guess for the forcing is assumed to be a potentially time-dependent function of quasi--finite rank $N$, that is, 
$f^{(0)}=l^{(0)}+{{F}}(l^{(0)})$, where $l^{(0)} = P_Nf^{(0)}$ can be chosen  arbitrarily as well.
Though the initial guesses can, in principle, be chosen arbitrarily, surely in practice one would choose guesses that seems reasonable in the hope to enhance performance. 
We refer to $f^{(0)}$ as the \textit{stage-$0$ approximation to} or \textit{reconstruction of} $g$. 
At this initial stage, we propose to reconstruct the unobserved scales in $\phi$ over the time interval $I=[0,\infty)$ via
  \begin{align}\label{eq:nudge:0}
  \bdy_t{\psi^{(0)}}+{\bv\cdotp\nabla{\psi^{(0)}}}={\kap\De{\psi^{(0)}}}+f^{(0)}+\mu P_N \phi-\mu P_N{\psi^{(0)}},\quad {\psi^{(0)}}(0,x)={\psi^{(0)}_0}(x),
  \end{align}
  where $\mu>0$, and $\sigma_0$ is the identity map $\sigma_0 \phi=\phi$.
  Denote the solution of \eqref{eq:nudge:0} (which is unique in the setting considered here) by
    \begin{align}\notag
        {\psi^{(0)}}(t,x)={\psi^{(0)}}(t,x;{\psi^{(0)}_0},f^{(0)},\OO^{(0)}),
    \end{align}
    where $\OO^{(0)}= \{ P_N \sigma_0 \phi(t)\}_{t \geq 0}$ are the available observations 
    (this cumbersome notation will soon make more sense).
    %
    Then the stage-$0$ reconstruction over $I$ of the true scalar field, $\phi$, will be defined as 
    \begin{align}\label{eq:theta:stage:1}
        {\phi^{(0)}(t,x)}:=P_N\phi(t,x; g , \phi_0)+Q_N{\psi^{(0)}(t,x;{\psi^{(0)}_0},f^{(0)},\OO^{(0)})},
    \end{align}
    for any $x \in \bbT^d$ and any $t \in I$.
Subsequently, the stage-$1$ approximation $f^{(1)}$ to $g$ is defined as 
\begin{align}\label{def:force:complete:stage:1}
        f^{(1)}:=l^{(1)}+{{F}}(l^{(1)}),\quad \text{over the interval~}I, 
    \end{align}
where $l^{(1)}$ is obtained from 
\begin{align}\label{def:force:large:stage:1}
        l^{(1)}:=\bdy_tP_N\phi^{(0)}+{P_N(\bv\cdotp\nabla\phi^{(0)})-\kap\De P_N\phi^{(0)}}.
    \end{align}
In particular, the equation \eqref{eq:td:intro} is used as a ``sieve" to filter state errors and produce a new approximation to the large scales of the true forcing.
This approximation, $f^{(1)}$, is not expected to be good as long as the synchronisation error stays large, which may be the case initially. 
More precisely, at stage $0$, the synchronisation error refers to the error between state-approximation $\phi^{(0)}$ and the true state $\phi$. 
Hence we will consider $f^{(1)}$ as an approximation to $g$ only after a time $t_1 \geq 0$. 
Note that we construct the stage-$1$ approximation to the forcing $g$ using the information obtained from stage-$0$ exclusively.
We may thus iterate this procedure.

Consider a sequence of positive time increments $t_1, t_2, \dots$ (we set $t_0 = 0$) and let $\tau_s$ denote the translation operator
\begin{align}\label{def:tau}
        (\tau_s\ph)(t):=\ph(s+t),\quad \text{for all}\ s,t\geq0.
    \end{align}
We will make use of the following shorthand notation:          
\begin{align}\label{def:tau:sigma}
        \tau_j:=\tau_{t_j}\quad\text{and}\quad \sigma_j:=\tau_{t_1 + \ldots + t_j}.
    \end{align}
Assume we have have completed the stage-$(j-1)$ approximation, that is we have constructed ${\phi^{(j-1)}}$ as an approximation of $\si_{j-1}\phi$ and  $f^{(j)}$, constructed from $\phi_{j-1}$, as an approximation of  $\si_{j-1}g$. 
Then using the observations $\OO^{(j)} = \{ P_N \sigma_j \phi(t)\}_{t \geq 0}$ in order to construct the stage-$j$ approximation to $\phi$, we consider ${\psi^{(j)}}$ which satisfies over $I$ the equation
\begin{align}\label{eq:ttheta:stage:j}
  \bdy_t{\psi^{(j)}}+{(\si_j\bv)\cdotp\nabla{\psi^{(j)}}-\kap\De\psi^{(j)}}=\tau_{j}f^{(j)}+\mu P_N\si_{j}\phi-\mu P_N{\psi^{(j)}},\quad {\psi^{(j)}}(0,x)=\psi^{(j)}_0(x),
    \end{align}
where    $\psi^{(j)}_0(x) := \psi^{(j-1)}_0(t_j,x; \psi^{(j-1)}_0, f^{(j)}, \OO^{(j)})$.
Denote the solution of \eqref{eq:ttheta:stage:j} by
    \begin{align}\notag
        {\psi^{(j)}}(t,x)={\psi^{(j)}}(t,x;{\psi^{(j)}_0},\tau_jf^{(j)}, \OO^{(j)}).\notag
    \end{align}
Then the stage-$j$ approximation of $\phi$ is given by
    \begin{align}\label{eq:theta:stage:j}
        {\phi^{(j)}(t, x) }=P_N\si_{j}\phi(t, x; \phi_0, g) +Q_N{\psi^{(j)}(t,x; \psi^{(j)}_0, \tau_jf^{(j)}, \OO^{(j)})},
    \end{align}
    for $x \in \bbT^d$.
    The stage-$(j+1)$ large-scale approximation $l^{(j+1)}$ to $P_Nf$ will be defined by
    \begin{align}\label{def:force:large:stage:j}
        l^{(j+1)}:=\bdy_tP_N{\phi^{(j)}}+P_N {((\sigma_j\bv) \cdotp \nabla{\phi^{(j)}})-\kap\De P_N\phi^{(j)}}.
    \end{align}
Hence, the stage-$(j+1)$ approximation to $g$ will be given by
    \begin{align}\label{def:force:complete:stage:j}
        f^{(j+1)}=l^{(j+1)}+{{F}}(l^{(j+1)}).
    \end{align}
As before, we expect $f^{(j+1)}$ to be an improved approximation to $\si_{j}g$ only after the synchronisation error has decreased, that is after some time $t_{j+1}$.

Hence, the algorithm produces a sequence of forcings, $f^{(1)},f^{(2)},\dots$ and $\phi^{(1)}, \phi^{(2)}, \ldots $ such that $f^{(j)}$ approximates $\si_{j-1}g$ and $\phi^{(j)}$ approximate $\si_{j} \phi$, where we only expect 
$\tau_jf^{(j)}$ to be a good approximation of $\sigma_jg$ and $\tau_{j} \phi^{(j-1)}$ for $\si_{j} \phi$ over the interval $[0,\infty)$, provided that $t_j,\dots, t_1$ are sufficiently large, sufficiently many scales $N$ are observed, and $\mu$ is suitably tuned. 
In other words, $f^{(j)}$ is considered  as a good approximation for $g$ and $\phi^{(j-1)}$ for $\phi$ for times $t \in [t_1 + \ldots + t_j, \infty)$ if we shift $f^{(j)}$ and $\phi^{(j-1)}$ appropriately, that is $\tau_{-t_1 - \ldots - t_{j-1}}f^{(j)}$ and  $\tau_{-t_1 - \ldots - t_{j-1}}\phi^{(j-1)}$.
In general though, the algorithm may be defined for any sequence of times $t_j\geq0$, for $j\geq1$, and $t_0=0$. 

Ultimately, we are able to prove the following (see \cref{thm:converge:sieve:td:detailed} for details):
\begin{Thm}\label{thm:main:sieve:td:intro}
  Let $d = 2$ or~$3$.
Assume that $f$ is of quasi-finite rank $N_0$.
Provided that $N \geq N_0$ is large enough and that $\mu$ is chosen accordingly, depending on $\kappa, N$ and the size of $\bv$, there exists a relaxation time $t_* > 0$ such that the sequence of quasi-finite forces $\{f^{(j)}\}_{j\geq0}$ constructed above satisfies
\begin{align}\label{eq:converge:sieve:td}
        \sup_{t\geq 0}\Abs{\tau_{j+1}f^{(j+1)}(t)-\si_{j+1} g(t)}{}\to 0,\quad\sup_{t\geq0}|\tau_{j+1}\psi^{(j)}(t)-\si_{j+1}\phi(t)|\goesto0, \quad \text{as $j \to \infty$},
    \end{align}
for $t_{j} := j t_*$ for all $j \in \N$.
\end{Thm}
%
Here, $\Abs{\cdotp}$~stands for a generic norm.
We state a precise version of this theorem in \cref{thm:converge:sieve:td:detailed} and \cref{thm:converge:sieve:td:detailed:strong}, where, in particular, $\Abs{\cdotp}$ will denote different Sobolev norms. 
%
\subsubsection{Sieve Algorithm for Stationary Transport Equation}\label{sect:stationary:td}
In applications where the velocity field $\bv$ as well as the forcing $g$ are time independent, the solution $\phi$ of the transport equation~\eqref{eq:td:intro} will asymptotically be independent of time as well, satisfying the equation
\begin{equation}
\bv \cdot \nabla \phi = \kappa \Delta \phi + g.\notag
\end{equation}
The following version of the Sieve Algorithm for reconstructing both $\phi$ and the forcing $g$ from a {\em single} observation of the form $P_N\phi$ is therefore of interest.
Let $\phi^{(j-1)}$ be the stage-$(j-1)$ approximation to $\phi$.
These quantities are time independent.
The stage-$j$ approximation $l^{(j)}$ to the large scales $P_Nf$ of the forcing is defined through
\begin{equation}
        P_N (\bv \cdotp \nabla \phi^{(j-1)})
        - \kap \De P_N\phi^{(j-1)} =: l^{(j)}.\notag
    \end{equation}
Hence, the stage-$j$ approximation to $g$ will be given by
\begin{align}
        f^{(j)}=l^{(j)}+{{F}}(l^{(j)}).\notag
    \end{align}
The stage-$j$ approximation to $\phi$ is basically constructed as in the time--dependent case, although no time shifts are necessary, and we only retain the final state of the synchronisation step.
More specifically, 
\begin{align}
        {\phi^{(j)}} = P_N\phi+Q_N{\psi^{(j)}(t_j, \cdot)},\notag
    \end{align}
where ${\psi^{(j)}}$ satisfies over ${I=(0,\infty)}$
    \begin{align}
        \bdy_t{\psi^{(j)}}+{\bv\cdotp\nabla{\psi^{(j)}}-\kap\De\psi^{(j)}}=f^{(j)}+\mu P_N\phi-\mu P_N{\psi^{(j)}},\quad {\psi^{(j)}}(0,\cdot)=\phi^{(j-1)}.\notag
    \end{align}
    For this algorithm, the statement of \cref{thm:main:sieve:td:intro} holds verbatim (except that the supremum in Eq.~\eqref{eq:converge:sieve:td} is superfluous).
    The proof follows the same lines as that of \cref{thm:main:sieve:td:intro} and will be omitted for the sake of brevity.%
\subsubsection{Sieve Algorithm for the 2D NSE}\label{sect:def:sieve:nse}
We will consider incompressible Navier-Stokes equations as given by
\begin{align}\notag
       \bdy_t\bu+\bu\cdotp\nabla\bu=-\nabla p+{\nu}\De\bu+\mbg,\quad \nabla\cdotp\bu=0,
    \end{align}
over the periodic box $\bbT^2=[0,2 \pi]^2$, where $\bu,\mbg$ are assumed to be mean-free, that is $\int_{\bbT^2} \bu = \int_{\bbT^2} \mbg = 0$.
Let ${\PP}$ denote the Leray-Helmholtz projection onto divergence-free vector fields. We note that $\PP$ commutes with $\De$ in the setting of periodic boundary conditions. We ultimately consider the projected system 
    \begin{align}\label{eq:nse:intro}
        \bdy_t\bu+B(\bu, \bv)={\nu}\De\bu+{\PP}\mbg ,
    \end{align}
    with the bilinear form $B(\bu, \bv) = \PP (\bu \cdot \nabla \bv)$.
    The Sieve Algorithm for 2D NSE will proceed in a sequences of stages similar to the Sieve Algorithm for the transport-diffusion equation and will be applied to reconstruct the \textit{non-potential} component, $\PP\mbg$, of $\mbg$. Since the potential component of $\mbg$ can be absorbed into $-\nabla p$, we will assume for the sake of convenience that 
    \begin{align}
        \mbg=\PP\mbg.\notag
    \end{align}
    We will work in situations where \eqref{eq:nse:intro} has unique solution $\bu(t)=\bu(t;\bu_0,\mbg)$ corresponding to initial condition $\bu(0)=\bu_0$ over $t\in I=[0,\infty)$, and we define the algorithm as follows. 
      As before, at stage $j=0$, we initialize the algorithm with a guess, ${\mbl^{(0)}}={P_N}{\mbl^{(0)}}$, for the large-scale features of the force $g$ and subsequently get a guess for the full force $g$
      \begin{equation}
        {\mbf^{(0)}}:={\mbl^{(0)}}+{{F}}({\mbl^{(0)}}).\notag
    \end{equation}
    Furthermore, as a guess $\bv_0^{(0)}$ for the initial condition, we set $\bv_0^{(0)}=P_N\bu_0+Q_N\bv_0^{(0)}$, where $Q_N\bv_0^{(0)}$ can be chosen arbitrarily.
    Observations are given as $\OO^{(0)}=\{P_N \bu(t)\}_{t\geq0}$.
    With $\bv_0^{(0)}, \mbf^{(0)}$ given, we solve
\begin{align}\label{eq:bv0}
       \bdy_t\bv^{(0)}+B(\bv^{(0)},\bv^{(0)})
       =\nu \De\bv^{(0)}+{\mbf^{(0)}}+\mu {P_N}\bu-\mu {P_N}\bv^{(0)},\quad \bv^{(0)}(0)=\bv_0^{(0)},
   \end{align}
over $t\in I$, where $\mu>0$.
We denote the unique solution of \eqref{eq:bv0} by $\bv^{(0)}(t)=\bv^{(0)}(t;\bv_0^{(0)},\mbf^{(0)},\OO^{(0)})$ (we dropped the $x$ dependence to simplify notation).
%
The high modes of $\bv^{(0)}$ are now combined with the observations to form our stage-$0$ approximation of the true velocity:
%
\begin{align}
  {\bu^{(0)}(t)
    :={P_N}\bu(t;\bu_0,\mbg)+{Q_N}\bv^{(0)}(t;\bv_0^{(0)},\mbf^{(0)},\OO^{(0)})}. \notag
\end{align}
We obtain an approximation of the true large-scale features of the forcing, ${P_N}\mbl$, via
    \begin{align}\label{eq:g1}
        \mbl^{(1)}:=\bdy_t{P_N}\bu^{(0)}-\nu {P_N}\De \bu^{(0)}+{P_NB}(\bu^{(0)},\bu^{(0)}).
    \end{align}
    Finally, we conclude the stage-$0$ with proposing the following approximation of $\mbg$, the true forcing:
    \begin{align}\label{eq:G1}
        {\mbf^{(1)}}:=\mbl^{(1)}+{{F}}(\mbl^{(1)}).
    \end{align}
Although $\mbf^{(1)}$ is defined on $I$, we will only consider it as an improvement to $\mbf^{(0)}$ for times $t\geq t_1$, where $t_1>0$ is taken sufficiently large. 
This will conclude stage-$0$.
At stage-$j$ for $j>0$, we are given $\bv^{(j-1)}(t; \bv_0^{(j-1)},\mbf^{(j-1)},\OO^{(j-1)})$ for $t\ge 0$ as well as $f^{(j)}$.
Taken $t_j >0$ sufficiently large, and with observations $\OO^{(j)}=\{P_N\si_j\bu(t)\}_{t\geq0}$, we consider the solution $\bv^{(j)}=\bv^{(j)}(t;\bv_0^{(j)},\mbf^{(j)},\OO^{(j)})$ of the equation
\begin{align}\label{eq:vj}
        \bdy_t\bv^{(j)}+B(\bv^{(j)},\bv^{(j)})=\nu \De\bv^{(j)}+{\tau_j\mbf^{(j)}}+\mu{P_N}\si_j\bu-\mu {P_N}\bv^{(j)},\quad \bv^{(j)}(0)=\bv^{(j)}_0,
    \end{align}
for $t\in I$, initialized by $\bv_0^{(j)} := \bv^{(j-1)}(t_j; \bv_0^{(j-1)},\mbf^{(j-1)},\OO^{(j-1)})$, where $\si_j$ and $\tau_j$ are defined as in Equations~(\ref{def:tau},\ref{def:tau:sigma}).
We then propose a reconstruction of the true velocity by
    \begin{align}
      {\bu^{(j)}(t)
        :={P_N}\si_j\bu(t;\bu_0,\mbg)
        + {Q_N}\bv^{(j)}(t;\bv_0^{(j)},\mbf^{(j)},\OO^{(j)})},
      \quad \bu^{(j)}(0)
      = \bu_0^{(j)}={P_N}\si_j\bu+{Q_N}\bv_0^{(j)},\notag
    \end{align}
    for $t\in I$.
    Thus, an approximate reconstruction of the true large-scale features of the force will be given by
    \begin{align}\label{eq:gjplus1}
        {\mbl^{(j+1)}}:=\bdy_t{P_N}\bu^{(j)}-\nu {P_N}\De\bu^{(j)}+{P_N}B(\bu^{(j)},\bu^{(j)}).
    \end{align}
We then define the approximate reconstruction of ${\mbg}$ by
    \begin{align}\label{eq:Gjplus1}
        {\mbf^{(j+1)}}:={\mbl^{(j+1)}}+{{F}}({\mbl^{(j+1)}}).
    \end{align}
This force can only be expected to be an {improvement over $\mbf^{(j)}$} for times $t_{j+1}>0$ sufficiently large. This will conclude stage $j>0$.
Proceeding in this fashion, we generate a sequence of forces
   $
        {\mbf^{(0)}}, {\tau_1\mbf^{(1)}}, {\tau_2\mbf^{(j)}}\dots
    $
    and a sequence of velocity fields
    $
    \tau_1\bu^{(0)}, \tau_2 \bu^{(1)}, \tau_3 \bu^{(2)} \dots .
    $
The remarks made at the end of the construction of the Sieve Algorithm for the transport-diffusion equation in \cref{sect:def:sieve:td} also apply for the case of 2D~NSE.
Our main result for this algorithm is then the following (see \cref{thm:main:sieve:a} for its precise version):
\begin{Thm}\label{thm:main:sieve:nse:intro}
  Assume that ${\mbg}$ is of quasi--finite rank $N_0$ and that $d = 2$.
  Provided that $N \geq N_0$ is large enough and that $\mu$ is chosen accordingly, depending on $\kap, N$, and the size of $\mbg$, there exists a relaxation time $t_*>0$ such that the sequence of quasi-finite rank $N$ forces $\{{\mbf^{(j)}}\}_{j\geq0}$ and approximate states $\{\bu^{(j)}\}_{j\geq0}$ constructed above satisfies
    \begin{align}\notag
        \sup_{t\geq0}\Abs{\tau_{j+1}\mbf^{(j+1)}(t)-\si_{j+1}\mbg(t)}\goesto0,\quad \sup_{t\geq0}|\tau_{j+1}\bu^{(j)}(t)-\si_{j+1}\bu(t)|\goesto0,\quad \text{as $j\goesto\infty$},
    \end{align}
for $t_j=jt_*$, where $j\in\N$. 
\end{Thm}
Again, $\Abs{\cdotp}$~stands for a generic norm here while the precise version of this theorem (\cref{thm:converge:sieve:td:detailed}) contains several statements for different Sobolev norms.
\subsection{Nudging Algorithm}\label{sect:nudging:intro}
In contrast to the Sieve Algorithm, the Nudging Algorithm processes observations in a more straightforward way, in the manner of an on-line continuous data assimilation algorithm.
As before, we describe the algorithm and main result for the transport-diffusion equation, as well as the Navier-Stokes equation afterwards.

Suppose $\phi$ is a solution to \eqref{eq:td:intro} with known velocity field $\bv$, but unknown source $g$. The corresponding nudging system for source-reconstruction is then given by

\begin{align}\label{eq:nudging:td}
\begin{split}
  \pdd_t \psi + \bv \cdot \nabla (P_N \phi + Q_N \psi)
  & = \kappa \Delta (P_N \phi + Q_N \psi) + l + {F}(l)
   + \mu_1( P_N\phi -P_N \psi),\\
   \pdd_t l & = \mu_2 (P_N\phi - P_N\psi),
\end{split}
\end{align}
with initial conditions $\psi|_{t = 0} = \psi_0$ and $f|_{t=0} = f_0$, where $\mu_1, \mu_2 > 0$ are the corresponding nudging parameters, $\bv$ is the same vector field appearing in \eqref{eq:td:intro} for the solution $\phi$. 
Our main result regarding \eqref{eq:nudging:td} is the following (see \cref{thm:1.30} for precise version):
\begin{Thm}\label{thm:nudging:intro}
  Assume that $g$ is independent of time and of quasi--finite rank $N_0$.
  Provided that $N \geq 0$ is large enough and that $\mu$ is chosen accordingly, depending on $\kap, N$, and the size of $\bv$, one has
    \begin{align}\notag
    \Abs{f(t) - g} \to 0\quad\text{and}\quad \Abs{{\psi(t) - \phi(t)}} \to 0, \quad\text{as $t\to\infty$},
    \end{align}
at an exponential rate.
\end{Thm}
%
As before, the precise version of this theorem holds with several different Sobolev norms in place of $ \Abs{\cdotp}$.
%
The Nudging Algorithm can also be used no matter whether the velocity field $\bv$ is time dependent or not. 
Unlike the Sieve Algorithm, however, we do not have a more economical version of the Nudging Algorithm when the $\bv$ is time-independent (see \cref{sect:stationary:td}).
On the other hand, like the Sieve Algorithm, the nudging approach can also be adapted to accommodate nonlinear systems such as the 2D NSE \eqref{eq:nse:intro}. Indeed, we consider the system
    \begin{align}\label{eq:NS_v}
        \begin{split}
        \partial_t \bv & + P_N B(P_N \bu + Q_N \bv, P_N \bu + Q_N \bv) + Q_NB(P_N \bu + Q_N \bv,\bv) \\
        &= \nu \Delta (P_N\bu + Q_N\bv) + \mbl + {F} (\mbl) + \mu_1 (P_N\bu- P_N\bv),\\
        \partial_t \mbl &= \mu_2 (P_N\bu- P_N\bv),
        \end{split}
    \end{align}
with the bilinear form $B(\bu, \bv) = \PP (\bu \cdot \nabla \bv)$ and with initial conditions $\bv(0) = \bv_0$, $\mbl(0) = \mbl_0$, for some $\mbl_0$ such that $\mbl_0=P_N\mbl_0$. Here, $\bu$ is solution of \eqref{eq:nse:intro}, and $\mu_1,\mu_2>0$ are the nudging parameters. 
Note that slightly different versions of the bilinear form are applied in the low versus the high modes in Equation~\eqref{eq:NS_v}.
We are now ready to state the convergence result for the nudging-based algorithm for reconstructing external forces in the 2D-NSE (see \cref{thm:NS_nudg} for a precise statement):
\begin{Thm}
  Assume that $g$ is independent of time and of quasi--finite rank $N_0$.
  Provided that $N \geq 0$ is large enough and that $\mu$ is chosen appropriately, depending on $\nu, N$, and the size of $\mbg$, one has 
\begin{align}\notag
   \Abs{\mbf(t) - {\bg}} \goesto 0 \quad \text{and}\quad
   \Abs{\bv(t) - \bu(t)} \goesto 0, \quad\text{as $t\to\infty$},
\end{align}
at an exponential rate.
\end{Thm}

The main differences between the Sieve and Nudging Algorithm described above can be briefly summarized by follows: while the framework of the Nudging Algorithm presented above can only accommodate reconstruction of \textit{time-independent} forces, it provides a conceptually streamlined approach, both in its analytical study and in its practical implementation, when compared to the Sieve Algorithm. Indeed, although the Sieve Algorithm is able to provably reconstruct unknown time-dependent forces, the manner in which it does so requires the user to carry forward the historical information for each of the previous approximations to the unknown force that it generates along the way. In contrast, the Nudging Algorithm requires one to simply integrate the nudging system forward-in-time, so that the approximate forcing is generated continuously in time. We reserve a systematic numerical study on the efficacy, efficiency, and practical limitations of the two algorithms for a future work. We refer the reader to \cref{sect:conclusion} for a more detailed comparison between the results obtained between the two cases and two systems.

\subsection{Relation to Previous Work}\label{sec:previous_work}
Identification of dynamic models from noisy and incomplete observations is a vast field with a long history. 
More recently, ideas from machine learning have been brought into this field, which was then rebranded in the geosciences as combining data assimilation and machine learning. 
These works have drawn due attention to the fact that reconstructing dynamic models from data in a realistic setting is inextricably linked to data assimilation.
Specifically about the problem of reconstructing forcings, a substantial body of literature can be found, mainly in the context of transporting atmospheric tracers, a case of particular scientific and societal importance.
To the best of our knowledge, there is much less work devoted to the reconstruction of forcings in infinite-dimensional, nonlinear equations such as Navier--Stokes and related equations of geophysical fluid dynamics.
Broadly speaking, existing work on reconstructing source terms of atmospheric tracers falls into two categories.
Works of the first category focus directly on identifying ``clouds'' of atmospheric pollutants by estimating the parameters of explicit functional representations of such clouds, called {\em puffs} or {\em plumes}~\citep{HUTCHINSON2017130}.
These are often either exact or approximate solutions of the {\em stationary} transport--diffusion equation in either two or three dimensions~\citep{Sharan_mathematical_model_dispersion_pollutants_low_wind_1996}.
The transporting velocity as well as the diffusion are considered spatially homogenous, the concentration decays to zero at infinity, a no--flux condition is imposed at the ground in three dimensional situations, and the sources are taken as one or several Dirac~delta functions.
A plethora of methods for estimation of the plume parameters have been discussed in the literature, including least squares~\citep{SINGH2015402} and various Bayesian approaches~\citep{SENOCAK20087718}.
Particularly challenging is the estimation of source numbers and locations; implementations use for instance MCMC or ABC~\citep{ANNUNZIO2012580,THOMSON20071128}.
These methods are not dynamical in character and presumably more suitable for the assessment of single discharge events from relatively few ($\sim \mathcal{O}(10)$) sources.
The second category uses data assimilation approaches to estimate the tracer concentrations or the forcings (or both) in a dynamical fashion, akin to the approaches in the present paper.
The unknown forcings are treated as additional dynamic states and are estimated along with the tracer concentrations.
Standard techniques of data assimilation are used such as 4DVAR~\citep{elbern_emission_estimation_4DVAR_2007,Elbern1999ImplementationOA,elbern_Ozone_episode_analysis_4DVAR_2001} or various versions of the Kalman Filter or Kalman Smoother~\citep{wu_observability_chemical_emissions_atmospheric_2022,wu_observation_locations_2016} in case the focus is on estimating the forcings.
If the governing equations are linear and the stochastic perturbations are Gaussian, these techniques even satisfy various criteria of optimality.
The approaches discussed in the present paper have similarities in that they are dynamical in character and can be regarded as leveraging a very simple data assimilation scheme (namely nudging) to the simultaneous estimation of dynamical states and model components. 

The nudging approach is a classical one going back to  \citet{HokeAnthes1976} and was introduced to in the context of systems of ordinary differential equations (ODEs). In a seminal work~\citep{DuaneTribbiaWeiss2006}, it was observed that the ability of nonlinear systems to synchronize~\citep{PecoraCarroll1990} could be facilitated through nudging and therefore leveraged for the purposes of data assimilation, even for partial differential equations (PDEs). Many studies on nudging and synchronization-based techniques for data assimilation have since been carried out since the classical work of~\citet{HokeAnthes1976}, but mostly in the context of nonlinear systems of ODEs~\citep[see][]{ZouNavonLedimet1992, AurouxBlum2008, PazoCarrassiLopez2016, PinheirovanLeeuwenGeppert2019}.
The convergence of the nudging scheme as applied to PDEs was investigated without mathematical rigour until the works of \citet{BlomkerLawStuartZygalakis2013, AzouaniOlsonTiti2014}, where convergence analysis of the nudging scheme was established in the context of the two-dimensional Navier-Stokes equations.
These two works have since served as a paradigm for establishing rigorous results regarding, for instance, accuracy and stability, and in particular, for obtaining quantitative error estimates in the context of nonlinear PDEs. These ideas have recently matured enough to carry out analyses of sensitivity, consistency, asymptotic normality, and convergence for parameter estimation in nonlinear PDEs~\citep[see for instance][]{CialencoGlattHoltz2011, CarlsonHudsonLarios2020, CarlsonLarios2021, CarlsonHudsonLariosMartinezNgWhitehead2021, Martinez2022viscosity}. In a recent work \cite{Martinez2022force}, an algorithm was introduced for reconstructing unknown source terms in the context of the 2D NSE and its convergence analyses were carried out. In \cite{FarhatLariosMartinezWhitehead2024}, a similar algorithm was also studied, both analytically and numerically, where reconstruction of the external forcing was robustly observed. Even in the regime of perfect model and perfect observations, the problem of source reconstruction is decidedly non-trivial due to the presence of nonlinearity. It is shown in \cite{Martinez2022force, FarhatLariosMartinezWhitehead2024} that one could nevertheless leverage the ability of the system itself to systematically filter errors with continuous observation of sufficiently many length scales of the state. It is important to point out that the results of both works \cite{Martinez2022force, FarhatLariosMartinezWhitehead2024} require the range of length scales excited by the force to be limited to the those that are already observed. In this regard, the present work constitutes a notable extension of those previous results by allowing the force to inject energy into system beyond the length scales which are directly observed. One notable example of forcings that quasi-finite class can accommodate are ones that are highly localized in space, i.e., those with exponentially decaying Fourier spectrum. In fact, forcings whose spectrum decay coherently, e.g., according to some power law, belongs to the class of quasi-finite forcings.

In spite of these recent advances, there is no claim to optimality with those methods (or the methods discussed in the present paper).
Rather, their advantage lies in simplicity and feasibility of implementation.
For similar reasons, the presented approaches permit a satisfactory mathematical analysis regarding their asymptotic behaviour. 
To the best of our knowledge, no such analysis has so far been carried out for any other data assimilation method when applied to the simultaneous reconstruction of state and forcings in transport--diffusion and two--dimensional Navier--Stokes models.

Finally, we note an interesting connection between the nudging algorithm as in Equation~\eqref{eq:nudging:td} and concepts from Control Theory, more specifically robust control. 
The nudging algorithm achieves the asymptotic reconstruction of the (partially observed) state $\phi$, despite the fact that the true forcing $g$ is not known. 
It can therefore be considered a data assimilation method that is {\em robust} with respect to certain types of model error. 
This is achieved through a feedback term comprising the assimilation error $\phi - \psi$ as well as its {\em integral} over time; indeed, by the second line in Equation~\eqref{eq:nudging:td}, the variable $f$ might be regarded as another nudging term proportional to the integral over time of the assimilation error.
This strategy is well known in control theory and implemented for instance in proportional--integral~(PI) controllers~(see e.g.~\cite{anderson2007optimal}, Sec.9.3, or the book by~\cite{astrom1995pid_controllers}, in particular Ch.~3).
Stability and robustness of PI~(and other) controllers where analysed in a famous paper by \cite{Maxwell1867governors}, generally considered to be the first application of dynamical stability theory to the problem of process control.
\section{Notation and Mathematical Background}\label{sect:notation}
First we develop the precise functional framework of our results and introduce some notation.
Throughout, we will denote $\T^d=[0,2 \pi]^d$ where $d = 2$ or~$3$.
We will introduce the functional framework for \eqref{eq:td:intro} and \eqref{eq:nse:intro} separately, but will recycle the same notation; it will be clear from the context how to interpret the notation.
\subsection{Functional Setting of Transport-Diffusion Equation}\label{sect:notation:td}
We define the Hilbert spaces $H := \{\phi \in {L^2}({\bbT^d}): \int_{\bbT^d} \phi(x) \idd x = 0\}$, $V := H \cap {H^1}({\bbT^d})$, $\bbH := \{\bv\in{L^2}(\bbT^d)^d:\int_{\bbT^d}\bv(x)\idd x=0\}$, and $\bbV := {\{\bv \in {H^1}(\bbT^d)^d: \nabla\cdotp \bv = 0\}}$. We let $V^*,\bbV^*$ denote the dual space of $V$, $\bbV$, respectively.

We write $\abs{\cdot}$ for the norms on $H$ and on $\bbH$ and $\lb\cdot,\cdot\rb$ for their inner product.
Recall that every $\phi \in H$ has an expansion of the form
\begin{align}
\phi = \sum_{\mbk \in \Z^d, \mbk \neq 0} \phi_\mbk \exp( i \scp{\mbk}{\cdot}),\notag
\end{align}
which is convergent in ${L^2}$.
On $V$ we use the norm $\norm{\phi} := \left( (2\pi)^d\sum_{\mbk \in \Z^d, \mbk \neq 0} |\mbk|^2 |\phi_\mbk|^2 \right)^{1/2}$ which is equivalent to the usual (homogeneous) Sobolev norm on $H^1(\T^d)$. On $V^*$ we use $\norm{\phi}_*$$:=$$\left( (2\pi)^d\sum_{\mbk \in \Z^d, \mbk \neq 0} |\mbk|^{-2} |\phi_\mbk|^2 \right)^{1/2}$, which is equivalent to the Sobolev norm on $H^{-1}(\T^d)$.
We also write $\norm{\cdot}$, $\norm{\cdot}_*$ for the corresponding norms on $\bbV$, $\bbV^*$.
We write $\abs{\cdot}_p$ for the norm on $L^{p}({\bbT^d})$; on the Sobolev space ${H^\al}({\bbT^d})$ we use the semi-norm $\norm{\phi}_{\alpha} := \left((2\pi)^d \sum_{\mbk \in \Z^d, \mbk \neq 0} |\mbk|^{2\alpha} |\phi_\mbk|^2 \right)^{1/2}$, for any $\al\in\mathbb{R}$. 

With this notation, one sees that $|\cdot|=|\cdot|_2=\lVert\cdot\rVert_0$ on $H$, $\lVert\cdot\rVert=\lVert\cdot\rVert_1$ on $V$, and $\|\cdot\|_*=\|\cdot\|_{-1}$ on $V^{*}$.
We let $P_N$ be the projection in $H$ onto the first $N$ modes, that is
\begin{align}
    P_N\phi = \sum_{0 < \abs{\mbk} \leq N} \phi_{\mbk} \exp(i \scp{\mbk}{\cdot}),\notag
\end{align}
and we write $Q_N := I- P_N$.
We will rely on the following elementary existence and uniqueness theorem for \eqref{eq:td:intro}, which will suffice for our purposes (proof can be found in the appendix for convenience). It will be convenient to rewrite \eqref{eq:td:intro} in functional form. To this end, we let
    \begin{align}\label{def:L}
        {L(t):=\bv(t)\cdotp\nabla-\kap\De}
    \end{align}
    and rewrite \eqref{eq:td:intro} as 
       \begin{align}\label{def:td:weak}
      \frac{\dd\phi}{\dd t} + L(t)\phi
      = g.
    \end{align}
    We will consider so--called weak and strong solutions to Equation~\eqref{def:td:weak}, defined as follows:
\begin{Def}\label{def:td:sweak:strong}
      Let $d\geq2$ and $T>0$,
       $\phi_0\in H$, $g\in L^\infty(0,T; V^*)$, and $\bv\in L^{\frac{2p}{p-d}}(0,T; L^{p}(\T^d)^d)$, for some $p\in(d,\infty]$ satisfying $\nabla\cdotp\bv=0$ in the sense of distribution.
    A \emph{weak solution over $[0,T]$} of the initial value problem corresponding to \eqref{def:td:weak} is any $\phi\in C([0,T);H)\cap L^2(0,T;V)$ such that $\frac{\dd\phi}{\dd t}\in L^2(0,T;V^*)$, Equation~\eqref{def:td:weak} holds as an equation in $V^*$ for a.a.\ $t \in [0,T]$, and  $\phi(0)=\phi_0$ in $H$. 
    Furthermore
        \begin{align}\notag
            \frac{1}2|\phi(t)|^2+\nu\int_0^t\|\phi(s)\|^2ds= \frac{1}2|\phi_0|^2+\int_0^t(g(s),\phi(s))ds,
        \end{align}
    holds for all $t\in[0,T]$.
If $\phi_0\in V$, $g\in L^\infty(0,T;H)$, and $\bv\in L^{\frac{2p}{p-d}}(0,T;\bbV)$, a solution is a \emph{strong solution over $[0,T]$} if $\phi\in C([0,T);V)\cap L^2(0,T; H^2)$, $\frac{\dd\phi}{\dd t}(t)\in L^2(0,T;H)$, 
Equation~\eqref{def:td:weak} holds as an equation in $H$ for a.a.\ $t \in [0,T]$, and furthermore $\phi(0) = \phi_0$ in $V$.
\end{Def}
\begin{Thm}\label{thm:td:exist}
  Let $d\geq2$, and $T>0$, $\phi_0\in H$, $g\in L^\infty(0,T; V^*)$, and $\bv\in L^{\frac{2p}{p-d}}(0,T; L^{p}(\T^d)^d)$, for some $p\in(d,\infty]$ satisfying $\nabla\cdotp\bv=0$ in the sense of distribution.
    Then there exists a unique weak solution of \eqref{def:td:weak} over $[0,T]$, and furthermore $\phi(0)=\phi_0$.
    If, additionally, $\phi_0\in V$, $g\in L^\infty(0,T;H)$, and $\bv\in L^{\frac{2p}{p-d}}(0,T;\bbV)$, then the weak solution is also a strong solution.
\end{Thm}
In the study of the reconstruction of source terms for the transport diffusion equation, our analysis will assume stronger regularity assumptions on the advecting velocity than the conditions presented in \cref{thm:td:exist}. In particular, we will assume that $\bv\in L^\infty(0,\infty;L^p(\T^d)^d)$, for some $p>2$, or that $\bv\in L^\infty(0,\infty;W^{1,d}(\T^d)^d)$. 
\subsection{Functional Setting of the Navier-Stokes Equation}\label{sect:notation:nse}
In the case of the 2D NSE, we will set $d=2$, we will instead define ${\bbH={\PP}(H\times H)}$ and ${\bbV={\PP} (V\times V)}$, where $H, V$ are defined as in \cref{sect:notation:td}, and $\PP$ denotes the Leray-Helmholtz projection onto divergence-free vector fields. As in \cref{sect:notation:td}, for each $N>0$, we let $P_N:{L^2}({\bbT^2})^2\goesto {L^2}({\bbT^2})^2$ denote projection onto finitely many Fourier modes corresponding to wavenumbers $|k|\leq N$. We then let
\begin{align}\notag
        P_N={\PP} P_N{\PP} ,\quad Q_N={\PP}(I-P_N){\PP}={\PP}-P_N.
    \end{align}
We will make use of the same notation for $L^p$ over $\T^2$, as well as their norms.
Let
\begin{align}\label{def:Stokes}
        A=-\PP\De,
    \end{align}
denote the Stokes operator and its domain by $\bbD(A) = \bbV\cap H^2(\T^2)^2$. We recall that in the setting of periodic boundary conditions, $\PP$ commutes with $\De$, so that we simply have $A=-\De$. It can be shown that $\bbV = \bbD(A^{1/2})$ and $\| \bu \| = |A^{1/2}\bu|$.
Furthermore, the Sobolev norms $\no{\bu}{\al}$ defined in \cref{sect:notation:td} are equivalent to $\Abs{A^{\al/2}\bu}$, for all $\al\geq0$.
Let us recall the definitions of weak and strong solution for \eqref{eq:nse:intro} and the classical existence and uniqueness associated to them~\citep[see for instance][]{ConstantinFoiasBook,TemamNSEBook}.
\begin{Def}\label{def:nse:weak:strong}
  Given $T>0$, $\mbf\in L^2(0,T;\bbV^*)$, and $\bu_0\in\bbH$, a {\em Leray-Hopf weak solution over $[0,T]$} of the initial value problem corresponding to \eqref{eq:nse:intro} is any  $\bu\in C([0,T);\bbH)\cap L^2(0,T;\bbV)$ such that $\frac{\dd\bu}{\dd t}\in L^2(0,T;\bbV^*)$ and
    \begin{align}\label{def:nse:ivp}
        \frac{\dd\bu}{\dd t}+\nu A\bu+B(\bu,\bu)=\mbf,\quad \bu(0)=\bu_0,
    \end{align}
    where the equation holds in $\bbV^*$ for a.a.\ $t \in [0,T]$, and moreover,
  \begin{align}\label{est:energy:inequality}
    \frac{1}2|\bu(t)|^2
    +\nu\int_{0}^t\|\bu(s)\|^2ds
    &=\frac{1}2|\bu_0|^2+\int_{0}^t\lb \mbf(s),\bu(s)\rb ds,
    \end{align}
  holds for $t\in[0,T]$.
  If $\mbf\in L^2(0,T;\bbH)$ and $\bu_0\in\bbV$, a solution is a {\em strong solution over $[0,T]$} if $\bu\in C([0,T);\bbV)\cap L^2(0,T;\bbD(A))$ such that $\frac{\dd\bu}{\dd t}\in L^2(0,T;\bbH)$ and Equation~\eqref{def:nse:ivp} holds in $\bbH$ for a.a.~$t \in [0, T]$.
\end{Def}
%
  %
  Note, moreover, that if $\bu$ is a strong solution, then one has
  \begin{align}\label{est:enstrophy:equality}
    \frac{1}2\|\bu(t)\|^2
    +\nu\int_{0}^t|A\bu(s)|^2ds
    &=\frac{1}2\|\bu_0\|^2+\int_{0}^t\lb \mbg(s),\bu(s)\rb ds,
    \end{align}
  for all $t \in [0, T]$.
\begin{Thm}\label{thm:nse:exist:uniq}
  Suppose $d=2$. Given $T>0$, $\bu_0\in \bbH$, and $\mbf\in L^2(0,T;\bbV^*)$, there exists a unique Leray-Hopf weak solution of \eqref{eq:nse:intro} over $[0,T]$. If $\bu_0\in \bbV$ and $\mbg\in L^2(0,T;\bbH)$, then the Leray-Hopf weak solution $\bu$ is also a strong solution over $[0,T]$.
\end{Thm}
We will also make use of several local-in-time and global-in-time bounds whenever the force, $\mbg$, is either locally bounded or globally bounded in time~\citep[see][]{FoiasManleyRosaTemamBook2001}: let $G_*$ and $G$ denote generalized Grashof numbers defined by
\begin{align}\label{def:Grashof}
        G_*:=\frac{\sup_{t\geq0}\|\mbg(t)\|_*}{\nu^2}\,\quad 
        G:=\frac{\sup_{t\geq0}|\mbg(t)|}{\nu^2},
    \end{align}
After an application of the Cauchy-Schwarz inequality and Young's inequality in \eqref{est:energy:inequality} and \eqref{est:enstrophy:equality}, one has
\begin{align}\label{est:enstrophy:inequality:G}
    \begin{split}
  |\bu(t)|^2 +\nu\int_{t_0}^t\|\bu(s)\|^2ds
  &\leq|\bu(t_0)|^2+\nu^2G_*^2(t-t_0),
        \\
  \|\bu(t)\|^2+\nu\int_{t_0}^t|A\bu(s)|^2ds
  &\leq\|\bu(t_0)\|^2+\nu^2G^2(t-t_0),
    \end{split}
    \end{align}
for all $t\geq0$. Moreover, for any Leray-Hopf weak solution corresponding to $(\bu_0,\mbg)$, it holds that
\begin{align}
        \Abs{\bu(t)}^2&\leq e^{-\nu t}\Abs{\bu_0}^2 +2\nu^2 G_*^2(1-e^{-\nu t}),\notag
    \end{align}
for all $t\geq0$, while any strong solution satisfies
    \begin{align}
        \no{\bu(t)}{}^2&\leq e^{-\nu t}\no{\bu_0}{}^2+2\nu^2 G^2(1-e^{-\nu t}),\notag
    \end{align}
    for all $t\geq0$.
    Therefore, if we set $R := \sqrt{2} \nu G$ and $R_* := \sqrt{2} \nu G_*$, we find
    \beq{est:uniform:bounds}
       \begin{split}
          \Abs{\bu(t)}^2
          & \leq (\al^2-1) R_*^2 e^{-\nu t} + R_*^2,
          \quad \text{if}
          \quad \al \geq \frac{\Abs{\bu_0}}{R_*},\\
          \no{\bu(t)}{}^2
          &\leq (\al^2-1) R^2 e^{-\nu t} + R^2,
          \quad\text{if}
          \quad \al \geq \frac{\no{\bu_0}{}}{R},
        \end{split}
    \eeq
    %
    for all $t\geq0$.
    Hence, if $\al>1$, then $\Abs{\bu(t)}\leq \sqrt{2}R_*$, $\Abs{\bu(t)}\leq \sqrt{2}R$,  and $\no{\bu(t)}{}\leq \sqrt{2} R_*$, for all $t\geq T_1$, where $T_1$ depends only on $\al,\kap_0,\nu,G$.
    Otherwise, $\Abs{\bu(t)}\leq R_*$, $\Abs{\bu(t)}\leq R$, and $\no{\bu(t)}{}\leq R_*$, for all $t\geq0$.
    Bounds in $\bbD(A)$ and higher-order norms are also available and developed in the generality required for our purposes, for instance, in \cite{DascaliucFoiasJolly2005, BiswasBrownMartinez2022, Martinez2022force}: there exists $T_2>0$, depending only on $R_*$, such that
\begin{align}\label{def:absorbing:ball:H2}
        \Abs{A\bu(t)}^2\leq 2c_2^2\nu^2(\sig+G)^2G^2=c_2^2\left[\sig+\frac{\sqrt{2}}{2}\left(\frac{R}{\nu}\right)\right]^2R^2=:2R_2^2,
    \end{align}
for all $t\geq  T_2$, for some universal constant $c_2\geq1$, and shape factor $\sig:=\|\mbg\|_{L^\infty(0,\infty;\bbV)}/\|\mbg\|_{L^\infty(0,\infty;\bbH)}.$
Moreover, for any $\al>0$
\begin{align}
        \Abs{A\bu(t)}\leq (1+\al^2)^{1/2}R_2,\quad \text{if}\quad |A\bu_0|\leq \sqrt{2}\al R_2,\notag
    \end{align}
for all $t\geq0$. We emphasize that $R_2$ is a function of $R$ and $\sig$.
With a view on the analysis of the batch algorithm later on, it will be useful to summarize the above estimates.
Recall the shorthand notation \eqref{def:tau:sigma}; if $t_1\geq T_1$, then $t_0+t_1+\dots t_j\geq T_1$, and the estimates for $\no{\bu}{}$ can be written as
\begin{align}\notag
        \sup_{t\geq0}\|\bu(t;\si_j\bu_0,\si_j\mbg)\|^2\leq R^2\times \begin{cases}
        \al^2,&j=0,\ \al>1\\
        2,&j>0,\ \al>1\\
        1,&j\geq0,\ \al\leq1.
        \end{cases}
    \end{align}
\subsection{General Functional Inequalities and the Trilinear Form}\label{sect:trilinear_est} 
We recall H\"older's inequality:
    \begin{align}
        |\phi\psi|_1\leq |\phi|_p|\psi|_q,\quad \text{whenever}\quad \frac{1}p+\frac{1}q=1.\notag
    \end{align}
Given  $0 \leq \alpha < \beta$ and $\phi\in H$, we also have the
{\em generalised Poincar\'{e} inequality}:
    \begin{align}\label{est:Poincare}
        \norm{Q_N \phi}_{\alpha} & \leq \frac{1}{N^{\beta - \alpha}} \norm{Q_N \phi}_{\beta},   
    \end{align}
For $\phi\in H$, we also have {\em Bernstein inequalities}: there exists a constant $C>0$, depending only on the dimension $d>0$, for all $N\geq0$, and $\al\geq0$, such that
    \begin{align}
        \norm{P_N \phi}_{\beta}  \leq N^{\beta - \alpha} \norm{P_N \phi}_{\alpha},\qquad
        \abs{P_N \phi}_{\infty}  \leq CN^{\frac{d}{2}} \abs{P_N \phi},\label{est:Bernstein}
    \end{align}

Given $p\in(2,\infty)$, $s>0$, and $\phi\in H^s(\T^d)\cap H$,
we have the {\em Sobolev embedding inequality}:
\begin{align}\label{est:Sobolev}
    \abs{\phi}_{p}&\leq c_{p,s}\norm{\phi}_{s},\quad \text{whenever}\quad \frac{1}p=\frac{1}2-\frac{s}d.
\end{align}
Note that the case $p=2$, \eqref{est:Sobolev} trivially holds, so that the constant is simply given by
    \begin{align}
        c_{p,s}=c_{2,0}=1.\notag
    \end{align}

In the special case $d=2$, we have the {\em Ladyzhenskaya}, {\em Agmon}, and {\em Br\'ezis-Gallouet interpolation inequalities}:
\begin{align}
    \abs{\phi}_4^2&\leq c_L\|\phi\||\phi|,\qquad \Abs{\nabla\phi}_4^2\leq c_L'\norm{\phi}_2\|\phi\|,\quad \phi\in H^1(\T^2)\label{est:Ladyzhenskaya}
\end{align}
\begin{align}
    \abs{\phi}_\infty^2&\leq c_A\|\phi\|_2|\phi|,\quad \phi\in H^2(\T^2)\label{est:Agmon}
\end{align}
\begin{align}
    \abs{\phi}_\infty&\leq c_{BG}\|\phi\|\left[1+\log\left(\frac{|A\phi|^2}{|\phi|^2}\right)\right]^{1/2},\quad \phi\in H\cap H^2(\T^2)\label{est:BrezisGallouet}
\end{align}
where $c_L,c_L',c_A, c_{BG}$ are the constants of interpolation and are dimensionless.

A functional that is central to the analysis presented in this paper is the following trilinear form
    \begin{align}\label{def:trilinear}
    b(\bu, \phi, \psi) := \scp{\bu \cdot \nabla \phi}{\psi}
 = \sum_{i=1}^d \int_{\T^d}  \bu_i(x) \pdd_i \phi(x) \psi(x) \idd x
    \end{align}
for sufficiently smooth functions. %
For the 2D~NSE, we need to consider the following ``vectorial'' variant of the trilinear form, for which we will use the same symbol
    \begin{align}\label{def:trilinear:vec}
        b(\bu, \bv, \bw)
        = \sum_{i,j=1}^d \int_{\T^d}  \bu_i(x) \pdd_i \bv_j(x) \bw_j(x) \idd x
        = \sum_{j = 1}^d b(\bu, \bv_j, \bw_j),
    \end{align}
for $\bu, \bw \in \bbH$ and $\bv \in \bbV$, whenever the integral makes sense.

The following proposition establishes the functional settings over which the functionals \eqref{def:trilinear}, \eqref{def:trilinear:vec} are well-defined and possess useful boundedness properties. 
The inequalities asserted therein can be derived by repeated application of the general functional inequalities above and invoking the identities \eqref{eq:trilinear:skew}, \eqref{eq:trilinear:Aid} below, as needed. We omit the details.
\begin{Prop}[Properties and estimates of the trilinear form]\label{thm:trilinear}
%
Suppose $\nabla \cdot \bu = 0$.
Then
    \begin{align}\label{eq:trilinear:skew}
        b(\bu, \phi, \psi) = -b(\bu, \psi, \phi),
    \end{align}
and 
    \begin{align}\label{eq:trilinear:Aid}
        b(\bv,\bu,A\bu)+b(\bu,\bv,A\bu)+b(\bu,\bu,A\bv)=0.
    \end{align}
Furthermore, for all $d\geq2$, we have the following estimates:
\begin{enumerate}
\item For $p, q, r \in [1, \infty]$ so that $\frac{1}{p} + \frac{1}{q} + \frac{1}{r} = 1$ we have 
\begin{align}\label{est:trilinear:Holder}
|b(\bu, \phi, \psi)| \leq \abs{\bu}_{p} \abs{\nabla \phi}_{q} \abs{\psi}_r.
\end{align}

\item There exists $C>0$, for all $N\geq0$, $\bu,\bv,\bw\in\bbH$, such that
    \begin{align}\label{est:trilinear:Bernstein}
        |b(\bu, \bv, P_N \bw)|&\leq CN^{\frac{d}2}\abs{\bu} \abs{\bv}\norm{P_N \bw}.
    \end{align}

\item There exists $C>0$, for all $N>0$, $\bu,\bv,\bw\in\bbH$ such that
    \begin{align}\label{est:trilinear:BrezisWainger}
        |b(\bu,\bv,P_N\bw)|&\leq C(1+\log_+ N)^{1/2}|\bu||\bv|\|\nabla P_N\bw\|_{d/2}
    \end{align}

\item For $p \in (2, \infty]$, $\bu\in L^p(\T^d)^d$, $\nabla\phi\in H^{d/p}(\T^d)$, $\psi\in H$ it holds that
    \begin{align}\label{est:trilinear:Sobolev1}
        |b(\bu, \phi, \psi)| \leq c_{2p/(p-2),d/p}\abs{\bu}_{p} \norm{\nabla \phi}_{d/p} \abs{\psi}.
    \end{align}
\item For $\bu\in \bbH$ such that $\nabla\bu\in L^d(\T^d)^d$, $\phi\in H^2$, it holds that
    \begin{align}\label{est:trilinear:Sobolev2}
        |b(\bu, \phi, \Delta \phi)| \leq c_{2d/(d-1),1/2}^2\abs{\nabla \bu}_d \norm{\phi} \abs{\Delta \phi}. 
    \end{align}

\end{enumerate}

\noindent Lastly, in the special case $d=2$, we have the following inequality:

\begin{enumerate}\setcounter{enumi}{5}

\item For $\bu\in\bbV$, $\phi,\psi\in V$, it holds that
    \begin{align}\label{est:trilinear1}
        |b(\bu, \phi, \psi)| \leq c_L \abs{\bu}^{1/2} \norm{\bu}^{1/2}\norm{\phi}^{1/2}\norm{\psi}^{1/2}  \min\left\{\abs{\phi}^{1/2}  \norm{\psi}^{1/2},\norm{\phi}^{1/2}\abs{\psi}^{1/2}\right\}. 
    \end{align}

\end{enumerate}
\end{Prop}
Note that \eqref{est:trilinear:Sobolev2} is not a direct application of \eqref{est:Sobolev} but makes use of the identity \eqref{eq:trilinear:skew} before applying \eqref{est:Sobolev}. Morover, one immediately sees that \eqref{eq:trilinear:skew}, \eqref{eq:trilinear:Aid} imply
    \begin{align}\notag
        &b(\bu,\bv,\bv)=b(\bu, \bu, A\bu) = 0.
    \end{align}

\subsection{Quasi-Finite Forcing}\label{sect:quasi:finite} 
Our methodology for reconstructing the surface fluxes or forcings in transport-diffusion or 2D Navier-Stokes equations from low-mode observations will rely on the assumption that the reference forcing $g$, that is, the forcing that we seek to reconstruct, is of {\em quasi--finite rank}. We define this concept now. 
\begin{Def}\label{thm:1.10}
  A subset $\Gam \subset H^\beta$  is of {\em quasi--finite rank} $N \in \N$ if there exist $\alpha, \beta \in \R$ and a Lipschitz map $F:P_N(H^{\alpha}) \to Q_N(H^{\beta})$ so that for any $\phi \in \Gamma$ we have $Q_N \phi = F(P_N \phi)$.
  We will refer to the map $F$ as the {\em Lipschitz enslaving map} of $\Gamma$.
  Further, $N$ is the {\em rank} and the pair $(\alpha, \beta)$ is the {\em order} of $\Gamma$.
  The union of all sets of quasi-finite rank $N$, order $(\al,\be)$, and Lipschitz enslaving map $F$ is denoted by $\Gamma(N,F,\al,\be)$.
\end{Def}
Note that in this definition we assume that $P_N(H^{\alpha}) = P_N(H^{\beta})$ as these spaces are finite~dimensional. 
  We stress however that the Lipschitz property of the enslaving map $F$ is understood with respect to the norm $\norm{.}_{\alpha}$ on the domain and the norm $\norm{\cdot}_{\beta}$ on the range.
  The enslaving map and in particular the Lipschitz constant will therefore depend on the rank $N$ and the order $(\alpha, \beta)$ as is clarified in the following lemma:
\begin{Lem}\label{thm:1.20}
  Suppose $\Gamma\subset\Gamma(N,{F},\al,\be)$. Then for any $M > N$, there exist $\Gamma_M \supset \Gamma$ of quasi--finite rank $M$ and the same order with some Lipschitz enslaving map $F_M$ which has the same Lipschitz constant as ${F}$.
  Furthermore, if $\tilde{\alpha}, \tilde{\beta} \in \R$ with $\tilde{\beta} \leq \beta$, then $\Gamma$ is of order $(\tilde{\alpha}, \tilde{\beta})$ with the same Lipschitz enslaving map, now regarded as a map from $P_N(H^{\tilde{\alpha}})$ to $Q_N(H^{\tilde{\beta}})$, with Lipschitz constant $\|F\|_{L,\tilde{\al},\tilde{\be}}=N^{(\alpha - \tilde{\alpha})_+ - (\beta - \tilde{\beta})} \norm{F}_{L,\al,\be}$.
\end{Lem}
\begin{proof}
  For any $g \in \Gamma$ we have $g = P_N g + {F}(P_N g)$.
  If $M > N$, then this implies $Q_M g = Q_M {F}(P_N g) = Q_M {F}(P_N P_M g)$.
  On the other hand $g = P_M g + Q_M g$ and therefore $g = P_M g + F_M(P_M g)$ with $F_M := Q_M \circ {F} \circ P_N$.
  Clearly, $F_M$ maps $P_M(H^{\alpha})$ to $Q_M(H^{\beta})$, hence $g \in \Gamma_M:= \Gamma(M, F_M, \alpha, \beta)$.
It is now straightforward to check that ${F}$ and $F_M$ have the same Lipschitz constants.
The remaining statements are also straightforward and follows by repeated application of the inequalities \eqref{est:Poincare} and \eqref{est:Bernstein}.
\end{proof}
\begin{Rmk}\label{rmk:canonical}
  The Lemma~\ref{thm:1.20} shows that for any $\Gam\subset\Gam(N,F,\al,\be)$, there is a  family
    \begin{align}\notag
        \cG := \{\Gam_M:M\geq N\},
    \end{align}
    where each $\Gamma_M$ is of quasi-finite rank $M$ and order $(\alpha, \beta)$ with Lipschitz enslaving map $F_M$ so that
    \begin{enumerate}
    \item $\Gamma_N = \Gamma$, 
    \item $\cG$ is increasing, that is $\Gamma_{M_1} \subset \Gamma_{M_2}$ for $M_1 \leq M_2$, 
    \item\label{itm:4.10} $\norm{F_M}_{L, \alpha, \beta} = \norm{F}_{L, \alpha, \beta}$ for all $M \geq N$.
      \end{enumerate}
Later, we will encounter similar increasing families of quasi--finite rank sets which however do {\em not} satisfy item~\ref{itm:4.10}, that is the Lipschitz constant of $F_M$ may depend on $M$.
  \end{Rmk}
\subsubsection{Reconstruction of Functions from Spectral Data and its Connection to the Notion of Function Sets with Quasi--Finite Rank}
The purpose of this section is to provide further insight into the concept of quasi--finite rank function sets.
As the material will not be needed for the remainder of the paper, the reader may skip this material and proceed to Section~\ref{sect:algorithm1}.
Given $\al,\be\in\mathbb{R}$, for each $k\in\mathbb{Z}^d$, suppose that $F_k:H^\al\goesto \mathbb{C}^d$ is Lipschitz and $F_k(0)=0$. We define a mapping $F$ such that
    \begin{align}\label{eq:F:Fourier}
    (F(f))(x):=\sum_{k\in\mathbb{Z}^d}F_k(f)e^{ik\cdotp x}.
    \end{align}
By uniqueness of Fourier coefficients, we see that $F_k(f)=(F(f))_k$, for all $f\in H^\al$, whenever $F$ is well-defined, in which case, we may identify $F$ with $\{F_k\}_k$ and can write $F\sim\{F_k\}_{k}$ to denote the relation \eqref{eq:F:Fourier}. 
We now provide conditions that ensure \eqref{eq:F:Fourier} is well-defined and possesses additional regularity properties.

\begin{Lem}\label{lem:E:Lipschitz}
Let $\al,\be\in\mathbb{R}$ and $d\geq1$. 
Suppose that $F_k\in\Lip(H^\al,\mathbb{C}^d)$ and $F_k(0)=0$, for each $k\in\mathbb{Z}^d$, and that $\{F_k\}_k$ satisfies
    \begin{align}\label{cond:Fk:Lip}
        \sum_{k}\|F_k\|_{L,\al}^2|k|^{2\be}<\infty,
    \end{align}
where $\|F_k\|_{L,\al}$ denotes the Lipschitz norm of $F_k$. Then $F:H^\al\goesto H^\be$ is well-defined and $F\in\Lip(H^\al, H^\be)$. 
In particular
    \begin{align}\label{eq:F:Lip}
        \|F\|_{L,\al,\be}\leq\left(\sum_{k}\|F_k\|_{L,\al}^2|k|^{2\be}\right)^{1/2}.
    \end{align}
\end{Lem}
\begin{proof}
If \eqref{cond:Fk:Lip} holds, then we have
\begin{align}
  & \|F(f)-F(g)\|_\be^2 =\sum_k|F_k(f)-F_k(g)|^2|k|^{2\be} 
  \notag
    \\
    &\leq \|f-g\|_\al^2\sum_{k}\|F_k\|_{L,\al}|k|^{2\be}.\notag
\end{align}
Thus $F\in\Lip(H^\al,H^\be)$. In particular, since $F_k(0)=0$, we see that
    \begin{align}\notag
        \|F(f)\|_\be^2\leq \|f\|_\al^2\sum_k\|F_k\|_{L,\al}|k|^{2\be}<\infty,
    \end{align}
as desired.
\end{proof}
A converse to the estimate~\eqref{eq:F:Lip} cannot hold, however.
Indeed,
\beqn{equ:4.100}
\|F\|_{L,\al,\be}^{2}
\geq \frac{\|F(f)-F(g)\|_\be^2}{\|f-g\|_\al^2}
= \sum_k \frac{|F_k(f)-F_k(g)|^2}{\|f-g\|_\al^2} |k|^{2\be} 
\geq \frac{|F_k(f)-F_k(g)|^2}{\|f-g\|_\al^2} |k|^{2\be}, 
\eeq
for all $k$, and taking the supremum over $f, g$ we find
\beqn{equ:4.110}
\|F\|_{L,\al,\be}^{2}
\geq |F_k|^2_{L,\alpha} |k|^{2\be}, 
\eeq
and it is possible to construct examples where this inequality is (arbitrarily close to) an equality for arbitrarily many $k$, meaning that the right hand side of~\eqref{eq:F:Lip} can be arbitrarily large.
For the remainder of this section we motivate the use of quasi--finite rank function classes to represent the forcing. 
As we have seen, assuming the forcing to be of quasi--finite rank, rather than merely a truncated Fourier series, requires extra effort in the analysis.
In this section we will provide evidence that this extra effort is warranted.
Consider a square integrable function $f:\bbT \to \R$ on the (one--dimensional) torus with Fourier coefficients $\{f_k, k \in \Z\}$. 
It is well known that the convergence of the truncated Fourier series $f_N = \sum_{\abs{k} \leq N} f_k \exp( i \scp{k}{\cdot})$ to $f$ is very fast if the function $f$ is smooth (as a function on $\T$). 
In fact, if $f$ is analytic, the convergence is geometric even in the uniform topology.
In this situation, the truncated Fourier series provide more than adequate approximations of $f$.
This is no longer true though if the function $f$ has jump discontinuities (e.g.\ if it is not periodic as a function on $[0, 1]$).
The truncated series $f_N$ then exhibits the well--known Gibbs phenomenon along the discontinuities.
Moreover, the convergence slows down dramatically in the {\em entire} domain as the coefficients typically decay merely like $O(1/\abs{k})$.
This is true irrespective of the smoothness of $f$ between the discontinuities.
However, although the truncated Fourier series does not provide adequate approximations of $f$ in such situations, one might still hope that using the Fourier coefficients up to order $N$ in some other way might yield better approximations.
There is indeed a considerable body of literature on methods 
to reconstruct functions from finitely many Fourier coefficients (typically framed as attempts to avoid the Gibbs phenomenon and to restore uniform convergence everywhere except at the discontinuities).
More recently it was pointed out that these methods need to be {\em stable}, that is, exhibit a form of local Lipschitz continuity with respect to the provided Fourier data.
Several stable methods have been proposed, using what we call quasi-finite rank function classes in the present paper. 
The Lipschitz continuity of the high modes as functions of the low modes forms an essential ingredient to ensure the stability of those methods.
As a motivating example, we will discuss the core ideas of the method presented in~\citep{Hrycak_pseudospectral_fourier_reconstruction_modified_iprm_2009} but refer to the literature for more detailed discussions of this and other methods.
We also stress that to the best of our knowledge, stable methods for the reconstruction of functions from Fourier data have been fully investigated in dimension~1 only.
This means that some work is still required before such methods can be combined with the approaches presented in the present paper to reconstruct forcings in geophysical problems.
Carrying out this work is clearly beyond the scope of this paper, but we believe that by considering quasi--finite function classes as part of our methodology, it will be easy to accommodate future results regarding the stable reconstruction of functions from Fourier data.
The method presented in~\citep{Hrycak_pseudospectral_fourier_reconstruction_modified_iprm_2009} considers functions $f:[0, 1] \to \R$ which are analytic but not necessarily periodic.
Let $\{f_k, \abs{k} \leq N\}$ be the first $N$ Fourier coefficients of $f$ and write $f_N = \sum_{\abs{k} \leq N} f_k \exp(i \scp{k}{\cdot})$ as before.
The method proceeds by applying a second projection step in which $f_N$ is approximated by a linear combination of Legendre polynomials.
More specifically, let $P^{(l)}$ be the Legendre polynomial of order $l = 0, 1, 2, \ldots$.
An approximation $f_{N, L} = \sum_{l = 0}^{L-1} a_l P^{(l)}$ of $f_N$ is found by solving the equation
\beq{equ:5.10}
\sum_{l = 0}^{L-1} a_l P_k^{(l)} = f_k, 
\qquad \abs{k} \leq N
\eeq
for the parameters $a = (a_0, \ldots, a_{L-1})$, where $\{P^{(l)}_k, k \in \Z \} $ are the Fourier coefficients of the Legendre polynomial $P^{(l)}$ for each $l \in \N$.
In~\citep{Hrycak_pseudospectral_fourier_reconstruction_modified_iprm_2009} it is shown that Equation~\eqref{equ:5.10} has to be kept overdetermined and solved in a least~squares sense in order to keep the problem well conditioned. 
Namely, it is demonstrated that provided $N \geq L^2$, the condition number of this problem obeys a universal bound, while it may grow unchecked if the coefficients are determined from a less overdetermined problem.
Key to the analysis are tight bounds on the Fourier coefficients $\{P^{(l)}_k, k \in \Z \} $ of the Legendre polynomials $P^{(l)}$.
It is shown that $f_{N, L}$ converges uniformly to $f$ on $(0, 1)$ with exponential rate in $L$ (and thus root--exponential rate in $N$).
The method can easily be extended to functions that are piecewise analytic on $[0, 1]$ with a finite number of jump discontinuities.
It should be stressed however that in addition to the Fourier coefficients, the location of the discontinuities must be known to apply the method.
We will now use results from~\citep{Hrycak_pseudospectral_fourier_reconstruction_modified_iprm_2009} to demonstrate that the set of functions $\Gamma_N := \{\sum_{l = 0}^{L-1} a_l P^{(l)}: (a_0, \ldots, a_{L-1}) \in \R^{L}\}$ forms a quasi--finite rank function class if $N \geq L^2$, and we will determine the Lipschitz constant of the map $F_N$ which links the low with the high modes.
We write $p$ for the matrix with elements $p_{l, k} := P_k^{(l)}$ where $k \in \Z$ and $l = 0, \ldots, L-1$.
The parameters $a = (a_0, \ldots, a_{L-1})$ are linked to the Fourier coefficients of $g := \sum_{l = 0}^{L-1} a_l P^{(l)}$ through
\beq{equ:5.20}
(a p)_k = g_k, \qquad k \in \Z.
\eeq
In~\citep{Hrycak_pseudospectral_fourier_reconstruction_modified_iprm_2009}, it is shown that $p$ is injective, and it follows from elementary linear algebra that the solution of Equation~\eqref{equ:5.20} satisfies 
\beq{equ:5.30}
\norm{a}^2 \leq \si_{N}^{-2} \sum_{\abs{k} \leq N} \abs{g_k}^2,
\eeq
whenever $(g_k)_{k \in \Z} \in \text{Im}(p)$, where $\sigma_N$ is the smallest nonzero singular value of the row--truncated matrix $(p_{., k})_{\{\abs{k} \leq N\}}$.
In~\citep{Hrycak_pseudospectral_fourier_reconstruction_modified_iprm_2009}, the lower bound $\si^2_N \geq 1 - \frac{8}{\pi} \arcsin{\frac{1}{\pi}}$ is established, provided that $N \geq  L^2$.
On the other hand, from the fact that the Legendre polynomials form an orthonormal system in $L^2([0, 1])$ it follows that $p$ is an isometry so that in particular 
\beq{}
\sum_{\abs{k} > N} \abs{g_k}^2 \leq \norm{a}^2.\notag
\eeq
Combining this with Equation~\eqref{equ:5.30}, we find that $\Gamma_N$ is a quasi–finite rank function class, with the Lipschitz constant of the map $F_N$ (with respect to values in $Q_N(H_{-1})$) given by $\abs{F_N}_L = \frac{c}{N}$, where $c = \frac{1}{ \sqrt{1 - \frac{8}{\pi} \arcsin{\frac{1}{\pi}}}}$.
Summarising the discussion in this section, it seems to be of independent interest to develop the concept of quasi-finite rank sets of functions further.
  Thereby a unifying framework could be established which would cover the techniques presented above as well as other approaches from the literature.
  Such a framework however would allow to identify the essential mathematical properties of such function sets and potentially facilitate solving open problems such as to develop the example above in higher dimensions.
\section{Reconstruction of Forcings in Transport--Diffusion Equations}\label{sect:algorithm1}
\subsection{Sieve Algorithm}
In order to state our main convergence result concerning the Sieve Algorithm for \eqref{eq:td:intro}, let us first state the corresponding well-posedness result for the nudged system \eqref{eq:nudge:0}. Note that we will omit the proof of this result since \eqref{eq:nudge:0} is a lower-order perturbation of the linear equation \eqref{eq:td:intro}. In particular, its proof follows along the lines of that of \cref{thm:td:exist} in a straightforward way; the reader is thus referred to \cref{sect:appendix} for relevant details.

\begin{Thm}\label{thm:td:nudge:exist}
  Let $d\geq2$. Suppose $T>0$, $\psi_0\in H$, $f\in L^\infty(0,T; V^*)$ and $\bv\in L^{\frac{2p}{p-d}}(0,T; L^{p}(\T^d)^d)$, for some $p\in(d,\infty]$  satisfying $\nabla\cdotp\bv=0$ in the sense of distribution.
    For all $\mu, N>0$, there exists a unique solution $\psi\in C([0,T);H)\cap L^2(0,T;V)$ such that $\frac{\dd\psi}{\dd t}\in L^2(0,T;V^*)$ and    
    \begin{align}\label{def:td:nudge:weak}
      \frac{\dd\psi}{\dd t}(t) + \til{L}(t)\psi(t)
      = f(t),
    \end{align}
    where $\til{L}(t) := L(t)+\mu P_N,$ with $L(t)$ given as in Equation~\eqref{def:L}, holds as an equation in $V^*$ for a.a.~$t \in [0,T]$, and furthermore
      $\psi(0)=\psi_0$ in $H$.
    If, additionally, $\psi_0\in V$, $f\in L^\infty(0,T;H)$, and $\bv\in L^{\frac{2p}{p-d}}(0,T;\bbV)$, then $\psi\in C([0,T);V)\cap L^2(0,T; H^2(\bbT^d))$, $\frac{\dd\psi}{\dd t}(t)\in L^2(0,T;H)$, %
    Equation~\eqref{def:td:nudge:weak} holds as an equation in $H$ for a.a.~$t \in [0,T)$, and furthermore
    $\psi(0) = \psi_0$ in $V$.
\end{Thm}

Next, let us fix the following notions in connection with the vector field $\bv$ which we invoke when necessary.
\begin{Def}\label{def:V:W}
For $d\geq2$ and $t \geq 0$,we define
\begin{align}
  V_{\eps,d}(t)
  & := 
    c_{p_\eps,\eps}\abs{\bv{(t)}}_{{d/\eps}},\ \text{for}\ \eps\in[0,d/2),\quad V_{d/2,d}(t):=C\abs{\bv(t)},\ \text{for}\ \eps=d/2,\label{def:VtN}\\
  U_{d}(t)
  & := \min\left\{
  c_{{p_{1/2}},1/2}^2\abs{\nabla\bv{(t)}}_{{d}},
  \abs{\bv{(t)}}_{{\infty}}\right\}, \label{def:Wt}
\end{align}
where $p_\eps=2d/(d-2\eps)$ and ${c_{p,s}}$ refers to the constant in \eqref{est:Sobolev}. For convenience, we will define $c_{\infty,d/2}:=C$, where $C$ is the constant from \eqref{est:Bernstein}, so that $V_{d/2,d}(t)=c_{\infty,d/2}|\bv(t)|$; one will see in our analysis below that this choice of convention is consistent in our framework. Also, observe that $V_{0,d}(t)=|\bv(t)|_\infty$.
Then let
    \begin{align}
    V_{\epsilon, d} &:= \sup_{t \geq 0} V_{\eps,d}(t) \label{def:VepsdN}
    \\
        U_{d} &:= \sup_{t \geq 0} {U_{d}(t)}.\label{def:Vprime:d}
    \end{align}
We then define the following collections of velocity fields:
    \begin{align}
        \mathcal{V}_{\eps,d}&:=\{\bv\in L^\infty(0,\infty;\bbH):V_{\eps,d} <\infty\},\notag
        \\
        \mathcal{V}_{d}&:=\bigcup_{\eps\in[0,d/2)}\mathcal{V}_{\eps,d}\notag
        \\
        \mathcal{U}_{d}&:=\{\bv\in L^\infty(0,\infty;\bbH):U_{d}<\infty\}.\notag
    \end{align}
\end{Def}

Suppose $d=2,3$. Let $\bv\in\mathcal{V}_{d}$ and consider $\Gamma \subset V^*$ a set of quasi--finite rank $N_0$.
  Let
  \beq{equ:5.1000}
  \cG(\bv) := \{\Gamma_N: N \geq N_0\},
  \eeq
  be a family where each $\Gamma_N \subset V^*$ is of quasi-finite rank $N$ with Lipschitz enslaving map $F_N$ so that
    \begin{enumerate}
    \item $\Gamma_{N_0} = \Gamma$, 
    \item $\cG$ is increasing, that is $\Gamma_{M_1} \subset \Gamma_{M_2}$ for $M_1 \leq M_2$, 
    \item there exists $N\geq N_0$ such that 
      \begin{align}\label{equ:5.1010}
            \left(\frac{1+\no{F_N}{L,*}}{N}\right)\min_{\eps\in[0,d/2]}\left\{\frac{V_{\eps,d}}{\kap}N^{\eps}\right\}<\frac{\sqrt{2}}{2}.
      \end{align}
    \end{enumerate}
Likewise, if $\bv \in \mathcal{V}_{d} \cap \mathcal{U}_{d}$ and $\Gamma \subset H$ a set of quasi--finite rank $N_0$, we define $\tilde{\cG}(\bv)$ in the same way, except that $\Gamma_N \subset H$ for all $N \geq N_0$, and the condition~\eqref{equ:5.1010} is replaced with 
    \begin{align}\label{equ:5.1020:strong}
        \max \left\{\left(\frac{1+\no{F_N}{L,0}}{N} \right)\min_{\eps\in[0,d/2]}\left\{\frac{V_{\eps,d}}{\kap}N^{\eps}\right\},  \frac{1}{N}\left(\frac{U_{d}}{\kap}\right) \right\}< \frac{1}2,
      \end{align}
      for sufficiently large $N$.
The construction presented in \cref{rmk:canonical} will result in a suitable $\cG$ (or $\tilde{\cG}$) but other choices are possible in which the Lipschitz constants in~\eqref{equ:5.1010} (or \eqref{equ:5.1020:strong}) depend on $N$.
\begin{Thm}\label{thm:converge:sieve:td:detailed}
Suppose that $\bv\in\mathcal{V}_{d}$ and let $\cG(\bv) := \{\Gamma_N\}_{N \geq N_0}$ be a quasi-finite family of functions as defined above. Assume that the true forcing, $g$, in Equation~\eqref{eq:td:intro} is a member of $\Gamma_{N_0}$.
  Choose $N$ and $\mu$ in the Sieve Algorithm (in particular \eqref{eq:ttheta:stage:j}) so that
  \begin{align}\label{cond:mu:kap:N:convergence}
      (1+\no{F_N}{L,*})^2\min_{\eps\in[0,d/2]}\left\{\frac{V_{\eps,d}}{\kap} N^{\eps}\right\}^2\kap< \mu \leq \frac{N^2\kap}{2}.
    \end{align}
  Then for any $f^{(0)}\in L^\infty(0,\infty;\Gam_N)$ we have 
  \begin{align}
      \lim_{j\goesto\infty}\sup_{t\geq0}\Abs{\tau_{j+1}{\psi^{(j)}}(t)-\si_{j+1}\phi(t)}
      =\lim_{j\goesto\infty}\sup_{t \geq 0}\no{\tau_{j+1}f^{(j)}(t)-\si_{j+1}g(t)}{*}=0,\notag
    \end{align}
  where $t_j=t_*$, for all $j\geq0$, and $t_*$ is any positive number satisfying
  \begin{align}
     e^{-\mu t_{*}}
    + \left(\frac{\kap}{\mu}\right)\left(1+\no{F_N}{L,*}\right)^2\min_{\eps\in[0,d/2]}\left\{\frac{V_{\eps,d}}{\kap}N^{\eps}\right\}^2
    < 1.\notag
    \end{align}
\end{Thm}

Under stronger assumptions on the velocity field and enslaving map $F$, we obtain convergence in a stronger topology.

\begin{Thm}\label{thm:converge:sieve:td:detailed:strong}
Suppose that $\bv\in\mathcal{V}_{d}\cap\mathcal{U}_{d}$ and 
let $\tilde{\cG}(\bv) := \{\Gamma_N\}_{N \geq N_0}$ be a quasi-finite family of functions as defined above. Assume that the true forcing, $g$, in Equation~\eqref{eq:td:intro} is a member of $\Gamma_{N_0}$.
  Choose $N$ and $\mu$ in the Sieve Algorithm so that
    \begin{align}\label{cond:mu:kap:N:convergence:strong}
            \max\left\{ (1+\no{F_N}{L,0})\min_{\eps\in[0,d/2]}\left\{\frac{V_{\eps,d}}{\kap} N^{\eps}\right\}, \frac{U_{d}}{\kap}\right\}^2\kap< \mu\leq \frac{N^2\kap}{4},
    \end{align}
    then for any $f^{(0)}\in L^\infty(0,\infty;\Gam_N)$,
    \begin{align}\notag
      \lim_{j\goesto\infty}\sup_{t\geq0}\no{\tau_{j+1}{\psi^{(j)}}(t)-\si_{j+1}\phi(t)}
      =\lim_{j\goesto\infty}\sup_{t \geq 0}\Abs{\tau_{j+1}f^{(j)}(t)-\si_{j+1}g(t)}
      =0,
    \end{align}
    where $t_*$ given by
    \begin{align}
      e^{-\mu t_{*}}
    + \left(\frac{\kap}{\mu}\right)\left(1+\no{F_N}{L,0}\right)^2\min_{\eps\in[0,d/2]}\left\{\frac{V_{\eps,d}}{\kap}N^{\eps}\right\}^2
    < 1.\notag
    \end{align}
\end{Thm}
%



\begin{Rmk}\label{rmk:Peclet}
    \begin{enumerate}
    \item In the context of \eqref{def:td:weak}, an important non-dimensional quantity is the P\'eclet number, $\Pe$. This non-dimensional quantity is defined by the ratio of a characteristic velocity scale and a characteristic diffusive scale, and thus distinguishes between the relative dominance of advective effects to diffusive ones. In our framework, we may define generalized P\'eclet numbers by:              
    \begin{align}\label{def:Peclet:sieve}
       \Pe_{\eps,d}:=\frac{V_{\eps,d}}{\kap},\quad \Pe_d:=\frac{U_d}{\kap}.
    \end{align}
Thus our conditions \eqref{cond:mu:kap:N:convergence}, \eqref{cond:mu:kap:N:convergence:strong} can be recast in terms of the P\'eclet numbers \eqref{def:Peclet:sieve}.

    \item  Note that  due to the properties of $\cG$ (or $\tilde{\cG}$, respectively), we can always find $N, \mu$ so that the conditions~\eqref{cond:mu:kap:N:convergence} or~\eqref{cond:mu:kap:N:convergence:strong} are satisfied. In the particular case when only $\bv\in L^\infty(0,\infty;H)$ is known, i.e., the minimum is achieved at the endpoint $\eps=d/2$, then \eqref{cond:mu:kap:N:convergence} imposes a balance between $\bv$, $F_N$, and $N$ such that $\sup_{t\geq0}|\bv(t)|/\kap$ be sufficiently small relative to $(1+\|F_N\|_{L,*})N^{d/2-1}$ in order for the condition to be non-trivial.

    \item Of course, a larger $N$ means more observations are required in order to deploy the algorithm. When $d=2$, then in the endpoint case when the minimum is achieved at $\eps=d/2=1$, the conditions \eqref{equ:5.1010}, \eqref{equ:5.1020:strong} are consistent with $\bv$ being fixed independently of the observational resolution, $N$, as long as $\limsup_N\|F_N\|_{L,*}<\infty$. However, when $d=3$, it is impossible to satisfy \eqref{equ:5.1010} or \eqref{equ:5.1020:strong} unless $N_0$ is sufficiently small, relative to $\sup_{t\geq0}|\bv(t)|/\kap$, when the minimum is achieved at $\eps=d/2=3/2$. We believe this to be a technical assumption of the present analysis. This is to be contrasted with the conditions imposed on the enslaving map in the Nudging Algorithm, where this technical constraint is not encountered.  Nevertheless, whether this condition can be improved or not in the Sieve Algorithm deserves further attention. 
 \end{enumerate}
\end{Rmk}
\subsection{Proof of the Main Theorems for Sieve Algorithm} 
We will now prove  \cref{thm:converge:sieve:td:detailed} and \cref{thm:converge:sieve:td:detailed:strong}. First, recall \eqref{def:L}, so that \eqref{eq:td:intro} may be rewritten as \eqref{def:td:weak}. {Note that the corresponding translated operators, $\tau_jL$ and $\si_jL$ are defined accordingly by $(\tau_jL)(t)\phi=L(t+t_j)\phi$ and $(\si_jL)(t)=(\tau_{\tau_1+\dots+\tau_j}L)(t)$, where $\tau_j, \si_j$ are defined in \eqref{def:tau:sigma}. We will often omit the dependence of $L$ on $t$ below and simply write $L$ or $\si_jL$ for convenience, although it is to be understood that the $\bv$ is generally time-dependent throughout.}

Within the functional setting of the transport--diffusion equations in Section~\ref{sect:notation:td} and~\cref{thm:td:exist}, the Sieve Algorithm corresponding to \eqref{def:td:weak} is well-defined. Moreover, all of the analysis performed below is rigorously justified within the context of \cref{thm:td:exist}. The proof of \cref{thm:converge:sieve:td:detailed} will proceed in three steps, which will constitute the subsequent three sections. First we will obtain suitable representations for the model and synchronization errors in \cref{sect:model:error:td}. Secondly, we will obtain quantitative estimates for the model and synchronization error in \cref{sect:sync:error:td}. The final step  of \cref{thm:converge:sieve:td:detailed}, which will combine the analysis of \cref{sect:model:error:td} and \cref{sect:sync:error:td} will be carried out \cref{sect:proof:sieve:td}.

\subsubsection{Error Representation and Analysis of Model Error}\label{sect:model:error:td}
Firstly, let us recall the construction of $\psi^{(j)}, f^{(j)}$ from \eqref{eq:nudge:0}--\eqref{def:force:complete:stage:j} in \cref{sect:def:sieve:td}. We then denote the stage-$(j+1)$ model error by 
    \begin{align}
        {h}^{(j+1)}:&=f^{(j+1)}-\si_{j}g,\notag
    \end{align}
{where $f^{(j+1)}$ is defined by \eqref{def:force:complete:stage:j}}. We {then denote the model error on large-scales, i.e., through frequencies $|k|\leq N$, at this stage by}
    \begin{align}\label{def:model:error:td:low}
        e^{(j+1)}:=P_N {h}^{(j+1)}
         = l^{(j+1)}-\si_{j}P_Ng,
    \end{align}
and invoke \eqref{def:force:complete:stage:1}, \eqref{def:force:complete:stage:j} to equivalently represent ${h}^{(j+1)}$ as
    \begin{align}\label{eq:model:error:equiv}
    {h}^{(j+1)}
    &=e^{(j+1)}+Q_N\left({F}(l^{(j+1)})-{F}(\si_{j}P_Ng)\right)\\
    & = e^{(j+1)} +Q_N\left({F}( e^{(j+1)} +\si_jL )-{F}(\si_jL)\right),\notag
    \end{align}
    having used~\eqref{def:model:error:td:low}.
    As ${F}$ is Lipschitz, the main step is to estimate the low modes $e^{(j+1)}$ of the model error, which we will do in terms of the synchronization error.
Indeed, we introduce
\begin{align}\label{def:ze}
        {z}^{(j)}:={\psi^{(j)}}-\si_j\phi.
    \end{align}
Observe that by \eqref{eq:theta:stage:1}, \eqref{eq:theta:stage:j}, we have
    \begin{align}\label{eq:qj:high}
        \phi^{(j)}-\si_j\phi=Q_N\psi^{(j)}-Q_N\si_j\phi=Q_N{z}^{(j)},
    \end{align}
for all $j\geq0$. Then using \eqref{def:td:weak}, \eqref{def:force:large:stage:1}, \eqref{def:force:large:stage:j}, and \eqref{def:L} we see that
    \begin{align}
        e^{(j+1)}=l^{(j+1)}-\si_jP_Ng=P_N {(\si_jL)}{\phi^{(j)}}-P_N(\si_jL)\phi=P_N{(\si_jL)}Q_N{z}^{(j)}.\notag
    \end{align}
As $P_N$ commutes with $-\De$, we see that
\beq{eq:model:error:commuted}
        e^{(j+1)}=P_N({(\si_j\bv)}\cdotp\nabla Q_N{z}^{(j)}).
\eeq

We will prove the following.

\begin{Lem}\label{lem:sieve:model:error}
Suppose $\bv\in \mathcal{V}_{d}$. For each $j\geq0$, we have
    \begin{align}
         \no{e^{(j+1)}}{*}
        &\leq\min_{\eps\in[0,d/2]}\left\{(\si_jV_{\eps,d})(t)N^{\eps}\right\}\Abs{{z}^{(j)}}\label{est:ejplus1:Hminus1a}
        \\
        \Abs{e^{(j+1)}}&\leq \min_{\eps\in[0,d/2]}\left\{(\si_jV_{\eps,d})(t) N^{\eps}\right\}\no{{z}^{(j)}}.\label{est:ejp1:Hminus1}
    \end{align}
\end{Lem}
\begin{proof}
We multiply Equation~\eqref{eq:model:error:commuted} with some $\ph \in V$ and integrate. 
If $\eps\in(0,d/2)$, the estimate~\eqref{est:trilinear:Sobolev1} (with $p = d/\epsilon$) and \eqref{est:Poincare} then gives 
\beq{est:ejplus1:Hminus1:duality}\begin{split}
\lb e^{(j+1)}, \varphi \rb
& = b(\si_j\bv, Q_N{z}^{(j)}, P_N \varphi)=-b(\si_j\bv, P_N\varphi,Q_N{z}^{(j)})\\
& \leq c_{p_{\eps},\eps}  \Abs{Q_N {z}^{(j)}} \Abs{\si_j\bv}_{d/\eps} \norm{\nabla P_N \varphi}_{\eps}\\
& \leq c_{p_{\eps},\eps}  N^{\eps} \Abs{Q_N {z}^{(j)}} \Abs{\si_j\bv}_{d/\eps} \norm{\varphi},
    \end{split}\eeq
where ${c_{p_\eps,\eps}}$ is the constant from \cref{def:V:W}, and the estimate~\eqref{est:Bernstein} was used in the last line.
If $\eps=d/2$, then by invoking \eqref{est:trilinear:Bernstein}, we obtain
    \begin{align}
        |\lb e^{(j+1)},{\varphi}\rb|&\leq  CN^{\frac{d}2}|Q_N{z}^{(j)}||\si_j\bv|\|P_N\varphi\|\notag
        \\
        &\leq  CN^{\frac{d}2}|Q_N{z}^{(j)}||\si_j\bv|\|\varphi\|.\label{est:ejplus1:Hminus1:duality:endpoint} 
    \end{align}
Lastly, if $\eps=0$, by using estimate~\eqref{est:trilinear:Holder} with $(p,q,r) = (\infty,2,2)$, we also have
    \begin{align}\label{est:ejplus1:Hminus1:duality:infty} 
        |\lb e^{(j+1)},{\varphi}\rb|
        \leq \Abs{\si_j\bv}_{\infty}\no{P_N\ph}{}
        \Abs{Q_N{z}^{(j)}}
        \leq \Abs{\si_j\bv}_{{\infty}} \no{ {\varphi}}{}
        \Abs{{z}^{(j)}}.
    \end{align}
Since \eqref{est:ejplus1:Hminus1:duality}, \eqref{est:ejplus1:Hminus1:duality:endpoint}, and \eqref{est:ejplus1:Hminus1:duality:infty} each hold for all $\varphi\in V$, we deduce
\begin{align}\label{est:ejplus1:Hminus1b}
        \no{e^{(j+1)}}{*}
        \leq \min\left\{c_{{p_\eps},\eps} N^{\eps}\Abs{\si_j\bv}_{{d/\eps}},CN^{\frac{d}2}|\si_j\bv|,|\si_j\bv|_\infty\right\} \Abs{{z}^{(j)}}.
    \end{align}
Upon combining \eqref{est:ejplus1:Hminus1:duality} and \eqref{est:ejplus1:Hminus1b}, we obtain
    \begin{align}
        \no{e^{(j+1)}}{*}\leq \min_{\eps\in[0,d/2]}\left\{N^{\eps}(\si_jV_{\eps,d})(t)\right\}\Abs{{z}^{(j)}},\notag
    \end{align}
where $V_{\eps,d}(t)$ is defined in \eqref{def:VtN}, which is precisely \eqref{est:ejplus1:Hminus1a}.

Next, we estimate $\Abs{e^{(j+1)}}$ in terms of $\no{{z}^{(j+1)}}{}$. To do so, observe that if $\eps\in(0,d/2)$, then we use \eqref{est:ejplus1:Hminus1:duality} with $\varphi = e^{(j+1)}$ and find that
    \begin{align}
        \Abs{e^{(j+1)}}^2
        & \leq c_{{p_\eps},\eps} \Abs{Q_N {z}^{(j)}} \Abs{\si_j\bv}_{{d/\eps}} \no{\nabla P_N e^{(j+1)}}{{\eps}} \notag\\
        & \leq c_{{p_\eps},\eps}  N^{\eps} \no{{z}^{(j)}}{} \Abs{\si_j\bv}_{{d/\eps}} \Abs{e^{(j+1)}}. \notag
    \end{align}
Division by $\Abs{e^{(j+1)}}$ then yields
\begin{align}\label{est:ejplus1:L2a}
        \Abs{e^{(j+1)}} \leq c_{{p_\eps},\eps} N^{\eps}\Abs{\si_j\bv}_{{d/\eps}}\no{{z}^{(j)}}{}.
\end{align}
Similarly, if $\eps=d/2$, then upon using \eqref{est:Bernstein}, we have
    \begin{align}
        |e^{(j+1)}|^2&\leq  CN^{\frac{d}2+1}|Q_N{z}^{(j)}||\si_j\bv||P_Ne^{(j+1)}|\notag
        \\
        &\leq  CN^{\frac{d}2}\|Q_N{z}^{(j)}\||\si_j\bv||e^{(j+1)}|\notag,
    \end{align}
so that
    \begin{align}
        |e^{(j+1)}|&\leq CN^{\frac{d}2}\|Q_N{z}^{(j)}\||\si_j\bv|.\label{est:ejplus1:L2b}
    \end{align}
Lastly, if $\eps=0$, we estimate \eqref{eq:model:error:commuted} with a direct application of H\"older's inequality yields
      \begin{align}
        \Abs{e^{(j+1)}}
        \leq \Abs{\si_j\bv}_{{\infty}} \no{{z}^{(j)}}{},\notag
    \end{align}
Upon combining these \eqref{est:ejplus1:L2a}, \eqref{est:ejplus1:L2b}, we therefore deduce
    \begin{align}\label{est:ejplus1:L2}
        \Abs{e^{(j+1)}}\leq \min_{\eps\in[0,d/2]}\left\{N^{\eps}(\si_jV_{\eps,d})(t)\right\}\no{{z}^{(j)}}{},
    \end{align}
    which proves \eqref{est:ejp1:Hminus1}.
\end{proof}
Since $\Gamma$ is quasi-finite of some rank and order, we may deduce the following.
\begin{Lem}\label{cor:sieve:state:error}
Suppose $\bv\in\mathcal{V}$ and that $\Gamma\subset V^*$ is of quasi-finite rank $N$, for some $N>0$, with Lipschitz enslaving map $F_N$. If $\Gamma$ has order $(-1,-1)$, then  \begin{align}\label{est:model:error:Hminus1:td}
        \|{h}^{(j+1)}\|_{*}& \leq  (1 + \norm{F_N}_{L,*}) \min_{\eps\in[0,d/2]}\left\{(\si_jV_{\eps,d})(t) N^{\eps}\right\}\Abs{{z}^{(j)}},
    \end{align}
where $\no{F}{L,*}$ is the Lipschitz constant of $F_N:P_NV^*\goesto Q_NV^*$.  If $\Gamma$ has order $(0,0)$, then
    \begin{align}\label{est:model:error:L2:td}
         \Abs{{h}^{(j+1)}}
        &\leq (1+\no{F_N}{L,0})\min_{\eps\in[0,d/2]}\left\{(\si_jV_{\eps,d})(t) N^{\eps}\right\}\no{{z}^{(j)}}{},
    \end{align}
where $\no{F_N}{L,0}$ is the Lipschitz constant of $F:P_NH\goesto Q_NH$.
\end{Lem}
\begin{proof}
If ${F_N}:P_NV^*\goesto Q_NV^*$, we estimate the synchronization error in $V^*$ by applying the triangle inequality to \eqref{eq:model:error:equiv} and then use \eqref{est:ejplus1:Hminus1a} of \cref{lem:sieve:model:error} to obtain
\begin{align}\notag
\begin{split}
\norm{{h}^{(j+1)}}_{*}
& \leq\norm{e^{(j+1)}}_{*} + \norm{Q_N(F_N(l^{(j+1)})-F_N(\si_jL))}_{*}\\
& \leq (1 + \norm{F_N}_{L,*}) \norm{e^{(j+1)}}_{*}\\
& \leq (1 + \norm{F_N}_{L,*})\min_{\eps\in[0,d/2]}\left\{(\si_jV_{\eps,d})(t) N^{\eps}\right\} \Abs{{z}^{(j)}} .
\end{split}
\end{align}
Alternatively, if ${F}:P_NH\goesto Q_NH$, then by instead applying \eqref{est:ejplus1:L2} of \cref{lem:sieve:model:error}, we obtain
\begin{align}
        \Abs{{h}^{(j+1)}}
        &\leq\Abs{e^{(j+1)}}+\Abs{Q_N(F(l^{(j+1)})-F_N(\si_jL))}\notag\\
        &\leq (1+\no{F_N}{L,0})\Abs{e^{(j+1)}}\notag\\
        &\leq (1+\no{F_N}{L,0})\min_{\eps\in[0,d/2]}\left\{(\si_jV_{\eps,d})(t) N^{\eps}\right\}\no{{z}^{(j)}}{}.\notag
\end{align}
Therefore, to complete the convergence analysis for the model error in $V^*$ and $H$, it remains to study the synchronization error $\zeta^{(j)}$ in both ${H}$ and ${V}$, respectively. 
\end{proof}

\subsubsection{Analysis of the Synchronization Error}\label{sect:sync:error:td}
Observe from \eqref{eq:ttheta:stage:j} and \eqref{def:ze} that ${z}^{(j)}$ obeys the following evolution equation:
    \begin{align}\label{eq:sync:error:td}
        \bdy_t{z}^{(j)}+{(\si_jL)}{z}^{(j)}=\tau_j {h}^{(j)}-\mu P_N{z}^{(j)},
    \end{align}
with initial value 
    \begin{align}\label{eq:sync:initial:td}
        {z}^{(j)}(0)= \tau_j \psi^{(j-1)} (0) - \sigma_j \phi(0)=\psi^{(j-1)}(t_j)-\si^{(j-1)}\phi(t_j)=\tau_j{z}^{(j-1)}(0).
    \end{align}
We derive the following differential inequalities for the synchronization error.

\begin{Lem}\label{lem:sieve:sync:error}
For any $d\geq2$, if $\mu, N$ satisfies
\begin{align}\label{cond:mu:kap:N:L2:td}
\mu \leq \frac{N^2 \kap}{2},
\end{align}
then
\begin{align}\label{equ:3.100}
\frac{d}{dt}
\Abs{{z}^{(j)}}^2 + \mu\Abs{{z}^{(j)}}^2
\leq \frac{\no{\tau_j{h}^{(j)}}{*}^2}{\kap},
\end{align}
holds for all $t\geq 0$. For $d\in\{2,3\}$, let
    \begin{align}
        N_*:=\sqrt{2}\frac{U_{d}}{\kap},\notag
    \end{align}
where $U_{d}$ is defined by \eqref{def:Vprime:d}. Then
    \begin{align}\label{equ:3.110}
        \frac{d}{dt}\no{{z}^{(j)}}{}^2+\mu\no{{z}^{(j)}}{}^2\leq \frac{\Abs{\tau_j{h}^{(j)}}^2}{\kappa},
    \end{align}
for all $t\geq0$, for any $N\geq N_*$ and $\mu$ satisfying
    \begin{align}\label{cond:mu:kap:N:H1:td}
           \frac{N_*^2\kap}4\leq \mu\leq \frac{N^2\kap}{4},
    \end{align}
\end{Lem}

\begin{proof} To estimate $\Abs{{z}^{(j)}}$, we take the ${L^2}$ inner product of \eqref{eq:sync:error:td} with ${z}^{(j)}$ and use the fact $\bv$ is divergence-free to obtain the balance
    \begin{align}\label{eq:sync:L2:td}
        \frac{1}2\frac{d}{dt}\Abs{{z}^{(j)}}^2+\kap\no{{z}^{(j)}}{}^2=\lb \tau_j{h}^{(j)},{z}^{(j)}\rb-\mu\lb P_N{z}^{(j)},{z}^{(j)}\rb.
    \end{align}
By the Cauchy-Schwarz inequality, it follows that
    \begin{align}
        |\lb \tau_j{h}^{(j)},{z}^{(j)}\rb|&\leq \no{\tau_j{h}^{(j)}}{*}\no{{z}^{(j)}}{}\leq \frac{\no{\tau_j{h}^{(j)}}{*}^2}{2\kap}+\frac{\kap}2\no{{z}^{(j)}}{}^2.\label{est:sync:force:Hminus1:td}
    \end{align}
Furthermore
\begin{align}\label{eq:feedback:energy:td}
        -\mu\lb P_N{z}^{(j)},{z}^{(j)}\rb
        =\mu\lb Q_N{z}^{(j)},{z}^{(j)}\rb-\mu\Abs{{z}^{(j)}}^2
        \leq - \mu \Abs{{z}^{(j)}}^2+ \frac{\mu}{N^2}\no{{z}^{(j)}}{}^2.
    \end{align}
Since $\mu, N>0$ satisfy \eqref{cond:mu:kap:N:L2:td},
we can use~\eqref{est:sync:force:Hminus1:td} and~\eqref{eq:feedback:energy:td} in~\eqref{eq:sync:L2:td}, giving \eqref{equ:3.100}.

Next, we estimate $\no{{z}^{(j)}}{}$. To do so, we take the ${L^2}$ inner product of \eqref{eq:sync:error:td} with $-\De{z}^{(j)}$ to obtain the balance
    \beq{eq:sync:H1:td}
        \frac{1}{2} \frac{d}{dt}\no{{z}^{(j)}}{}^2+\kap\Abs{\De{z}^{(j)}}^2=-\lb \tau_j{h}^{(j)}, \De{z}^{(j)}\rb
        {+ b(\si_j\bv, {z}^{(j)}, \De{z}^{(j)})}
        +\mu\lb P_N{z}^{(j)},\De{z}^{(j)}\rb.
    \eeq
By H\"older's inequality, it follows that
    \beq{est:sync:force:td}
       |\lb\tau_j{h}^{(j)}, \De{z}^{(j)}\rb|
       \leq \Abs{\tau_j{h}^{(j)}}\Abs{\De{z}^{(j)}}
       \leq  \frac{\Abs{\tau_j{h}^{(j)}}^2}{2\kap}+\frac{\kap}2\Abs{\De{z}^{(j)}}^2.
       \eeq
For the second term on the right--hand--side of \eqref{eq:sync:H1:td}, we use the estimate~\eqref{est:trilinear:Holder} on the trilinear form $b$ with $(p,q,r)=(\infty,2,2)$ and find
\beq{est:sync:trilinear:1}\begin{split}
               |b(\si_j\bv, {z}^{(j)}, \De{z}^{(j)})|
               &\leq \Abs{\si_j\bv}_{\infty}\no{{z}^{(j)}}{}\Abs{\De{z}^{(j)}}\\
        & \leq \frac{\Abs{\si_j\bv}_{\infty}^2}{\kap}\no{{z}^{(j)}}{}^2+\frac{\kap}4\Abs{\De{z}^{(j)}}^2.
\end{split}    
\eeq
Note that this estimate is dimension-independent.

Now, assuming $d=2,3$, a similar estimate but with a different dependence on $\bv$ is obtained if we apply \eqref{eq:trilinear:skew} to write
    \begin{align}\notag
        b(\si_j\bv, {z}^{(j)}, \De{z}^{(j)})=-\sum_{i=1,2}b(\si_j\bdy_i\bv, {z}^{(j)}, \bdy_i{z}^{(j)}).
    \end{align}
We then apply estimate~\eqref{est:trilinear:Sobolev2}, followed by Young's inequality
    \begin{align}
    |b(\si_j\bv, \nabla{z}^{(j)}, \De{z}^{(j)})|
    &\leq c_{2d/(d-1),1/2}^2\Abs{\si_j\nabla\bv}_{d}\Abs{\De{z}^{(j)}}\no{{z}^{(j)}}{}\notag\\
        &\leq c_{2d/(d-1), 1/2}^4\frac{\Abs{\si_j\nabla\bv}_{{d}}^2}{\kappa}\no{{z}^{(j)}}{}^2+\frac{\kappa}4\Abs{\De{z}^{(j)}}^2,\label{est:sync:trilinear:2}
    \end{align}
where $c_{p,s}=c_{2d/(d-1),1/2}$ is the constant from  \eqref{est:Sobolev} when $p=2d/(d-1)$, $s=1/2$.
Combining \eqref{est:sync:trilinear:1} with~\eqref{est:sync:trilinear:2}, we may then deduce
    \begin{align}\label{est:sync:trilinear:td}
        | b(\bv,{z}^{(j)},\De{z}^{(j)})|
        \leq (\si_jU_{d})^2(t) \frac{1}{\kappa}\no{{z}^{(j)}}{}^2+\frac{\kappa}4\Abs{\De{z}^{(j)}}^2,
    \end{align}
{where $U_{d}(t)$ is defined by \eqref{def:Wt}.}
Lastly, we see that
    \begin{align}\label{eq:feedback:td}
        \mu\lb P_N{z}^{(j)},\De{z}^{(j)}\rb=-\mu\no{{z}^{(j)}}{}^2+\mu\no{Q_N{z}^{(j)}}{}^2,
    \end{align}
where we have applied that $\nabla, Q_N$ commute.
Invoking \eqref{est:Poincare} then yields
    \begin{align}
         \mu\no{Q_N{z}^{(j)}}{}^2\leq \frac{\mu}{N^2}\Abs{\De {z}^{(j)}}^2.\notag
    \end{align}
Finally, let us recall the assumption that $\mu, N>0$ satisfy \eqref{cond:mu:kap:N:H1:td}. Then upon returning to \eqref{eq:sync:H1:td} and combining \eqref{est:sync:force:td}, \eqref{est:sync:trilinear:td}, \eqref{eq:feedback:td}, we deduce \eqref{equ:3.110}, as desired.
\end{proof}

\subsubsection{Convergence of Synchronization and Model Errors }\label{sect:proof:sieve:td}
We are now ready to supply the proof of \cref{thm:converge:sieve:td:detailed} and in particular, the proof of convergence of the algorithm introduced in~\cref{sect:def:sieve:td}.

\begin{proof}[Proof of \cref{thm:converge:sieve:td:detailed}]
Suppose that $\mu$ and $N$ satisfy \eqref{cond:mu:kap:N:convergence}.
From the definition \eqref{eq:sync:initial:td} of the initial condition ${z}^{(j)}(0)$, we see that ${z}^{(j)}(0)= \tau_j {z}^{(j-1)}(0)$.
Using this fact as well as \eqref{est:model:error:Hminus1:td} for $j - 1$, we may apply \cref{lem:sieve:sync:error} where we replace $\tau_j{h}^{(j)}(t)$ in  \eqref{equ:3.100}, to obtain
\begin{align}
\frac{d}{dt}
\Abs{{z}^{(j)}}^2 + \mu\Abs{{z}^{(j)}}^2 \leq \frac{1}{\kap} \min_{\eps\in[0,d/2]}\left\{(\si_jV_{\eps,d})(t) N^{\eps}\right\}^2( 1 + \norm{F_N}_{L,*})^2 \Abs{\tau_j {z}^{(j-1)}}^2,\notag
\end{align}
and an application of the Bellman--Gr\"{o}nwall lemma gives

\begin{align}
        \Abs{{z}^{(j)}(t)}^2 
        & \leq e^{-\mu t}\Abs{\tau_j {z}^{(j-1)}(0)}^2
        +\frac{(1+\no{F_N}{L,*})^2 \min_{\eps\in[0,d/2]}\left\{V_{\eps,d} N^{\eps}\right\}^2}{\kap}
        \int_0^t e^{-\mu(t-s)}
         \Abs{\tau_j{z}^{(j-1)}(t)}^2 ds \notag \\ 
         & \leq \left( e^{-\mu t} + \frac{(1+\no{F}{L,*})^2 \min_{\eps\in[0,d/2]}\left\{V_{\eps,d} N^{\eps}\right\}^2}{\kap \mu} \right) \left(\sup_{t \geq 0} \Abs{\tau_j{z}^{(j-1)}(t)}\right)^2,\label{est:sync:error:Hminus1:td}
   \end{align}
for all $t\geq0$, where $V_{\eps,d}$ is defined by \eqref{def:VepsdN}.
Since $\mu, N$ satisfies \eqref{cond:mu:kap:N:convergence}, 
we may subsequently choose $t_*$ depending on $\|F\|_{L,*}$, $V_{\eps,d}$, $N$, $\kap$, $\mu$, $\eps$ so that for $t\geq t_*$
    \begin{align}
        e^{-\mu t} &+
        \left(\frac{\kap}{\mu}\right)(1+\no{F_N}{L,*})^2\min_{\eps\in[0,d/2]}\left\{\left(\frac{V_{\eps,d}}{\kap}\right)N^{\eps}\right\}^2,\notag
        \\
        &\leq e^{-\mu t_{*}} + \left(\frac{\kap}{\mu}\right)(1+\no{F_N}{L,*})^2\min_{\eps\in[0,d/2]}\left\{\left(\frac{V_{\eps,d}}{\kap}\right)N^{\eps}\right\}^2
        =: \lambda_*^2 < 1.\notag
    \end{align}
Taking $t_{j} := t_*$ for all $j \in \N$, it then follows from \eqref{est:sync:error:Hminus1:td} that
    \begin{align}\notag
        \sup_{t\geq0}\Abs{\tau_{j+1}{z}^{(j)}(t)}\leq\lam_*\sup_{t\geq0}\Abs{\tau_{j}{z}^{(j-1)}(t)},
    \end{align}
which, upon iteration, yields    \begin{align}\label{est:sync:error:Hminus1:iterated:td}
        \sup_{t\geq0}\Abs{\tau_{j+1}{z}^{(j)}(t)}\leq\lam_*^{j}\left(\sup_{t\geq0}\Abs{\tau_{1}{q}^0(t)}\right)\leq\lam_*^{j}\left(\no{Q_N\ze_0^0}{*}+\frac{1}{\mu\kap}\sup_{t\geq0}\no{{h}^0(t)}{*}\right).
    \end{align}
Finally then, we have from \eqref{est:model:error:Hminus1:td} and \eqref{est:sync:error:Hminus1:iterated:td} that the model error can be controlled as
    \begin{align}\label{est:model:error:Hminus1:conv:td}
    \sup_{t \geq 0}\no{\tau_{j+1}{h}^{(j+1)}(t)}{*}
        &\leq \lam_*^{j}(1+\no{F_N}{L,*})\min_{\eps\in[0,d/2]}\left\{V_{\de,d}N^{\de}\right\}\left(\no{Q_N\ze_0^0}{*}+\frac{1}{\mu\kap}\sup_{t\geq0}\no{{h}^0(t)}{*}\right),
    \end{align}
for all $j \in \N$.
The claimed convergence of both the {synchronization error and model error} now follows by passing to the limit $j\goesto\infty$ in \eqref{est:sync:error:Hminus1:iterated:td}, \eqref{est:model:error:Hminus1:conv:td}. This completes the proof.

\end{proof}

The proof of \cref{thm:converge:sieve:td:detailed:strong} is similar. We nevertheless point out the relevant details.

\begin{proof}[Proof of \cref{thm:converge:sieve:td:detailed:strong}]
The main claim is to establish convergence of ${h}^{(j)}$ in $H$ and ${z}^{(j)}$ in $V$. We thus follow the proof of \cref{thm:converge:sieve:td:detailed} above, but invoke~\eqref{est:model:error:L2:td}, \eqref{equ:3.110} instead of \eqref{est:model:error:Hminus1:td}, \eqref{equ:3.100}, respectively. This is possible since $\mu, N$ is assumed to satisfy \eqref{cond:mu:kap:N:convergence:strong}, which implies that it also satisfies~\eqref{cond:mu:kap:N:H1:td}, where any appearance of $\no{F_N}{L,*}$ is now replaced by $\no{F_N}{L,0}$.
\end{proof}

\subsection{Nudging Algorithm}
We first provide statements regarding the wellposedness of the initial value problem for the nudging system \eqref{eq:nudging:td}:
\begin{Thm}\label{thm:nudge:td:gwp}
Given $\bv\in L_{loc}^\infty(0,\infty;\bbH)$ and $g\in L_{loc}^\infty(0,\infty;V^*)$, for all $\psi_0\in H$, $l_0\in V^*$, $T>0$, there exists a unique $\psi\in C([0,T);H)\cap L^2(0,T;V)$ and $l\in C([0,T);V^*)$ satisfying \eqref{eq:nudging:td} for a.a.~$t\in [0,T]$ and $\psi(0)=\psi_0$ in $H$ as well as $l(0)=l_0$ in $V^*$.
\end{Thm}
A sketch of the proof of~\cref{thm:nudge:td:gwp} will be provided at the end of~\cref{sect:nudging:proof}.
\begin{Def}\label{def:class_W}
  For $\epsilon > 0$ we define
\beq{eq:td:v:bar}
W^2_{\epsilon, d} := \limsup_{t\to \infty}\frac{1}{t}\int_0^t \abs{\bv(s)}_{d/\epsilon}^2\, ds <\infty,
\eeq
and the set of vector fields
\beq{eq:td:v:bar:set}
\mathcal{W}_{\epsilon, d} := \{\bv \in L^{\infty}(0,\infty; \bbH) : W_{\epsilon, d} < \infty\}.
\eeq
\end{Def}
  \begin{Rmk}
  For instance weak solutions of the {2D NSE} with time independent forcing are in~$\mathcal{W}_{\epsilon, 2}$ for all $\epsilon > 0$.
\end{Rmk}
Now suppose that $\epsilon \in (0, 1)$, $d = 2$~or~$3$, $\bv \in \mathcal{W}_{\epsilon,d}$ and $\Gamma \subset V^*$ a set of quasi--finite rank $N_0$.
  Let
  \beq{def:G:nudging}
  \cG := \{\Gamma_N: N \geq N_0\},
  \eeq
  be a family where each $\Gamma_N \subset V^*$ is of quasi-finite rank $N$ with Lipschitz enslaving map $F_N$ so that
    \begin{enumerate}
    \item $\Gamma_{N_0} = \Gamma$, 
    \item $\cG$ is increasing, that is $\Gamma_{M_1} \subset \Gamma_{M_2}$ for $M_1 \leq M_2$, 
    \item for sufficiently large $N \geq N_0$ we have
      \beq{equ:5.1030}
      C_* \frac{\no{F_N}{L,*}}{N} \left(\frac{W_{\epsilon, d}}{\kap}\right)N^{\eps} < \frac{\sqrt{2}}{2},
      \eeq
    \end{enumerate}
    where $C_* = \sqrt{2} C_1$ and $C_1$ is one of the constants appearing in~\cref{thm:trilinear}.
The construction presented in~\cref{rmk:canonical} will result in a suitable $\cG$ but other choices are possible in which the Lipschitz constants in~\eqref{equ:5.1030} depend on $N$. 
\begin{Thm}\label{thm:1.30}
  Consider the system \eqref{eq:nudging:td}, where $\phi$ is the solution of \eqref{eq:td:intro} with the same vector field $\bv \in \mathcal{W}_{\epsilon,d}$.
Let $\Gamma \subset V^*$ a set of quasi--finite rank $N_0$, $\cG$  as in \eqref{def:G:nudging}, and assume that $g$, the true forcing in Equation~\eqref{eq:td:intro} is a member of $\Gamma_{N_0}$.
If $N$ is taken so that the condition~\eqref{equ:5.1030} is satisfied, there exists a choice of $\mu_1,\, \mu_2>0$ such that for the corresponding solution $(\psi, l)$ of the nudging system \eqref{eq:nudging:td} we have
  \begin{equation}
      \lim_{t\to\infty}\|f(t) - g(t)\|_* = 0 \quad \text{and}\quad \lim_{t\to \infty} |\phi(t) - \psi(t)| = 0,\notag
  \end{equation}
with exponential rate, where $f(t) = l(t)+ F_N(l(t))$.
\end{Thm}
\begin{Rmk}
\begin{enumerate}
  \item As in \cref{rmk:Peclet}, we note that we can introduce another P\'eclet-type number in this context:
    \begin{align}\label{def:Peclet:nudging}
        \widetilde{\Pe}_{\eps,d}:=\frac{W_{\eps,d}}{\kap}.
    \end{align}
    Thus \eqref{equ:5.1030} can also be viewed in terms of $\widetilde{\Pe}_{\eps,d}$.
  \item Suppose that the forcing $g$ in \eqref{eq:td:intro} is an element of some $\Gamma$ of finite rank $N_0$ and order $(\alpha, \beta)$ where $\beta \geq -1$ and so that $(\alpha + 1)_+ - (\beta + 1) < 1$.
  Then the family $\{\Gamma_N\}_{N \geq N_0}$ constructed in \cref{thm:1.20} has the properties required for the increasing family of sets in \cref{thm:nudging:intro}.
  Therefore the statement of the Theorem applies and in fact $l(t) + {F_N}(l(t)) \to g$ in $H^{\beta}$.
  \item Note that the magnitude of the roots $\lambda_{1,2}$ is irrelevant for the result to hold as long as they are positive.
   \item In order to deploy the methodology and use \eqref{eq:nudging:td} to reconstruct the forcing, knowledge of the observations $P_N \phi$, the vector field $\bv$ (both as functions of space and time) and of the set $\Gamma_N$ (which includes knowledge of the function $F_N$) is required, for the sufficiently large $N$ mentioned in the Theorem.
   \end{enumerate}
  \end{Rmk}
The proof of \cref{thm:1.30} is deferred to Section~\ref{sect:nudging:proof} and will make crucial use of Bellman-Gr\"onwall type inequality of a quadratic form. We will invoke \cref{thm:1.50} as a blackbox to prove \cref{thm:1.30} and supply the proof of \cref{thm:1.50} afterwards.

  \begin{Lem}\label{thm:1.50}
    Let $\xi, \eta$ be nonnegative and absolutely continuous functions on $\R_{\geq 0}$ satisfying the inequality
    \beq{equ:1.160}
      \frac{\dd}{\dd t} (\xi^2(t) + \eta^2(t))
      \leq a \xi^2(t) + b \eta^2(t) + c(t) \xi(t) \eta(t),
    \eeq
    where $c: \R_{\geq 0} \to \R$ is measurable and $a, b \in \R$.
    Then for every $\epsilon > 0$ there exists a $C_{\epsilon} > 0$ so that 
    \beqn{equ:1.200}
    \xi^2(t) + \eta^2(t)
    \leq  (\xi^2_0 + \eta^2_0) C_{\epsilon} \exp (t \rho_{\eps}),
    \eeq
    for $t \geq 0$, with
    \beq{equ:1.170}
    \begin{split}
    \rho_{\eps}  :&= \frac{1}{2} \left(a + b + \sqrt{ \left(a - b\right)^2
      + (1 + \epsilon) \overline{c}^2} \right) \\
      & = \frac{1}{2} \left(a + b + \sqrt{ \left(a + b\right)^2
      - \left(4 ab - (1 + \epsilon) \overline{c}^2 \right)} \right),
    \end{split}
    \eeq
    and $\overline{c}^2 := \limsup_{t \to \infty} \frac{1}{t} \int_0^t c(s)^2 \dd s$.
  \end{Lem}

\begin{Rmk}
  Note that $\overline{c}^2$ is well defined but in general might be infinite. 
  Note also that due to the first expression in ~\eqref{equ:1.170}, if $\overline{c}^2 < \infty$ then $\rho_{\eps} \in \R$.
\end{Rmk}

\subsection{Proofs of the Main Theorems for Nudging Algorithm}\label{sect:nudging:proof}
Let us first prove \cref{thm:1.30}, assuming \cref{thm:1.50}. We will then prove \cref{thm:1.50} immediately afterwards, thereby formally completing the proof of \cref{thm:1.30}.
\begin{proof}[Proof of \cref{thm:1.30}]
  First note that 
  \begin{align}
  \label{eq:f-g}
  \begin{split}
      \norm{f(t)-g}_{*}=\norm{l(t) + {F_N}(l(t)) - g}_{*} &\leq \norm{l(t) - P_N g}_{*} + \norm{{F_N}(l(t)) - {F_N}(P_N g)}_{*} \\
      &\leq (1 + \norm{F_N}_{L,*}) \norm{l(t) - P_N g}_{*},
      \end{split}
  \end{align} 
  so it is sufficient to consider the large-scale error dynamics for ${e(t)} := P_N g - l(t)$, as well as error dynamics for $\phi(t) - \psi(t)$, the latter of which we split into $P_N(\phi(t) - \psi(t)) =: p(t)$ and $Q_N(\phi(t) - \psi(t)) =: q(t)$.
  We obtain the error dynamics
  \begin{align}
  \pdd_t {p} + P_N \bv \cdot \nabla {q}
  & = {e} - \mu_1 {p}, \label{equ:1.50.1}\\
  \pdd_t {q} + Q_N \bv \cdot \nabla {q}
  & = \kappa \Delta {q} + {F_N}(P_N g) - {F_N}(l), \label{equ:1.50.2}\\
  \pdd_t {e} & = - \mu_2 {p}.  \label{equ:1.50.3}
  \end{align}
We consider the following change of variables and parameters: 
\begin{equation}\notag
    r := \frac{e}{\lambda_2} - p,  \quad \lambda_1 + \lambda_2 = \mu_1, \quad \lambda_1\lambda_2 = \mu_2
\end{equation}
where we assume that $\lambda_1 \geq \lambda_2$ without loss of generality.
We can then write ~\eqref{equ:1.50.1}-\eqref{equ:1.50.3} as 
\begin{align}
    %
  \pdd_t {p} + P_N \bv \cdot \nabla {q}
  & = \lambda_2 {r} - \lambda_1 {p}, \label{equ:1.55.1}\\
  \pdd_t {q} + Q_N \bv \cdot \nabla {q}
  & = \kappa \Delta {q} + {F_N}(l + \lambda_2 ({r} + {p})) - {F_N}(l), \label{equ:1.55.2}\\
  \pdd_t {r} - P_N \bv \cdot \nabla {q} & = - \lambda_2 {r}.  \label{equ:1.55.3}
  \end{align}
  We now take the $H$ scalar product of~\eqref{equ:1.55.2} with ${q}$, and then treat \eqref{equ:1.55.1}-\eqref{equ:1.55.3} jointly.
  
  By \eqref{eq:trilinear:skew} it follows $({q}, Q_N \bv \cdot \nabla {q}) = ({q}, \bv \cdot \nabla {q}) = 0$.
  Furthermore, by the Lipschitz property of $F_N$, we get 
  $$|({q}, {F_N}(P_N g) - {F_N}(l))| \leq \norm{{q}} \norm{{F_N}(P_N g) - {F_N}(l)}_{*} \leq \norm{{q}}\norm{F_N}_{L,*} \norm{e}_{*}=\lambda_2 \norm{{q}}\norm{F_N}_{L,*} \norm{{r} + {p}}_{*}.$$
      Using these facts and Young's inequality we find
  \beq{equ:1.70}
  \frac{d}{dt} \abs{{q}}^2
  \leq - \kappa \norm{{q}}^2 + \frac{\lambda_2^2 \norm{F}_{L,*}^2}{\kappa} \norm{{r} + {p}}_{*}^2.
  \eeq
  Next, we take the $V^*$ inner product of \eqref{equ:1.55.1}-\eqref{equ:1.55.3} with ${p}$ and ${r}$, respectively, then add the results to obtain  
\beq{equ:1.80}
  \frac{1}{2}\frac{d}{dt} (\norm{{p}}^2_{*} + \norm{{{r}}}^2_{*}) 
  + b(\bv, {q}, \Delta^{-1} {p})
  - b(\bv, {q}, \Delta^{-1} {r})
  \leq  \lambda_2 \norm{{p}}_{*} \norm{{r}}_{*} - \lambda_1 \norm{{p}}^2_{*} - \lambda_2 \norm{{r}}^2_{*},
  \eeq
  where we have used that $|\scp{\Delta^{-1} {p}}{{r}}| \leq \norm{{p}}_{*} \norm{{r}}_{*}$.
  By ~\eqref{est:trilinear:Sobolev1} and~\eqref{est:Bernstein} from ~\cref{thm:trilinear}, we obtain
\begin{align}\notag
        |b(\bv, {q}, \Delta^{-1} {r})| \leq C_1 \abs{\bv}_{d/\epsilon} \abs{q} N^{\epsilon} \norm{{r}}_{*}\quad\text{and}\quad |b(\bv, {q}, \Delta^{-1} {p})| \leq C_1 \abs{\bv}_{d/\epsilon} \abs{q} N^{\epsilon} \norm{{p}}_{*}.
\end{align}
  Next, we note that 
  \beq{equ:1.90}
  \lambda_2 \norm{{p}}_{*} \norm{{r}}_{*} - \lambda_1 \norm{{p}}^2_{*} - \lambda_2 \norm{{r}}^2_{*} \leq -\frac{\lambda_2}{2}(\norm{{p}}_{*}^2 + \norm{{r}}_{*}^2)
  \eeq
  due to the fact that $\lambda_{1} \geq \lambda_2$.
  In summary we obtain
     \beq{equ:1.100}
  \frac{1}{2}\frac{d}{dt} (\norm{{p}}^2_{*} + \norm{{r}}^2_{*})
  \leq 
  C_1 \abs{\bv}_{d/\epsilon} \abs{q} N^{\epsilon} (\norm{{p}}_{*} + \norm{{r}}_{*})
  -\frac{\lambda_2}{2}(\norm{{p}}_{*}^2 + \norm{{r}}_{*}^2)
  \eeq
   This relation is coupled to ~\eqref{equ:1.70}.
   To analyze these relations simultaneously, we introduce the variable $E := \norm{{p}}^2_{*} + \norm{{r}}^2_{*}$ and note that $\norm{{r} + p}^2_{*} \leq (\norm{{r}}_{*} + \norm{p}_{*})^2 \leq 2 E$.
   Next we multiply ~\eqref{equ:1.100} with a parameter $\alpha > 0$ (to be determined), add to ~\eqref{equ:1.70} and replace with $E$ accordingly.
   We obtain
   \beq{equ:1.140}
  \frac{d}{dt} (\abs{{q}}^2 + \alpha E)
  \leq - \kappa N^2 \abs{{q}}^2 + 2\frac{\lambda_2^2 \norm{F_N}_{L,*}^2}{\kappa} E
+ \alpha C_2 \abs{\bv}_{d/\epsilon} \abs{q} N^{\epsilon} E^{1/2}
  - \alpha \lambda_2 E;
  \eeq
  where we have also used the generalized Poincar\'{e} inequality~\eqref{est:Poincare}.
  Our analysis of this differential inequality makes use of \cref{thm:1.50}.

  Indeed, let us now apply \cref{thm:1.50} to the estimate~\eqref{equ:1.140} with $\xi = \abs{{q}}$ and $\eta = \sqrt{\alpha E}$.
  We find
  \beq{equ:1.180}
  \abs{{q}}^2 + \alpha E \leq (\abs{{q}_0}^2 + \alpha E(0)^2) C_{\epsilon} \exp(t \rho_{\eps}),
  \eeq
  with
  \beqn{equ:1.210}
    \begin{split}
    2\rho_{\eps} 
    & = -\left(\kappa N^2 + \lambda_2 
    - \frac{2 \lambda^2_2\norm{F_N}_{L,*}^2}{\alpha \kappa}\right) \\
    & \quad  + \left\{ \left(\kappa N^2 + \lambda_2 
    - \frac{2\lambda^2_2\norm{F_N}_{L,*}^2}{\alpha \kappa}\right)^2 
    - \left( 4\kappa N^2(\lambda_2 - \frac{2\lambda^2_2\norm{F_N}_{L,*}^2}{\alpha \kappa}) 
    - (1 + \epsilon) \alpha C_2^2 W_{\epsilon,d}^2 N^{2\epsilon}\right) \right\}^{\frac{1}{2}}
\end{split}
\eeq
where $W_{\epsilon,d}^2 $ is as in Definition~\ref{def:class_W} and finite by assumption.
(We are using the second expression in ~\eqref{equ:1.170} for $\rho_{\eps}$.)
We now need to demonstrate that for sufficiently large $N$ and appropriate choices of $\lambda_2, \alpha, \varepsilon$ we have $\rho_{\eps} < 0$.
For this it is sufficient that 
\begin{align}
\label{equ:1.190.1}
    \kappa N^2 + \lambda_2 
    & >  \frac{2 \lambda^2_2\norm{F_N}_{L,*}^2}{\alpha \kappa} \\
\label{equ:1.190.2}
4\kappa N^2\left(\lambda_2 - \frac{2\lambda^2_2\norm{F_N}_{L,*}^2}{\alpha \kappa}\right) 
    & >  \alpha C_2^2 W_{\epsilon,d}^2 N^{2\epsilon}
\end{align}
It is clear that~\eqref{equ:1.190.2} implies~\eqref{equ:1.190.1}.
To analyze~\eqref{equ:1.190.2}, we note that whenever $A_1, A_2, A_3$ are positive, in order that the inequality 
\beq{equ:1.200}
A_1 - \frac{A_2}{\alpha} > \alpha A_3
\eeq
is true for some $\alpha > 0$, it is sufficient that $A_1^2 - 4 A_2 A_3 > 0$.
We thus find that sufficient for the existence of an $\alpha$ satisfying~\eqref{equ:1.190.2} is the relation
\beq{equ:1.210}
0< 16 \kappa^2 N^4 \lambda_2^2 - 32 N^2 \lambda^2_2\norm{F_N}_{L,*}^2 C_2^2 W_{\epsilon,d}^2 N^{2\epsilon}
= 16 N^{2(1+\epsilon)} \lambda_2^2 (\kappa^2 N^{2(1-\epsilon)} - 2 \norm{F_N}_{L,*}^2 C_2^2 W_{\epsilon,d}^2),
\eeq
This however is satisfied as we are reconstructing a quasi-finite rank forcing term $g\in \Gamma\in \mathcal{G}_*$, taking $C_*= C_2$, and the nudging parameters $\mu_1, \mu_2$ so that and $\lambda_1 > \lambda_2>0$.
\end{proof}

Finally, let us now prove \cref{thm:1.50}, which will formally complete the proof of \cref{thm:1.30}.

\begin{proof}[Proof of Lemma~\ref{thm:1.50}]
  The inequality in display~\eqref{equ:1.160} gives
    \beqn{equ:1.220}
    \frac{\dd}{\dd t} (\xi^2(t) + \eta^2(t))
    \leq \rho(t) (\xi^2(t) + \eta^2(t)),
    \eeq
    if we let $\rho(t)$ be the largest eigenvalue of the symmetric matrix
    \beqn{equ:1.225}
      \begin{bmatrix} a& \frac{c(t)}{2} \\
        \frac{c(t)}{2} & b \end{bmatrix}.
    \eeq
Thus by the Bellmann--Gr\"{o}nwall lemma we obtain
    \beq{equ:1.240}
    \xi^2(t) + \eta^2(t)
    \leq  (\xi^2_0 + \eta^2_0) \exp \left( \int_0^t \rho(s) \idd s \right),
    \eeq
    for any $t \geq 0$.
But an elementary calculation shows that $\rho(t)$ is given by 
    \beq{equ:1.250}
    \rho(t) := \frac{1}{2} \left(a + b + \sqrt{ \left(a - b\right)^2
      + c(t)^2} \right).
    \eeq
    By H\"{o}lder's inequality, we find (now for $t > 0$)
    \beq{equ:1.260}
    \frac{1}{t}\int_0^t\rho(s) \idd s
    \leq 
    \frac{1}{2} \left(a + b + \sqrt{ \left(a - b\right)^2
      + \frac{1}{t}\int_0^tc(s)^2 \idd s} \right).
    \eeq
But for any $\epsilon > 0$ we can find $t_{\epsilon} \geq 0$ so that $e(t) \leq  \overline{c}^2 (1 + \epsilon)$ whenever $t \geq t_{\epsilon}$ due to the definition of $\overline{c}$.
    Using this estimate in~\eqref{equ:1.260} and applying it to~\eqref{equ:1.240} we find that
    \beq{equ:1.280}
    \xi^2(t) + \eta^2(t)
    \leq  (\xi^2_0 + \eta^2_0) \exp \left( \rho_{\eps} t\right),
    \eeq
    for $t \geq t_{\epsilon}$ and $\rho_{\eps}$ as in ~\eqref{equ:1.210}.
    As the left hand side of~\eqref{equ:1.280} is clearly also bounded if $t < t_{\varepsilon}$, we have 
    \beqn{equ:1.290}
    \sup_{t \geq 0}e^{-\rho_{\eps}t}\frac{\xi^2(t) + \eta^2(t)}{\xi^2_0 + \eta^2_0} =: C_{\epsilon} < \infty,
    \eeq
    finishing the proof.
        \end{proof}
Finally, we provide a brief outline of the proof of \cref{thm:nudge:td:gwp}.
\begin{proof}[Proof sketch of~\cref{thm:nudge:td:gwp}]
%
The proof of global existence and uniqueness of weak solutions to \eqref{eq:nudging:td} follows from applying the Galerkin method to the Equations~(\ref{equ:1.50.1}--\ref{equ:1.50.3}).
Let $M$ be large enough so that $M > N$ and therefore $P_M H \supset P_N H$.
The solutions $p_M,q_M,r_M$ to the order~$M$ Galerkin truncation of equations~(\ref{equ:1.50.1}--\ref{equ:1.50.3}) respectively exist locally in time since the nonlinearity is locally Lipschitz.
Furthermore, they satisfy the following apriori estimate:
employing a more straightforward variation of the analysis of the proof of \cref{thm:1.30} above, and upon setting $\tp=\sqrt{\mu}_2p$, we obtain the following basic energy balance:
    \begin{align}
        \frac{1}2\frac{d}{dt}\left(|\tp|^2+|e|^2\right)+\mu_1|\tp|^2&=-(P_N\bv\cdotp\nabla q,\tp)\notag
        \\
        \frac{1}2\frac{d}{dt}|q|^2+\kap\|q\|^2&=(F_N(P_Ng)-F_N(\ell),q).\notag
    \end{align}
It follows that
\beq{est:apriori:nudging5}
    \begin{split}
        & \frac{d}{dt}\left(|\tp|^2+|e|^2+|q|^2\right) +\mu_1|\tp|^2+\kap\|q\|^2\leq \left[C_A\left(\frac{\mu_2}{\mu_1}\right)N^2\left(\sup_{0\leq t\leq T}|\bv(t)|\right)^2+\frac{\|F_N\|_*^2}{\kap}\right]|q|^2.
    \end{split}
    \eeq
From this a priori estimate, one may then argue as in \cref{thm:td:exist} to complete the proof: since that estimate is uniform in~$M$, the compactness theorems of Aubin--Lions show the existence of subsequences as $M \to \infty$ which converge to a solution of equations~(\ref{equ:1.50.1}--\ref{equ:1.50.3}) satisfying the requirements of the theorem.
This approach is well--known, and for additional relevant details we refer the reader to \cref{sect:appendix}, where \cref{thm:td:exist} is proved, and to~\citet{RobinsonBook,TemamBook1997}.
\end{proof}

\section{Reconstruction of Forcings in the 2D--Navier--Stokes Equation}\label{sect:alg1}
In this section, we will fix the dimension parameter to be $d=2$. We will show that both the Sieve as well as the Nudging Algorithms can be applied to reconstruct forcings or surface fluxes in the 2D--Navier-Stokes equation. 
We will consider \eqref{eq:nse:intro} over the periodic box $\T^2=[0,2\pi]^2$.
Recall that we seek to reconstruct the \textit{non-potential} component of the force $\mbg$. Thus we will assume that $\PP\mbg=\mbg$. 
Throughout this section, we will assume that $\bu_0\in \bbV\cap H^2(\T^2)^2$ and $\bg\in L^\infty(0,\infty;\bbH)$, so that \eqref{eq:nse:intro} admits a unique global strong solution solution $\bu\in C([0,T]; \bbV)\cap {L^{2}}(0,T; \bbD(A))$, for every $T\geq 0$ (see \cite{ConstantinFoiasBook, TemamNSEBook}). We denote the unique solution corresponding to data $\bu_0,\mbg$ by $\bu=\bu(\cdotp;\bu_0,\mbg)$.

Given a solution $\bu$ of \eqref{eq:nse:intro}, force $\mbf$ such that ${\PP}\mbf=\mbf$, the nudged equation corresponding to observations $\OO=\{P_N\bu(t)\}_{t\geq0}$ is given by
    \begin{align}\label{eq:nse:nudge}
        \bdy_t\bv+B(\bv,\bv)=-\nu{A}\bv+\mbf-\mu {P_N}\bv+\mu {P_N}\bu,
    \end{align}
where $A$ denotes the Stokes operator \eqref{def:Stokes}. The solution theory \eqref{eq:nse:nudge} in the functional setting introduced in \cref{sect:notation:nse} was developed in \cite{AzouaniOlsonTiti2014}, where the following result was proved.
\begin{Thm}\label{thm:nse:existence:nudge}
Suppose $d=2$. Given a Leray-Hopf weak solution $\bu=\bu(\cdotp;\bu_0,\mbg)$ of \eqref{eq:nse:intro} over $[0,T]$, for any $\mu, N>0$, $\bv_0\in {\bbH}$ and {$\mbf\in L^\infty(0,T;\bbV^*)$}, there exists a unique Leray-Hopf weak solution of \eqref{eq:nse:nudge} over $[0,T]$. If $\bu$ is a strong solution, $\bv_0\in\bbV$ and $\mbf\in L^2(0,T;\bbH)$, then the Leray-Hopf weak solution $\bv$ is also a strong solution.  
\end{Thm}
\subsection{Sieve Algorithm}
We will now state and prove the precise form of \cref{thm:main:sieve:nse:intro}.
We will first introduce a suitable family of quasi-finite functions, similar to what we did in Section~\ref{sect:algorithm1}.
Let $\Gamma \subset \bbV$ be a set of functions with quasi--finite rank $N_0$ and order $(0, 0)$.
Given $\al,\be,\gam,\si>0$ and $\Rone, M_0>0$ as well as a nonnegative function $C(\al,\be,\gam)$ (the precise shape of which will be given in the proof), we define a family
  \beq{equ:5.1000_NS}\notag
  \cG := \{\Gamma_N: N \geq N_0\},
  \eeq
  where $\Gamma_N \subset \bbV$ is of quasi-finite rank $N$ and order $(0, 0)$ with Lipschitz enslaving map $F_N$ such that
    \begin{enumerate}
    \item $\Gamma_{N_0} = \Gamma$, 
    \item $\cG$ is increasing, that is $\Gamma_{M_1} \subset \Gamma_{M_2}$ for $M_1 \leq M_2$, 
    \item for sufficiently large $N$ we have
      \beq{equ:6.2010}
      C_F(N) < \frac{N}{2},
      \eeq
        where
        \begin{align}\notag
          C_F(N) := \frac{\sqrt{C(\al,\be,\gam)}}{\nu}
          \max\left\{
          \frac{M_0}{R},
          (1+\|F_N\|_{L,0})
          R \sqrt{\ln\left(e+\si+\frac{R}{\nu}\right)}
          \right\}.
    \end{align}
    \end{enumerate}
    \begin{Thm}\label{thm:main:sieve:a}
Let $\Gamma \subset \bbV$ be a set of functions with quasi--finite rank $N_0$ and order $(0, 0)$.
Let $\mbg$, the true forcing in the 2D--NSE~\eqref{eq:nse:intro} be an element in $L^\infty(0,\infty;\Gam)$,
while $\bv_0^0 \in \bbV \cap H^2(\T^2)^2$
and $\mbphi \in L^\infty(0,\infty;\bbH)$.
Let $\al,\be, \gam, R, M_0, \si$ be positive numbers so that the following bounds hold:
\begin{align}
    \begin{split}
  \|\bu_0\| & \leq \al R \\
  |A\bu_0| & \leq \be R\\
\no{\bv_0^0-\bu_0}{} & \leq \gam\Rone\\
\nu G & \leq R\notag\\
\frac{\sup_{t\geq0}\|\mbg(t)\|}{\sup_{t\geq0}|\mbg(t)|} & \leq \si\\
\sup_{t\geq0}\|\mbphi(t) - \mbg(t)\| & \leq M_0,
    \end{split}\notag
    \end{align}
where $G$ is the Grashof number of $\mbg$ defined in~\eqref{def:Grashof}.
Subsequently, define the family $\cG := \{\Gamma_N\}_{N \geq N_0}$ as above (with $C(\alpha, \beta, \gamma)$ given in the proof).
Then, provided $N \geq N_0$ is sufficiently large so that condition~\eqref{equ:6.2010} is satisfied, the parameter $\mu$ is chosen so that
\beq{cond:mu:sieve:a}
        C_F^2\nu\leq\mu< \frac{1}{4} N^2\nu, 
\eeq
there exists $t_* > 0$ such that the Sieve Algorithm initialized with $\bv^{(0)}_0$ and $\mbf^{(0)} = P_N \mbphi + F_N(P_N \mbphi)$ satisfies
\begin{align}\notag
\sup_{t\geq0}\Abs{\tau_{j+1}\mbf^{(j+1)}(t)-\si_{j+1}\mbg(t)}\leq\frac{1}2\sup_{t\geq 0}\Abs{\tau_j\mbf^{(j)}(t)-\si_j\mbg(t)},
\end{align}
for all $j\geq0$, where $t_j=t_*$, and moreover
\beq{eq:state:convergence}\notag
\lim_{j\goesto\infty}\sup_{t\geq0}\no{\tau_{j+1}\bu^{j}(t)-\si_{j+1}\bu(t)}{}=O(2^{-j}).
\eeq
\end{Thm}
\begin{Rmk}\label{rmk:NS_sieve}
\begin{enumerate}
   \item In order to deploy the methodology and use \cref{thm:main:sieve:a} to reconstruct the forcing, knowledge of the observations $P_N \mbphi$ and of the set $\Gamma_{N}$ (which includes knowledge of the function $F_N$) is required, for the sufficiently large $N$ mentioned in the Theorem.
\item There is no contradiction in assuming that $\Gamma_{N} \subset \bbV$ but of order $(0, 0)$; we recall that the order just refers to the regularity of the Lipschitz enslaving map.
This means that although $F_N$ have graphs in $\bbV$, namely $Q_{N} \mbg = F_N(P_{N} \mbg) \in \bbV$, the regularity of $F_N$ is only Lipschitz with values in $\bbH$, that is $F_N \in \Lip(P_{N}\bbH, Q_{N}\bbH)$.
Therefore, despite the true forcing $\mbg$ in the system \eqref{eq:nse:intro} as well as the approximations $\mbf(t)$ are elements of $\bbV$, the approximation is only in $\bbH$.
\end{enumerate}
\end{Rmk}
\subsection{Proof of Theorem \ref{thm:main:sieve:a}}
\subsubsection{Error Representation and Analysis of Model Error}
To prove \cref{thm:main:sieve:a}, 
denote the stage-$j$ synchronization error by $\bw^{(j)}=\bv^{(j)}-\si_j\bu$ and let us first observe that
    \begin{align}\label{def:state:error:high:j}
        \bu^{(j)}-\si_j\bu={Q_N}\bv^{(j)}-Q_N\si_j\bu={Q_N}\bw^{(j)}.
    \end{align}
Also, denote the stage-$(j+1)$ model error by
    \begin{align*}
        \bh^{(j+1)}:={\mbf^{(j+1)}}-{\si_{j}\mbg},\quad \mbe^{(j+1)}:= P_N\bh^{(j+1)}={\mbl^{(j+1)}}-{\si_{j}\mbl}.
    \end{align*}
Then
    \begin{align}\label{def:model:error:low:nse}
        \mbe^{(j+1)}&=\left(\bdy_t{P_N}\bu^{(j)}-\nu {P_N}A\bu^{(j)}+{P_N}(\bu^{(j)}\cdotp\nabla\bu^{(j)})\right)-\left(\bdy_t{P_N}\si_j\bu-\nu {P_N}A\si_j\bu+{P_N}(\si_j\bu\cdotp\nabla\si_j\bu)\right)\notag\\
        &={P_N}({Q_N}\bw^{(j)}\cdotp\nabla {Q_N}\bw^{(j)})+{P_N}(\si_j\bu\cdotp\nabla {Q_N}\bw^{(j)})+{P_N}(Q_N\bw^{(j)}\cdotp\nabla\si_j\bu).
    \end{align}
Then
    \begin{align}\label{def:model:error:nse}
        \bh^{(j+1)}=\mbe^{(j+1)}+({Q_N}F({\mbl^{(j+1)}})-{Q_N}F(\mbl)).
    \end{align}
We therefore see that in order to close the analysis for the model error at the current stage, $\bh^{(j+1)}$, we must show that the synchronization error from the previous stage, $\bw^{(j)}$, can be controlled by model error, $\bh^{(j)}$, {\em at the same stage}; this will be studied in the sequel. Let us then first establish estimates for $\bh^{(j+1)}$ in terms of $\bw^{(j)}$. 

\begin{Lem}
There exists $C>0$, for all $N$, for all $F\in\Lip(P_N\bbH,Q_N\bbH)$, such that
    \begin{align}\label{est:model:error:L2}
        \Abs{\bh^{(j+1)}(t)}{}\leq 
       C{(1+\norm{F}_{L,0})}\left[\no{\bw^{(j)}}{}+\|\si_j\bu\|\left(\log_+(e+{\Abs{A\si_j\bu}})\right)^{1/2}\right]\no{\bw^{(j)}}{},
    \end{align}
for all $t\geq0$. 
\end{Lem}

\begin{proof}
Suppose $F\in\Lip(P_N\bbH,Q_N\bbH)$ with Lipschitz norm $\norm{F}_{L,0}$. From \eqref{def:model:error:nse}, it follows that
    \begin{align}
        &\Abs{\bh^{(j+1)}}^2
        =\Abs{\mbe^{(j+1)}}^2+\Abs{F({\mbl^{(j+1)}})-F(\mbl)}^2\leq(1+\norm{F}_{L,0}^2)\Abs{\mbe^{(j+1)}}^2\notag.
    \end{align}
Now from \eqref{def:model:error:low:nse}, we have
    \begin{align*}
        |\mbe^{(j+1)}|^2\leq 2\left(|b (Q_N\bw^{(j)},Q_N\bw^{(j)}, \mbe^{(j+1)})|
        +|b(\si_j\bu, Q_N\bw^{(j)},\mbe^{(j+1)})|
        +|b(Q_N\bw^{(j)},\si_j\bu,\mbe^{(j+1)})|\right).
    \end{align*}
Recall that $\mbe^{(j+1)}=P_N\mbe^{(j+1)}$. We estimate these three terms with \eqref{est:trilinear:Holder}, \eqref{est:trilinear:Bernstein}, \eqref{est:Poincare}, and \eqref{est:Bernstein} to obtain
    \begin{align}\notag
        \begin{split}
        |b (Q_N\bw^{(j)},Q_N\bw^{(j)}, \mbe^{(j+1)})|
        &\leq C{N}\Abs{Q_N\bw^{(j)}}^2\|\mbe^{(j+1)}\|\leq C\no{\bw^{(j)}}{}^2|\mbe^{(j+1)}|
        \\
        |b(\si_j\bu, Q_N\bw^{(j)},\mbe^{(j+1)})|
        &\leq C|\si_j\bu|_\infty|Q_N\bw^{(j)}|\|\mbe^{(j+1)}\|\leq C|\si_j\bu|_\infty\|\bw^{(j)}\||\mbe^{(j+1)}|
        \\
        |b(Q_N\bw^{(j)},\si_j\bu,\mbe^{(j+1)})|
        &\leq C|Q_N\bw^{(j)}|\|\mbe^{(j+1)}\||\si_j\bu|_\infty\leq C|\si_j\bu|_\infty\|\bw^{(j)}\||\mbe^{(j+1)}|.
        \end{split}
    \end{align}
An application of \eqref{est:BrezisGallouet}, the property $|\ln(a/b)|\leq|\ln a|+|\ln b|$, and \eqref{est:Poincare}  yield
    \begin{align}\notag
        \Abs{\mbe^{(j+1)}}{}&\leq C\left[\no{\bw^{(j)}}{}+\|\si_j\bu\|\left(1+\left(\ln\frac{\Abs{A\si_j\bu}}{\|\si_j\bu\|}\right)^{1/2}\right)\right]\no{\bw^{(j)}}{}
        \\
        &\leq C\left[\no{\bw^{(j)}}{}+\|\si_j\bu\|\left(\log_+(e+{\Abs{A\si_j\bu}})\right)^{1/2}\right]\no{\bw^{(j)}}{},\notag
    \end{align}
Thus
    \begin{align}\notag
        \Abs{\bh^{(j+1)}}{}
        &\leq C{(1+\norm{F}_{L,0})}\left[\no{\bw^{(j)}}{}+\|\si_j\bu\|\left(\log_+(e+{\Abs{A\si_j\bu}})\right)^{1/2}\right]\no{\bw^{(j)}}{},\notag
    \end{align}
which is precisely \eqref{est:model:error:L2}.
\end{proof}

\subsubsection{Analysis of Synchronization Error}

In this section, we will establish estimates for the synchronization error $\bw^{(j)}$. First, observe from \eqref{eq:vj} and \eqref{eq:nse:intro}, that $\bw^{(j)}$ is governed by
    \begin{align*}
        \bdy_t\bw^{(j)}+B(\bw^{(j)},\bw^{(j)})+B(\si_j\bu,\bw^{(j)})+B\bw^{(j)},\si_j\bu)=-\nu A\bw^{(j)}-\mu {P_N}\bw^{(j)}+\tau_j\bh^{(j)},
    \end{align*}
with initial data $\bw^{(j)}(0)=\bv^0_0-\bu_0$, when $j=0$, and $\bw^{(j)}(0)=\bv^{(j)}_0-\si_j\bu(0)=\bv^{(j-1)}(t_j)-\si_{j-1}\bu(t_j)=\bw^{(j-1)}(t_{j})$, when $j\geq1$. The main result of this section is the following statement.

\begin{Lem}\label{lem:nse:sync:error}
Assume that $\mu ,N>0$ satisfy
    \begin{align}\label{cond:mu:wj}
        \mu\leq\frac{1}4\nu N^2.
    \end{align}
There exist a constant $C>0$ such that for any $\|\bu_0\|\leq\al R$, where $\al>0$, if
$\mu$ additionally satisfies 
    \begin{align}\label{cond:stability:H1:wj}
        \mu\geq 
        C_j(\al)R\left(\frac{R}{\nu}+\frac{\sqrt{2}}{2}\right),\quad C_j(\al):=C\begin{cases}
            \al^2,&j=0,\ \al\geq1\\
            2,&j>0,\ \al\geq1\\
            1,&j\geq0,\ \al<1.
        \end{cases}
    \end{align}
then
    \begin{align}\label{est:sync:error:H1:j}
    \no{\bw^{(j)}(t)}{}^2\leq e^{-\mu t+1}\no{\bw^{(j)}_0}{}^2+\frac{e}{\mu\nu}\left(\sup_{t\geq0}\Abs{\tau_j\bh^{(j)}(t)}{}\right)^2.
    \end{align}
\end{Lem}

To prove \cref{lem:nse:sync:error}, it will be convenient to develop the estimates in a general form, as sensitivity-type estimates. These are captured by 
\cref{thm:stability:H1} and its corollary below. We state and prove these now. We will then prove \cref{lem:nse:sync:error} afterwards.

\begin{Prop}\label{thm:stability:H1}
Given $\mbg\in L^\infty(0,\infty;{\bbH})$ and $\bu_0\in\bbV$, let $\bu$ denote the unique strong solution of \eqref{eq:nse:intro} such that $\bu(0)=\bu_0$. Given $\mbf\in L^\infty(0,\infty;{\bbH})$ and $\bv_0\in\bbV$, let $\bv$ denote the unique strong solution of \eqref{eq:nse:nudge} such that $\bv(0)=\bv_0$. There exists a constant $C>0$, such that if $\mu,N>0$ satisfy
    \begin{align}\label{cond:stability:b}
        \mu\leq \frac{1}4\nu N^2,
    \end{align}
then
    \begin{equation}\label{eq:stability:H1:a}
        \begin{split}
        &\no{\bv(t)-{\bu(t;\tau_{t'}\bu,\tau_{t'}\mbg)}}{}^2\\
        &\leq \exp\left(-\frac{7}4\mu t+\frac{C}{\mu}\int_0^t\Abs{A\bu(s;\tau_{t'}\bu,\tau_{t'}\mbg)}^2ds\right)\no{\bv(0;\bv_0,\mbf)-\bu(t';\bu_0,\mbg)}{}^2\\
        &\quad+\frac{1}{\nu}\int_0^t \exp\left(-\frac{7}4\mu (t-s)+\frac{C}{\mu}\int_s^t\Abs{A\bu(r;\tau_{t'}\bu,\tau_{t'}\mbg)}^2dr\right)\Abs{\mbf(s)-{\tau_{t'}\mbg(s)}}^2ds,
    \end{split}\end{equation}
holds for all $t,t'\geq0$, where $c_L, c_A$ are constants of interpolation defined in \eqref{est:Ladyzhenskaya}, \eqref{est:Agmon}.
\end{Prop}

We observe from \eqref{est:enstrophy:inequality:G} that
    \begin{align}
        \frac{C}{\mu}\int_s^t|A\bu(r;\tau_{t'}\bu,\tau_{t'}\mbg)|^2dr&\leq\frac{C}{\mu}\left(\frac{\|\bu(s;\tau_{t'}\bu,\tau_{t'}\mbg)\|^2}{\nu}+\nu^2G^2(t-s)\right)\notag
        \\
        &\leq \frac{C}{\mu}\left(\frac{\sup_{t\geq0}\|\bu(t;\tau_{t'}\bu,\tau_{t'}\mbg)\|^2}{\nu}+\nu G^2(t-s)\right),\notag
    \end{align}
for all $0\leq s\leq t$. Therefore it follows from \cref{eq:stability:H1:a}:
    
\begin{Cor}\label{cor:stability:H1}
There exists a constant ${{C}}>0$ such that for $t'\geq0$, if $\mu$ additionally satisfies
    \begin{align}\label{cond:stability:H1}
        \mu\geq {{C}}\left(\frac{\sup_{t\geq 0}\|\tau_{t'}\bu(t;\bu_0,\mbg)\|^2}{\nu}+\nu G\right)
    \end{align}
then
    \begin{align}
        &\no{\bv(t)-{\bu(t;\tau_{t'}\bu,\tau_{t'}\mbg)}}{}^2\notag\\
        &\quad\leq {e^{-\mu t+1}\no{\bv(0;\bv_0,\mbf)-\bu(t';\bu_0,\mbg)}{}^2}+\frac{e}{\nu}\int_0^t e^{-\mu(t-s)}\Abs{\mbf(s)-{\tau_{t'}\mbg(s)}}^2ds.\notag
    \end{align}
In particular
    \begin{align}
     &\no{\bv(t)-{\bu(t;\tau_{t'}\bu,\tau_{t'}\mbg)}}{}^2\notag\\
        &\quad\leq {e^{-\mu t+1}\no{\bv(0;\bv_0,\mbf)-\bu(t';\bu_0,\tau_{t'}\mbf)}{}^2}+\frac{e}{\nu\mu}\left(\sup_{t\geq0}\Abs{\mbf(t)-{\tau_{t'}\mbg(t)}}\right)^2.\notag
    \end{align}
\end{Cor}

\begin{proof}[Proof of \cref{thm:stability:H1}] 
Let $\bw=\bv-\bu$ and $\bh=\mbf-\bg$. Then
    \begin{align}\label{eq:nse:difference}
         \bdy_t\bw+B(\bw,\bw)+B(\bu,\bw)+B(\bw,\bu)=-\nu A\bw+\bh-\mu P_N\bw.
    \end{align}
Upon taking the $L^2$ inner product of \eqref{eq:nse:difference} with $A\bw$, then applying the identity \eqref{eq:trilinear:Aid} we obtain
    \begin{align}\label{eq:sync:error:enstrophy}
        \frac{1}2\frac{d}{dt}\no{\bw}{}^2+\nu\Abs{A\bw}^2+\mu\no{\bw}{}^2=b( \bw,\bw,A\bu)+\lb \bh,A\bw\rb+\mu\no{Q_N\bw}{}^2.
    \end{align}

For the trilinear term, we apply \eqref{est:trilinear:Holder}, \eqref{est:BrezisGallouet}, and Young's inequality to obtain
    \begin{align}\label{est:trilinear:H1}
        |b(\bw,\bw,A\bu)|&\leq c_{BG}\Abs{A\bu}\left[1+\left(\ln\left(\frac{\Abs{A\bw}^2}{\no{\bw}{}^2}\right)\right)^{1/2}\right]\|\bw\|^2\notag
        \\
        &\leq \frac{3c_{BG}}2\Abs{A\bu}\left(1+\ln\left(\frac{\Abs{A\bw}}{\no{\bw}{}}\right)^2\right)\|\bw\|^2.
    \end{align}
We estimate the second term in \eqref{eq:sync:error:enstrophy} with the Cauchy-Schwarz inequality and obtain
    \begin{align}\label{est:force:error}
        |\lb\bh,A\bw\rb|\leq\Abs{\bh}\Abs{A\bw}\leq\frac{1}{\nu}\Abs{\bh}^2+\frac{\nu}4\Abs{A\bw}^2.
    \end{align}
For the third term in \eqref{eq:sync:error:enstrophy}, we integrate by parts, apply the Cauchy-Schwarz inequality, Poincar\'e's inequality, and Young's inequality to obtain
    \begin{align}\label{est:spillover}
        {\mu}|\lb  {Q_N}\bw,A\bw\rb|&\leq {\mu} \no{Q_N\bw}{}^2\leq {\frac{\mu}{N^2}}\Abs{A\bw}^2.
    \end{align}
    
Under the assumption \eqref{cond:stability:b}, we may combine \eqref{est:trilinear:H1}, \eqref{est:force:error}, and \eqref{est:spillover} to arrive at
    \begin{align}\label{est:sync:error:enstrophy}
        \frac{d}{dt}\no{\bw}{}^2+\left\{2\mu+\nu\left[\left(\frac{\Abs{A\bw}}{\no{\bw}{}}\right)^2-\frac{3c_{BG}}{2}\left(\frac{\Abs{A\bu}}{\nu}\right)\left(1+\ln\left(\frac{\Abs{A\bw}}{\no{\bw}{}}\right)^2\right)\right]\right\}\no{\bw}{}^2\leq \frac{2}{\nu}\Abs{\bh}^2.
    \end{align}
Observe that the function $f(x)=x-\al\ln x$, for $x,\al>0$, takes on its minimum value at $\al$. Thus \eqref{est:sync:error:enstrophy} reduces to
    \begin{align}\notag
        \frac{d}{dt}\no{\bw}{}^2+\left[2\mu-\nu\frac{3c_{BG}}{2}\left(\frac{\Abs{A\bu}}{\nu}\right)\ln\frac{3c_{BG}}{2}\left(\frac{\Abs{A\bu}}{\nu}\right)\right]\no{\bw}{}^2\leq \frac{2}{\nu}\Abs{\bh}^2.
    \end{align}
By Young's inequality, we have $x\leq 3x^{4/3}/4\mu+\mu/4$. Also, $\ln x\leq C x^{1/2}$, for some $C>0$. Applying this for $x=3c_{BG}\Abs{A\bu}/(2\nu)\ln[3c_{BG}\Abs{A\bu}/(2\nu)]$, we deduce
    \begin{align*}
        \no{\bw(t)}{}^2&\leq \exp\left(-\frac{7}4\mu t+\frac{C\nu^2}{\mu}\int_0^t\left(\frac{\Abs{A\bu(s)}}{\nu}\right)^2ds\right)\no{\bw_0}{}^2\notag
        \\
        &\quad+\int_0^t\exp\left(-\frac{7}4\mu( t-s)+\frac{C\nu^2}{\mu}\int_s^t\left(\frac{\Abs{A\bu(r)}}{\nu}\right)^2dr\right)\Abs{\bh(s)}^2ds,
    \end{align*}
for some constant $C>0$, for all $t\geq0$. Repeating the argument for $\bu\mapsto\bu(t;\tau_{t'}\bu,\tau_{t'}\mbg)$, for any $t\geq t'$ and $t'\geq0$ yields \eqref{eq:stability:H1:a}.
\end{proof}

Finally, we turn to the proof of \cref{lem:nse:sync:error}.

\begin{proof}[Proof of \cref{lem:nse:sync:error}]
Since $G=\sqrt{2}R/(2\nu)$, we see that \eqref{cond:mu:wj}, \eqref{cond:stability:H1:wj} implies that $\mu$ satisfies \eqref{cond:stability:b}, \eqref{cond:stability:H1}, for all $j\geq0$, for some appropriately chosen constant ${{C}}$. We may thus deduce from \cref{cor:stability:H1} that
    \begin{align}
    \no{\bw^{(j)}(t)}{}^2\leq e^{-\mu t}\no{\bw^{(j)}_0}{}^2+\frac{1}{\mu\nu}\left(\sup_{t\geq0}\Abs{\tau_j\bh^{(j)}(t)}{}\right)^2,\notag
    \end{align}
holds for all $j\geq0$, and we are done.
\end{proof}

\subsubsection{Convergence of Synchronization and Model Errors}\label{sect:proof:sieve:nse}

The proof will proceed by induction. We will first prove \cref{thm:main:sieve:a}. 

\begin{proof}[Proof of \cref{thm:main:sieve:a}] Let $\mbf^0\in L^\infty(0,\infty;\bbH)$, $\bv_0^0\in\bbV$. Recall that the initial errors are assumed to satisfy
    \begin{align}\notag
        \|\bu_0\|\leq\al \Rone,\quad |A\bu_0|\leq \be R\quad\sup_{t\geq0}\Abs{\bh^0(t)}\leq M_0, 
        \quad \no{\bw_0^0}{}\leq \gam\Rone.
    \end{align}
For convenience, let us assume that $\gam,\be\geq1$. It will also be convenient to assume that $|A\bu_0|\leq \til{\be}R_2$, where $R_2$ is defined by \eqref{def:absorbing:ball:H2} and then convert back to $R$ via $\til{\be}R_2=\be R$.

When $j=0$, it follows from \eqref{est:sync:error:H1:j} and the assumption 
    \begin{align}\label{cond:mu:initial:error}
    \mu\geq \frac{2}{\gam^2}\frac{M_0^2}{\nu R^2},
    \end{align}
that
    \begin{align}\label{est:sync:error:reduced}
        \begin{split}
        \sup_{t\geq0}\no{\tau_{t_1'}\bw^0(t)}{}&\leq \frac{\sqrt{2}}{(\nu\mu)^{1/2}}\left(\sup_{t\geq0}\Abs{\bh^0(t)}{}\right)\leq \sqrt{2}\frac{M_0}{(\nu\mu)^{1/2}}\leq \gam \Rone,
        \end{split}
    \end{align}
for some $t_1'\geq 0$ satisfying
    \begin{align}\label{def:transients}
           e^{\mu t_1'}\geq \gam^2(\mu\nu)\frac{\Rone^2}{M_0^2},
    \end{align}
In particular, $\|\bw^1(0)\|\leq\gam R$. 

 Let $t_*\geq\max\{T_1,T_2, t_1'\}$, where $T_1,T_2$ are the absorbing times from the discussion following \eqref{est:uniform:bounds} and from \eqref{def:absorbing:ball:H2}, respectively. For $t_1:=t_*$, it now follows from \eqref{est:sync:error:reduced}, \eqref{est:model:error:L2} that
    \begin{align}\notag
        \sup_{t\geq 0}\Abs{\tau_1\bh^{1}(t)}\leq\frac{C{(1+\norm{F}_{L,0})}}{(\nu\mu)^{1/2}}\left[\frac{M_0}{(\nu\mu)^{1/2}}+C_1(\al)\Rone\left(\ln\left(e+\frac{\til{\be}R_2}{\nu}\right)\right)^{1/2}\right]\sup_{t\geq0}\Abs{\tau_0\bh^0(t)}{}\notag,
    \end{align}
for some sufficiently large absolute positive constant $C$. We may thus choose $\mu$ sufficiently large such that
    \begin{align}\notag
        \sup_{t\geq 0}\Abs{\tau_{1}\bh^{1}(t)}\leq \frac{1}2\sup_{t\geq 0}\Abs{\tau_{0}\bh^0(t)}.
    \end{align}
In particular, we may choose $\mu$ such that
    \begin{align}\label{cond:mu:final}
        \mu\geq C(C_1(\al))^2
        \ln(e+\be)\max\left\{\frac{M_0^2}{\nu^2 R^2},\ (1+\|F\|_{L,0})^2\frac{R^2}{\nu^2}\ln\left(e+\frac{R}{\nu}\right)\right\}\nu,
    \end{align}
for a sufficiently large absolute constant $C>0$, which, upon recalling \eqref{def:absorbing:ball:H2}, is produces the condition \eqref{cond:mu:sieve:a}. This completes the base case.

Now suppose $j\geq1$ and for $k=1,\dots, j$, we have
    \begin{align}\label{cond:induction:hypothesis:a}
        \sup_{t\geq 0}\Abs{\tau_{k}\bh^{k}(t)}\leq \frac{1}2\sup_{t\geq 0}\Abs{\tau_{k-1}\bh^{k-1}(t)},
    \end{align}
for some $t_k>0$, where $t_{k-1}=0$ for $k=1$, and for $k=2,\dots, j$,  $t_k'$ satisfies \eqref{def:transients}. Moreover, suppose it holds that
    \begin{align}\label{cond:induction:hypothesis:b}
        \max_{k=0,\dots, j}\sup_{t\geq 0}\no{\bw^k(t)}{}^2\leq \frac{\sqrt{2}}{(\nu\mu)^{1/2}}\left(\sup_{t\geq0}\Abs{\bh^0(t)}{}\right).
    \end{align}
Observe that by \eqref{cond:mu:final}, it follows from \eqref{cond:induction:hypothesis:b} that $\sup_{t\geq0}\no{\bw^k(t)}{}\leq \gam \Rone$,
and $k=0,\dots, j$. Also observe that \eqref{cond:induction:hypothesis:a} implies
    \begin{align}\label{est:model:error:uniform}
         \max_{k=1,\dots,j}\sup_{t\geq 0}\Abs{\tau_{k}\bh^{k}(t)}\leq \sup_{t\geq 0}\Abs{\bh^{0}(t)}
    \end{align}

Now recall that $\bv_0^{(j)}=\bv^{(j-1)}(t_j;\bv_0^{(j-1)})$, so that $\bw_0^{(j)}=\bw^{(j-1)}(t_j;\bw_0^{(j-1)})$. From \eqref{cond:induction:hypothesis:b}, we see that $\no{\bw^{(j)}_0}{}=\no{\bw^{(j)}(t_j)}{}\leq \gam \Rone$. We may thus choose $t_{j+1}=t_1=t_*\geq T_2$, so that $t_{j+1}$ satisfies \eqref{def:transients}. We may then deduce from \eqref{est:sync:error:H1:j} that
    \begin{align}\label{est:sync:error:H1:jplus1}
        \sup_{t\geq0}\no{\tau_{j+1}\bw^{(j)}(t)}{}\leq \frac{\sqrt{2}}{(\nu\mu)^{1/2}}\left(\sup_{t\geq0}\Abs{\tau_j\bh^{(j)}(t)}{}\right).
    \end{align}
In particular $\no{\bw^{(j+1)}_0}{}\leq\gam\Rone$
by \eqref{est:model:error:uniform}, \eqref{cond:mu:initial:error}, and \eqref{cond:mu:final}. We now apply \eqref{est:model:error:L2} for $k=j$, followed by \eqref{est:sync:error:H1:jplus1}, \eqref{est:model:error:uniform}, to obtain 
    \begin{align}
        \sup_{t\geq0}\Abs{\tau_{j+1}\bh^{(j+1)}(t)}{}
        \leq \frac{C{(1+\norm{F}_{L,0})}}{(\nu\mu)^{1/2}}\left[\frac{M_0}{(\nu\mu)^{1/2}}+C_{j+1}(\al)\Rone\left(\ln\left(e+\frac{\til{\be}R_2}{\nu}\right)\right)^{1/2}\right]\sup_{t\geq0}\Abs{\tau_j\bh^{(j)}(t)}{},\notag
    \end{align}
for some sufficiently large absolute constant $C>0$. Since $\mu$ satisfies  \eqref{cond:mu:sieve:a}, we have that \eqref{cond:mu:final} holds. We then finally deduce
    \begin{align}\label{eq:induction:step}
        \sup_{t\geq 0}\Abs{\tau_{j+1}\bh^{(j+1)}(t)}\leq \frac{1}2\sup_{t\geq 0}\Abs{\tau_j\bh^{(j)}(t)},
    \end{align}
as desired.

It remains to prove the convergence of the state error to zero. We recall from \eqref{def:state:error:high:j} that $\bu^{(j)}-\si_j\bu=Q_N\bw^{(j)}$. It then follows from our choice of $\mu$, \eqref{est:sync:error:H1:jplus1},  \eqref{eq:induction:step} that
    \begin{align}\notag
         \no{\bu^{(j+1)}(t)-\si_{j+1}\bu(t)}{}\leq \frac{\sqrt{2}}{(\mu\nu)^{1/2}}\left(\sup_{t\geq0}\Abs{\tau_j\bh^{(j)}(t)}{}\right)\leq \frac{M_0}{2^{j-1/2}},
    \end{align}
for all $t\geq0$. Thus
    \begin{align}\notag
        \lim_{j\goesto\infty}\sup_{t\geq 0}\no{\bu^{(j+1)}(t)-\si_{j+1}\bu(t)}{}\leq\lim_{j\goesto\infty}\frac{M_0}{2^{j-1/2}}=0,
    \end{align}
which completes the proof.
\end{proof}


 \subsection{Nudging Algorithm}

As in the Nudging Algorithm for the transport-diffusion equation, we first state a well-posedness result for the Nudging Algorithm corresponding to the Navier-Stokes equations. Contrary to the previous cases where the wellposedness of the nudged equations can be achieved by simple modifications of established proofs in the literature, the nudged Navier-Stokes equation \ref{eq:NS_v} requires a more careful presentation. A sketch of the proof will be provided in \cref{subsec:proof_wellpos_nudg}.  
\begin{Thm}\label{thm:NS_nudg_wellpsd}
For any $T > 0$,  $N>0$ and $\mu_1,\mu_2 >  0$ the nudging system~\eqref{eq:NS_v} has a unique weak solution on the interval $[0, T]$, that is, there is $\bv \in C([0, T); \bbH) \cap L^2(0, T; \bbV)$ with $\frac{\dd \bv}{\dd t}  \in L^2(0, T; \bbV^*)$ as well as $\mbl \in C([0, T);P_N\bbH)$ with $\frac{\dd \mbl}{\dd t}  \in L^2(0, T;P_N\bbH)$ so that the Equation~\eqref{eq:NS_v} holds for a.a.~$t \in [0, T]$.
\end{Thm}

To properly state the corresponding convergence result for the Nudging Algorithm, we introduce a few additional. Suppose that $\Gamma \subset \bbH$ a set of quasi--finite rank $N_0$ and with order $(-1, -1)$ (see \cref{rmk:NS_nudg}).
  Let
  \beq{equ:6.1020}\nonumber
  \cG := \{\Gamma_N:N \geq N_0\},
  \eeq
  be a family where each $\Gamma_N \subset \bbH$ is of quasi-finite rank $N$ with Lipschitz enslaving map $F_N$ and order $(-1,-1)$ so that
    \begin{enumerate}
    \item $\Gamma_{N_0} = \Gamma$, 
    \item $\cG$ is increasing, that is $\Gamma_{M_1} \subset \Gamma_{M_2}$ for $M_1 \leq M_2$, 
    \item for sufficiently large $N \geq N_0$ we have
      \beq{equ:6.1030}
      16 C_* 
      \left\{
        1 + 
        \frac{U(\mbg)}{\kappa}
        (1 + l_N \norm{F_N}_{L,*} )
      \right\}
      <N,
      \eeq
    \end{enumerate}
    where $l_N := 2 ((\log N)^{1/2} + 1)$, and $C_*$ is a function of the constants appearing in~\cref{thm:trilinear}, and furthermore
\begin{equation}
      U(\mbg) := \sup_{t\in [0,T]} \|\bu(t)\|,\label{eq:U}
\end{equation}
for each $T>0$.
The construction presented in \cref{rmk:canonical} will result in a suitable $\cG$ but other choices are possible in which the Lipschitz constants in~\eqref{equ:6.1030} depends on $N$.
\begin{Thm}\label{thm:NS_nudg}
If $\mbg\in \Gam$, then for $N \geq N_0$ sufficiently large there exists a choice of $\mu_1,\mu_2>0$ such that the corresponding solution $(\bv, \mbl)$ of the nudging system~\eqref{eq:NS_v} satisfies
\begin{align}\notag
   \lim_{t\to \infty} \| \mbf(t) - \mbg(t)\|_{*} = 0 \quad \text{and}\quad
   \lim_{t \to \infty} |\bu(t) - \bv(t)| = 0,
\end{align}
where $\mbf(t)=\mbl(t) + F (\mbl(t))$.
\end{Thm}
\begin{Rmk}\label{rmk:NS_nudg}
A similar remark as in~\cref{rmk:NS_sieve} applies here:
\begin{enumerate}
   \item In order to deploy the methodology and use \cref{thm:NS_nudg} to reconstruct the forcing, knowledge of the observations $P_N \phi$ and of the set $\Gamma_{N}$ (which includes knowledge of the function $F_N$) is required, for the sufficiently large $N$ mentioned in the theorem.
\item
Again, there is no contradiction in assuming that $\Gamma_{N} \subset \bbH$ but of order $(-1, -1)$.
Although the $F_N$ have graphs in $\bbH$, namely $Q_{N} \mbg = F_N(P_{N} \mbg) \in \bbH$, the regularity of $F_N$ is only Lipschitz with values in $\bbH^{-1}$, that is $F_N \in \Lip(P_{N}\bbH^{-1}, Q_{N}\bbH^{-1})$.
Therefore, despite the true forcing $\mbg$ in the system \eqref{eq:nse:intro} as well as the approximations $\mbf(t)$ being elements of $\bbH$, the approximation is only in $\bbH^{-1}$.
\end{enumerate}
\end{Rmk}
\subsection{Proof of Theorem \ref{thm:NS_nudg}}
The proof proceeds along the same lines as the proof for \cref{thm:1.30} although several new elements are needed to deal with the fact that the system is now nonlinear.
First note that 
  \begin{align}\notag
  \begin{split}
      \norm{\mbf(t)-\mbg}_{*}=\norm{\mbl(t) + {F_N}(\mbl(t)) - \mbg}_{*} &\leq \norm{\mbl(t) - P_N \mbg}_{*} + \norm{{F_N}(\mbl(t)) - {F_N}(P_N \mbg)}_{*} \\
      &\leq (1 + \norm{F_N}_{L,*}) \norm{\mbl(t) - P_N \mbg}_{*},
      \end{split}
  \end{align} 
  so it is sufficient to consider the large-scale error dynamics for ${\mbe(t)} := P_N \mbg - \mbl(t)$, as well as error dynamics for $\bw(t):=\bu(t) - \bv(t)$, the latter of which we split into $P_N\bw(t) =: \mbp(t)$ and $Q_N\bw(t) =: \mbq(t)$.
Upon taking the difference between \eqref{eq:nse:intro} and the first equation in~\eqref{eq:NS_v}, then making use of the fact that  $\bdy_tP_N\mbg=0$ in the second equation in~\eqref{eq:NS_v}, we obtain the system for the state and model errors:
\begin{align}
    \partial_t \mbp + B_1 & =  \mbe - \mu_1 \mbp \label{eq:NS_vL}, \\
    \partial_t \mbq + B_2 & = -\kappa A \mbq + F_N(P_N \mbg) - F_N(P_N \mbg - \mbe) \label{eq:NS_vH},\\
    \partial_t \mbe & = -\mu_2 \mbp, \label{eq:NS_f}
\end{align}
where
\beq{eq:NS_Bterms}\begin{split}
B_1 & := P_N B(\bu, \bu) - P_N B(P_N \bu + Q_N \bv, P_N \bu + Q_N \bv)\\
B_2 & := Q_N B(\bu, \bu) - Q_N B(P_N \bu + Q_N \bv, \bv).
\end{split}\eeq
Consider the following change of variables and parameters: 
    \begin{align}\notag
        \mbr := \frac{\mbe}{\lambda_2} - \mbp.\quad \lambda_1 + \lambda_2 = \mu_1,\quad\lambda_1 \lambda_2 = \mu_2.
    \end{align}
We note that by an appropriate choice of $\mu_1, \mu_2$, the roots $\lambda_{1}, \lam_2$ can be set to any desired nonnegative value and we can assume without loss of generality that $\lambda_1 \geq \lambda_2$.
We may then rewrite \eqref{eq:NS_vL}--\eqref{eq:NS_f} as
\begin{align}
    \partial_t \mbp + \lambda_1 \mbp & = \lambda_2 \mbr - B_1 \label{eq:NS_vLx}, \\
    \partial_t \mbq + \kappa A \mbq & = F_N(P_N \mbg) - F_N(P_N \mbg - \lambda_2 (\mbr + \mbp)) - B_2
    \label{eq:NS_vHx},\\
    \partial_t \mbr + \lambda_2 \mbr & = B_1  . \label{eq:NS_fx}
\end{align}
To analyse these equations, first we simplify the bilinear terms $B_1, B_2$ in~\eqref{eq:NS_Bterms}.
We use the identity $B(a, x) - B(b, y) = B(a-b, x) + B(a, x - y) - B(a-b, x-y)$ to obtain
\beq{eq:NS_Bterms2}\nonumber\begin{split}
B_1 & = P_N \left(B(\bu, \mbq) +  B(\mbq, \bu) - B(\mbq, \mbq)\right),\\
B_2 & = Q_N \left(B(\bu, \mbp + \mbq) + B(\mbq, u) - B(\mbq, \mbp + \mbq)\right).
\end{split}\eeq
Therefore
\beq{eq:NS_Bterms3}\begin{split}
(B_1, \mbp) & = b(\bu, \mbq, \mbp) +  b(\mbq, \bu, \mbp) - b(\mbq, \mbq, \mbp),\\
(B_1, A^{-1}\mbr) & = b(\bu, \mbq, A^{-1}\mbr) +  b(\mbq, \bu, A^{-1}\mbr) - b(\mbq, \mbq, A^{-1}\mbr),\\
(B_2, \mbq) & = b(\bu, \mbp, \mbq) + b(\mbq, \bu, \mbq) - b(\mbq, \mbp, \mbq).
\end{split}\eeq
Next we take the $\bbH$ scalar product of \eqref{eq:NS_vLx} \eqref{eq:NS_vHx} with $\mbp$ and $\mbq$, respectively, and the $\bbV^*$ inner product of \eqref{eq:NS_fx} with $\mbr$.
We will furthermore add the equations for $\mbp$ and $\mbq$ and note the cancellation of common terms in the first and third line of Equation~\eqref{eq:NS_Bterms3}.
We obtain
\begin{align}
\begin{split}
    & \frac{1}{2}\frac{d}{dt} \abs{\mbp}^2 
    +  \frac{1}{2}\frac{d}{dt}\abs{\mbq}^2
    + \lambda_1 \abs{\mbp}^2 
    + \kappa \norm{\mbq}^2  \\
    & \quad = \lambda_2 (\mbr, \mbp) 
    + \big( F_N(P_N \mbg) - F_N(P_N \mbg - \lambda_2 (\mbr + \mbp)), \mbq \big)
    - b(\mbq, \bu, \mbp)
    - b(\mbq, \bu, \mbq), \end{split}\label{eq:NS_w}\\
    %
    & \frac{1}{2} \frac{d}{dt} \norm{\mbr}_*^2 
    + \lambda_2 \norm{\mbr}^2_* 
    = b(\bu, \mbq, A^{-1}\mbr) +  b(\mbq, \bu, A^{-1}\mbr) - b(\mbq, \mbq, A^{-1}\mbr). \label{eq:NS_r}
\end{align}
We next obtain a few estimates regarding the right hand-side of~\eqref{eq:NS_vLx}-\eqref{eq:NS_fx}.
Firstly, we have 
\beq{eq:4.99}
\lambda_2 (\mbr, \mbp) 
\leq \lambda_2 \norm{\mbr}_* \norm{\mbp}.
\leq \frac{ \lambda_2^2 \norm{F_N}_{L,*}^2}{\kappa} \norm{\mbr}^2_* + \frac{\kappa }{4 \norm{F_N}_{L,*}^2}\norm{\mbp}^2,
\eeq
by Young's inequality (and weight factors chosen with hindsight).
Using the properties of the map $F_N$ and subsequently Young's inequality, we find
\beq{equ:4.100}
\begin{split}
(F_N(P_N \mbg) - F_N(P_N \mbg - \mbe),\mbq)
& \leq  \norm{F_N(P_N \mbg) - F_N(P_N \mbg - \mbe))}_{*}\norm{\mbq}\\
& \leq \norm{F_N}_{L,*}  \norm{\mbe}_* \norm{\mbq}\\
& \leq \frac{\kappa}{2} \norm{\mbq}^2 + \frac{\norm{F_N}_{L,*}^2}{2\kappa} \norm{\mbe}_{*}^2\\
& \leq \frac{\kappa}{2} \norm{\mbq}^2 + \frac{\norm{F_N}_{L,*}^2 \lambda^2_2}{\kappa} (\norm{\mbr}_{*}^2 + \norm{\mbp}_{*}^2).
\end{split}
\eeq
To treat the trilinear terms in Equation~\eqref{eq:NS_r}, we use the estimate~\eqref{est:trilinear:BrezisWainger} in the equations \eqref{eq:b1} and \eqref{eq:b2} below and the estimate~\eqref{est:trilinear1} in \eqref{eq:b3} and \eqref{eq:b4}
\begin{align}
   \abs{b(\mbq,\bu,A^{-1}\mbr)} 
   & \leq C\frac{(\log N)^{1/2}}{N} \norm{\bu} \norm{\mbq} \norm{\mbr}_*,
    \label{eq:b1}\\
\abs{b(P_N\bu,\mbq,A^{-1}\mbr)}
& \leq C\frac{(\log N)^{1/2}}{N} \norm{\bu} \norm{\mbq} \norm{\mbr}_*, \label{eq:b2}
    \\
   \abs{b(Q_N\bu,\mbq,A^{-1}\mbr)}
   & \leq \frac{C}{N}\norm{\bu}\norm{\mbq}\norm{\mbr}_*, \label{eq:b3}
    \\
   \abs{b(\mbq, \mbq ,A^{-1}\mbr)}
   & \leq \frac{C}{N} \norm{\mbq}^2 \norm{\mbr}_*, \label{eq:b4}
\end{align}
which gives
\begin{align}
  & b(\bu, \mbq, A^{-1}\mbr) +  b(\mbq, \bu, A^{-1}\mbr) - b(\mbq, \mbq, A^{-1}\mbr) \nonumber\\
  & \leq \frac{C l_N}{N} \norm{\mbq} \norm{\mbr}_* U(\mbg) + \frac{C}{N} \norm{\mbr}_*\norm{\mbq}^2\label{equ:4.125.a}\\
& \leq \frac{\kappa}{8\alpha} \norm{\mbq}^2
  + \frac{2 \alpha C^2 U(\mbg)^2 l_N^2}{N^2\kappa} \norm{\mbr}_*^2
  + \frac{C}{N} \norm{\mbr}_* \norm{\mbq}^2, \label{equ:4.125}
\end{align}
where we have used Young's inequality on the first term in~\eqref{equ:4.125.a}; the factor $\alpha$ will be set later, and we recall that $l_N = 2 ((\log N)^{1/2} + 1)$.
For the nonlinear terms in Equation~\eqref{eq:NS_w} we use the estimates~\eqref{est:trilinear1} and \eqref{eq:U} and find
\beq{equ:4.110}\begin{split}
 b(\mbq, \bu, \mbp) + b(\mbq, \bu, \mbq)
& \leq \frac{C (\log N)^{1/2}}{N}  \norm{\bu} \norm{\mbq} \norm{\mbp}
+ \frac{C}{N} \norm{\bu} \norm{\mbq}^2 \\
& \leq \frac{\kappa}{4}   \norm{\mbq}^2 
+  \frac{C^2 U(\mbg)^2 l_N^2}{N^2\kappa} \norm{\mbp}^2,
\end{split}\eeq
where we again applied Young's inequality to the first term in the second line, while for the second term we use that by condition~\eqref{equ:6.1030} we have $\frac{C}{N} U(\mbg) \leq \frac{\kappa}{8}$ if we let $C_* \geq C$.
We now apply these  estimates~\eqref{eq:4.99}, \eqref{equ:4.100}, \eqref{equ:4.125}, \eqref{equ:4.110} to the right hand sides of Equations~\eqref{eq:NS_w}, \eqref{eq:NS_r} and find
\begin{align}
    & \frac{1}{2}\frac{d}{dt}( \abs{\mbp}^2 
    + \abs{\mbq}^2)
    + \lambda_1 \abs{\mbp}^2 
    + \frac{\kappa}{2} \norm{\mbq}^2  \nonumber \\
    & \quad \leq \lambda_2 \norm{\mbr}_* \norm{\mbp} 
     + \frac{\norm{F_N}_{L,*}^2 \lambda^2_2}{\kappa} (\norm{\mbr}_{*}^2 + \norm{\mbp}_{*}^2)
+ \frac{\kappa}{4}   \norm{\mbq}^2 
+  \frac{C^2 U(\mbg)^2 l_N^2}{N^2\kappa} \norm{\mbp}^2,
\label{eq:NS_w2}\\
    %
    & \frac{1}{2}\frac{d}{dt}\norm{\mbr}_*^2 
    + \lambda_2 \norm{\mbr}^2_* 
\leq \frac{\kappa}{8\alpha} \norm{\mbq}^2
  + \frac{2 \alpha C^2 U(\mbg)^2 l_N^2}{N^2\kappa} \norm{\mbr}_*^2
  + \frac{C}{N} \norm{\mbr}_* \norm{\mbq}^2.
\label{eq:NS_r2}
\end{align}
Next we use $\norm{\mbp} \leq N \abs{\mbp}$ and $\norm{\mbp}_* \leq \abs{\mbp}$ in all instances except in the very last (trilinear) term in Equation~\eqref{eq:NS_r2}.
To deal with that term, let
\beq{equ:4.140}\nonumber
t_* := \inf \left\lbrace t \geq 0: \norm{r(t)}_* > \frac{\kappa}{\alpha} \right\rbrace.
\eeq
Observe that $t_*>0$ if $\alpha$ is sufficiently small.
Upon restricting $t$ to the interval $[0,t_*]$, we have $\norm{r(t)}_* \leq \frac{\kappa}{\alpha}$ which we use to estimate that term:
\begin{align}
  \frac{C_*}{N} \norm{\mbr}_* \norm{\mbq}^2 
  \leq \frac{C_* \kappa}{\alpha N} \norm{\mbq}^2 
  \leq \frac{\kappa}{\alpha 16} \norm{\mbq}^2, \notag
\end{align}
where we also used that $N \geq 16 C$ by condition~\eqref{equ:6.1030}.
Upon multiplying \eqref{eq:NS_r2} by $\alpha$ and adding the result to \eqref{eq:NS_w2}, we obtain
\beq{equ:4.150}
\begin{split}
    & \frac{1}{2} \frac{d}{dt} \left( 
     \abs{\mbp}^2 + \abs{\mbq}^2 + \alpha\norm{\mbr}^2_* \right)
    + \frac{\kappa}{16} \norm{\mbq}^2 \\
    & + \left\{ 
        \lambda_1  
        - \frac{\kappa N^2}{4 \norm{F_N}_{L,*}^2} 
        - \frac{\norm{F_N}_{L,*}^2 + C_*^2 U(\mbg)^2 l_N^2}{\kappa} 
    \right\} \abs{\mbp}^2 \\
    & + \left\{
    \lambda_2
    - 2\frac{\norm{F_N}_{L,*}^2 \lambda_2^2 }{\alpha \kappa} 
    - 2\frac{\alpha C_*^2 U(\mbg)^2 l_N^2}{N^2\kappa}
    \right\} \alpha \norm{\mbr}_*^2\\
    & \leq 0.
\end{split}\eeq
for some non-dimensional constant $C_* \geq C > 0$.
The coefficient of $\alpha \norm{\mbr}^2_*$ in \eqref{equ:4.150}, regarded as a function of $\lambda_2$, is a second order polynomial with discriminant
\beq{equ:4.160}\nonumber
\delta_N := 1 - 16 \left(\frac{C_* \norm{F_N}_{L,*} U(\mbg)l_N}{N\kappa}\right)^2
\eeq
Due to our condition~\eqref{equ:6.1030}, we see that $\delta_N>0$ for sufficiently large $N$ and any $\alpha>0$.
Hence if we pick such an $N$, there exists  $\lambda_2$ so that the coefficient is positive.
With $N$ and $\lambda_2$ so chosen, we next pick $\lambda_1 > \lambda_2$ so that the coefficient of $\abs{\mbp}^2$ is also positive, and therefore
\beq{equ:4.170}\nonumber
\frac{d}{dt} \big( \abs{\mbp}^2 + \abs{\mbq}^2 + \alpha \norm{\mbr}_*^2 \big) 
\leq -\epsilon \big( \abs{\mbp}^2 + \abs{\mbq}^2 + \alpha \norm{\mbr}_*^2 \big), 
\eeq
for some $\epsilon > 0$.
After integrating over $[0,t]$, for $t\leq t^*$, we obtain
\beq{equ:4.180}
\abs{\mbp(t)}^2 + \abs{\mbq(t)}^2 + \alpha \norm{\mbr(t)}_*^2 \leq e^{-\epsilon t} \big( \abs{\mbp}^2 + \abs{\mbq}^2 + \alpha \norm{\mbr}_*^2 \big)\Bigg{|}_{t = 0}.
\eeq
This implies 
\begin{align}\label{equ:4.190}
\alpha \norm{\mbr(t)}_*^2
\leq
\left( \abs{\mbp(t)}^2 + \abs{\mbq(t)}^2 + \alpha \norm{\mbr(t)}_*^2 \right)\Bigg{|}_{t = 0},
\end{align}
which holds for all $\alpha > 0$.
Now, upon choosing $\alpha$ such that
\beq{equ:4.191}
    \left( \abs{\mbp(t)}^2 + \abs{\mbq(t)}^2 + \alpha \norm{\mbr(t)}^2_* \right)\Bigg{|}_{t = 0}\leq \frac{\kappa^2}{2\alpha},
\eeq
we may deduce from \eqref{equ:4.190} that
\beq{equ:4.200}
\sup_{0\leq t\leq t_*} \norm{\mbr(t)}^2_* \leq \frac{\kappa^2}{2 \alpha^2}.
\eeq
On the other hand, if we assume that $t_*$ is finite, then due to the definition of $t_*$ and the continuity of $\norm{\mbr}_*$ as a function of $t$, we must have
\beq{equ:4.205}\nonumber
\norm{\mbr(t_*)}^2_* = \frac{\kappa^2}{\alpha^2}.
\eeq
However, this contradicts~\eqref{equ:4.200}. We therefore deduce that $t_* = \infty$.
In particular,~\eqref{equ:4.180} is valid for all $t \geq 0$, and we find
\beq{equ:4.210}\nonumber
\abs{\mbp(t)}^2 + \abs{\mbq(t)}^2 + \alpha \norm{\mbr(t)}^2_*  \to 0,\quad\text{as}\ t\goesto\infty,
\eeq
as desired, finishing the proof.
\begin{Rmk}\label{rmk:lam} 
\begin{enumerate}
    \item
    In the above proof, if $C(\lam_2)$ is the coefficient of $\alpha \norm{\mbr}^2_*$ in the rightmost term of~\eqref{equ:4.150}, the decay rate $\eps$ can be chosen as
    \begin{align}\notag
        \eps=\min\left\{N^2\frac{\kappa}{8},2C(\lam_2)\right\},
    \end{align}
    provided we set $\lambda_1$ sufficiently large.
    In particular, we may take $\lam_2 = \frac{\kappa \alpha}{4\norm{F_N}_*^2}$ as then $C(\lam_2) = \max_{\lam} C(\lam) = \de_N \frac{\kappa \alpha}{8\norm{F_N}_*^2}$.
    This demonstrates how the Lipschitz constants of the enslaving maps $F_N$ influence the decay rate of the error.
    \item The decay rate depends on the parameter $\alpha$ and therefore (through Equation~\eqref{equ:4.191}) on the initial condition. 
    Similarly, appropriate choices of $\lambda_1, \lambda_2$ (but not of $N$) will depend on $\alpha$ and thus the initial condition.
    This is a truly ``nonlinear'' effect and due to the rightmost term in Equation~\eqref{equ:4.125}, which has its origin in the trilinear interaction term $b(\mbq, \mbq, A^{-1}\mbr)$  between the high modes of the nudging error and the parameter mismatch.
\end{enumerate}  
\end{Rmk}
\subsection{Proof of Theorem \ref{thm:NS_nudg_wellpsd}}\label{subsec:proof_wellpos_nudg}
\newcommand{\Pnv}{\mbr}
\newcommand{\Qnv}{\mathbf{s}}
\newcommand{\Pnu}{P_N\bu}
\newcommand{\bV}{\mathbf{V}}
Let $\Pnv=P_N\bv$, $\Qnv=Q_N\bv$ and write
    \[
        B_1=P_NB(\Pnu+\Qnv,\Pnu+\Qnv),\quad B_2=Q_NB(\Pnu+\Qnv,\Pnv+\Qnv).
    \]
Then Equation~\eqref{eq:NS_v} can be written as 
    \begin{align}
        \bdy_t\Pnv-\nu \Delta\Pnu+B_1&=\mbl + \mu_1(\Pnu - \Pnv), \label{eq:NS_v_split1}\\
        \bdy_t\Qnv-\nu \Delta\Qnv+B_2&=F_N(\mbl), \label{eq:NS_v_split2}\\
        \bdy_t\mbl&=\mu_2(\Pnu-\Pnv).\label{eq:NS_v_split3}
    \end{align}
To show well--posedness of these equations we again use the Galerkin method.
Let $M$ be large enough so that $M > N$ and therefore $P_M H \supset P_N H$.
The solutions $\Pnv_M,\Qnv_M,\mbl_M$ to the order~$M$ Galerkin truncation of equations~\eqref{eq:NS_v_split1}--\eqref{eq:NS_v_split3} respectively exist locally in time since the nonlinearity is locally Lipschitz.
Global existence follows from the apriori estimate~\eqref{est:apriori:nudging4} below which we will prove presently. 
Since that estimate is uniform in~$M$, there exist subsequences as $M \to \infty$ which converge to a solution of equations~\eqref{eq:NS_v_split1}--\eqref{eq:NS_v_split3} satisfying the requirements of the theorem.
As before, we will only present a formal derivation of the apriori estimate~\eqref{est:apriori:nudging4} below and refer to~\citet{RobinsonBook,TemamBook1997} for further details see.
Multiplying equations~\eqref{eq:NS_v_split1}--\eqref{eq:NS_v_split3} by $\Pnv, \Qnv$ and~$\mbl$, respectively, and integrating gives the energy estimate
\begin{align}
       \frac{1}2\frac{d}{dt}|\Pnv|^2+\nu\|\Pnu\|^2+\mu_1|\Pnv|^2
       & =-b(\Pnu,\Pnu,\Pnv)-b(\Qnv,\Pnu,\Pnv)-b(\Pnu,\Qnv,\Pnv)-b(\Qnv,\Qnv,\Pnv)\notag \\
       & \quad +\lb\mbl,\Pnv\rb+\mu_1\lb \Pnu,\Pnv\rb\notag \\
        \frac{1}2\frac{d}{dt}|\Qnv|^2+\nu\|\Qnv\|^2&=-b(\Pnu, \Pnv, \Qnv)+b(\Qnv, \Qnv, \Pnv)+\lb F_N(\mbl),\Qnv\rb\notag\\
        \frac{1}2\frac{d}{dt}|\mbl|^2&=\mu_2\lb \Pnu,\mbl\rb - \mu_2\lb\Pnv,\mbl\rb.\notag
    \end{align}
We let $\bV=\mu_2^{1/2}\bv$ and add these equations.
Noting a delicate cancelling of two trilinear terms between the equations for $\Qnv$ and $\Pnv$ we obtain
    \begin{align}
        &\frac{1}2\frac{d}{dt}\left(|\bV|^2+|\mbl|^2\right)+\nu\|Q_N\bV\|^2+\mu_1|P_N\bV|^2\notag
        \\
        &=-\mu_2^{1/2}b(\Pnu,\Pnu,P_N\bV)-b(Q_N\bV,\Pnu,P_N\bV)+\mu_1\mu_2^{1/2}(\Pnu,\bV)\notag
        \\
        &\quad +\mu_2^{1/2}(F_N(\mbl),\bV)+\mu_2(\Pnu,\mbl).\notag
    \end{align}
We emphasize that the equation is at most linear in the energy $E=\frac{1}2\left(|\bV|^2+|\mbl|^2\right)$. 
We will therefore be able to demonstrate the following apriori bound which is global~in~time and locally uniform in time:
    \begin{align}\label{est:apriori:nudging4}
        \sup_{0\leq t\leq T}\left(|\bV(t)|^2+|\mbl(t)|^2\right)+\nu_N\int_0^T\|\bV(s)\|^2ds\leq C(\mu_1,\mu_2,N,T),
    \end{align}
where $\nu_N := \min\{\nu, \frac{\mu_1}{N}\}$ (we note that $\nu_N > 0$).
We now present the details: We have
    \begin{align}
        \mu_2^{1/2}|b(\Pnu,\Pnu,P_N\bV)|&\leq \mu_2^{1/2}|\Pnu\cdotp\nabla\Pnu||P_N\bV|\notag\\
        &\leq \mu_2^{1/2}|\Pnu|_{\infty}\|\Pnu\||P_N\bV|\notag
        \\
        &\leq C\frac{\mu_2}{\mu_1}(1+\log N)\|\Pnu\|^4+\frac{\mu_1}{4}|P_N\bV|^2,\notag
    \end{align}
and
    \begin{align}
        |b(Q_N\bV,\Pnu,P_N\bV)|&\leq c_L\|Q_N\bV\|^{1/2}|Q_N\bV|^{1/2}\|\Pnu\|\|P_N\bV\|^{1/2}|P_N\bV|^{1/2}\notag
        \\
        &\leq c_L\|Q_N \bV\|\|\Pnu\||P_N\bV|\notag
        \\
        & \leq \frac{\nu}{4}\|Q_N \bV\|^2+\frac{c_L^2}{\nu}\|\Pnu\|^2|P_N\bV|^2.\notag
    \end{align}
Also
    \begin{align}
        \mu_1\mu_2^{1/2}|(\Pnu,\bV)|
        & \leq \mu_1\mu_2|\Pnu|^2+\frac{\mu_1}{4}|P_N\bV|^2\notag\\
        \mu_2^{1/2}|(F_N(\mbl),\bV)|
        & \leq C_F^2 \frac{\mu_2}{\nu} |\mbl|^2 + \frac{\nu}{4} |Q_N \bV|^2\notag\\
        \mu_2(\Pnu,\mbl)&\leq\frac{\mu_1\mu_2}2|\Pnu|^2+\frac{\mu_2}{2\mu_1}|\mbl|^2.\notag
    \end{align}
Suppose $R>0$ is such that
    \begin{align}
        \sup_{t\geq0}\|\Pnu(t)\|\leq \nu R.\notag
    \end{align}
Then
    \begin{align}
        &\frac{d}{dt}\left(|\bV|^2+|\mbl|^2\right)
        +\nu\|Q_N \bV\|^2+\mu_1|P_N \bV|^2\notag \\
        &\leq  \left(c_L^2\nu R^2 + C_F^2\frac{\mu_2}{\nu} + \frac{\mu_2}{2 \mu_1}\right) 
        \left(|P_N \bV|^2+|\mbl|^2\right)
        + C\frac{\mu_2}{\mu_1}(1+\log N)\nu^4R^4+\frac{3}{2}\mu_1\mu_2\nu^2R^2.\notag
    \end{align}
Finally we estimate $\nu \| Q_N \bV \| + \mu_1 |P_N\bV| \geq \nu_N \| \bV \|$, and an application of Gronwall's inequality yields \eqref{est:apriori:nudging4}.
As mentioned, the derivation of the apriori bound relies on the cancelling of two trilinear terms between the equations for $\Qnv$ and $\Pnv$ and relies crucially on the specific form of the equations.
The main problem in analysing such systems seems is that two different stabilising mechanisms have to be exploited at the same time, namely the diffusion for the high modes of $\bv$, and the nudging term for the low modes of $\bv$ and the parameter $\mbl$.
The latter however is not symmetric, and this requires the introduction of a different scalar product (or energy--see also~\citet{Broecker02II} for a discussion of this problem in the context of synchronisation).
Yet as a result, the energy flow between low and high modes due to nonlinear interactions no longer cancels, and this problem had to be addressed by carefully chosing the precise form of the nonlinear term. 
\section{Conclusions}\label{sect:conclusion}
We will conclude the paper with a comparison between the Nudging and the Sieve Algorithm.
The algorithms are complementary in that they provide different results under different assumptions. 
The main conceptual difference is that while the Nudging Algorithm assumes the forcing to be constant in time, the Sieve Algorithm allows to reconstruct time dependent forcings (albeit up to a transient which has to be discarded). 
The price to pay is that multiple passes over the data are required while the Nudging Algorithm processes the observational data only once and in sequential fashion.
Also from a mathematical point of view, the Nudging Algorithm requires weaker assumptions than the Sieve Algorithm, but will also give weaker results.
We will now provide a more detailed comparison, which can also be found summarised in Table~\ref{tbl:comparison}.
\begin{table}
\begin{tabular}{p{0.15\textwidth}p{0.26\textwidth}p{0.27\textwidth}}
& 
\centerline{Transport--Diffusion}
& \centerline{2D--Navier-Stokes}\\
Sieve\newline (weak version)
& $d = 2, 3 $\newline
$ \Gamma \subset V^* $ of
order $(-1, -1)$ \newline
$ F_{N, *} \sim o(N^{1-\eps}) $ \newline
$ \phi - \psi \to 0 \text{~in ~} H $ \newline
$ f - g \to 0 \text{~in ~} V^*$ \newline
$\mathbf{v}$ \textit{can} be weak solution
&  \\
\rule{0cm}{4ex}Sieve\newline (strong version)
& $d = 2, 3 $\newline
$ \Gamma \subset H $ of
order $(0, 0)$ \newline
$ F_{N, 0} \sim o(N^{1-\epsilon})$ \newline
$ \phi - \psi \to 0 \text{~in ~} V $ \newline
$ f - g \to 0 \text{~in ~} H$ \newline
$\mathbf{v}$ cannot be weak solution
& $d = 2 $\newline
$ \Gamma \subset \bbV $ of
order $(0, 0)$~~(sic!)\newline
$ F_{N, 0} \sim o(N) $ \newline
$ \bu - \bv \to 0 \text{~in ~} \bbV $ \newline
$ \mbf - \mbg \to 0 \text{~in ~} \bbH$\newline
$\mathbf{u}$ must be a strong solution
\\[2ex]
Nudging
& $d = 2, 3 $\newline
$ \Gamma \subset V^*$ of
order $(-1, -1)$  \newline
$ F_{N, *} \sim o(\frac{N}{(\log N)^{1/2}}) $ \newline
$ \phi - \psi \to 0 \text{~in ~} H $ \newline
$ f - g \to 0 \text{~in ~} V^*$\newline
$\mathbf{v}$  can be {\em LH-weak} solution
& $d = 2$\newline
$ \Gamma \subset \bbH $ of
order $(-1, -1)$~~(sic!)  \newline
$ F_{N, *} \sim o(\frac{N}{(\log N)^{1/2}}) $ \newline
$ \bu - \bv \to 0 \text{~in ~} \bbH $ \newline
$ \mbf - \mbg \to 0 \text{~in ~} \bbV^{*}$\newline
$\mathbf{u}$ must be a strong solution
\end{tabular}\\[1ex]
\caption{\label{tbl:comparison} A summary of the main assumptions ensuring the convergence of the discussed algorithms. 
The ``little--$o$ notation'' like $F_{N, *} \sim o(\phi(N))$ means that $\frac{F_{N, *}}{\phi(N)} \to 0$ as $N \to \infty$.
The $\epsilon$ which appears in the conditions for the two Sieve Algorithms relates to the regularity of the advecting vector field (see text).
We stress that the listed assumptions may be incomplete and simplified, and the reader is referred to the main text for details.}
\end{table}
\subsection{Comparison for the Transport-Diffusion Equation}\label{sec:comparison_TD}
In the context of the Transport-Diffusion Equation, the Sieve Algorithm has two versions called the ``weak'' and the ``strong'' version.
The difference is in the stated convergence for the tracer, which is in $H$ for the weak version and $V$ for the strong version, and for the forcing, which is in $V^*$ for the weak version and $H$ for the strong version.
The assumptions of the weak and the strong version mostly differ regarding the required regularity of the advecting vector field $\bv$ and the set $\Gamma$ of quasi--finite rank.
While the weak version only requires $\Gamma \subset V^*$, the strong version needs forcings in $\Gamma \subset H$.
The enslaving maps $F$ have to be Lipschitz from $V^* \mapsto V^*$ resp.~$H \mapsto H$, with Lipschitz coefficient potentially growing but slower than $N^{1-\eps}$, where $\eps$ is related to integrability properties of the advecting vector field~$\bv$.
More specifically, both Sieve Algorithms require that $\abs{\bv}_{d/\eps}$ be bounded in time, for some $\eps \in [0,d/2]$, where $d$ is the dimension. 
(The strong Sieve Algorithms imposes additional assumptions on $\abs{\nabla \bv}$).
A particularly notable case is the endpoint $\eps=d/2$, which in dimension $d=2$ corresponds to velocity fields which have bounded kinetic energy uniformly for all time.
This therefore includes weak solutions for the Navier-Stokes equations as advecting velocity fields.
However, the maximum energy of the velocity field has to be sufficiently small in relation to the viscosity, or in other words the advecting velocity field has to have a sufficiently small P\'eclet~number, in order that the conditions imposed on $\mu, N$ can actually be satisfied.
This necessary condition on the smallness of the velocity field makes no explicit reference to the observational cut--off scale $N$, meaning that the class of admissible velocity fields cannot be increased by simply increasing $N$. 
This is in contrast to the Nudging Algorithm, which is able to accommodate \textit{Leray-Hopf} weak solutions of the Navier-Stokes equations with potentially very high P\'eclet~number.
This is because the condition guaranteeing convergence \eqref{equ:5.1030} allows to compensate for velocity fields with large P\'eclet~number by choosing a larger cutoff scale~$N$, if at the same time the enslaving maps $F_N$ can be chosen so as to have a Lipschitz constant decreasing to zero with $N$.
\subsection{Comparison for 2D~Navier Stokes Equations}\label{sec:comparison_NS}
Only one version for the Sieve Algorithm was analysed for the 2D--Navier~Stokes equation.
It would be desirable to have a version of the Sieve Algorithm applied to Navier--Stokes where the convergence of the velocity field $\bu - \bv \to 0 $ is in $\bbH$ rather than $\bbV$, while the convergence of the forcing term $\mbf - \mbg \to 0 $ is in $\bbV^*$ rather than $\bbH$, as is the case in the Nudging Algorithm applied to Navier--Stokes.
We were unable to prove such a theorem under essentially weaker conditions than in Theorem~\ref{thm:main:sieve:a}.
Comparing the Sieve Algorithm with the Nudging Algorithm we see that, as a rule of thumb, the Sieve Algorithm requires all function spaces to have ``one order'' of additional regularity.
Specifically, while the Nudging Algorithm requires $\Gamma \subset \bbH$ (but with order $(-1, -1)$, see \cref{rmk:NS_nudg}), the Sieve requires forcings in $\Gamma \subset \bbV$ (but with order $(0, 0)$, see \cref{rmk:NS_sieve}).
For both algorithms, the Lipschitz coefficients of the enslaving maps may potentially grow like $o(N)$ but with an additional logarithmic correction for the Nudging Algorithm.
Also for both algorithms, the underlying true solution $\bu$ must be a strong solution of the 2D--Navier~Stokes equation.
In terms of results, we obtain that the Sieve provides convergence in spaces that have, once again, ``one order'' of additional regularity when compared to the Nudging Algorithm. 
The convergence for the velocity field $\bv$ is in $\bbH$ for the Nudging Algorithm and in $\bbV$ for the Sieve Algorithm, while the convergence for the forcing $\mbg$ is in $\bbV^{*}$ for the Nudging Algorithm and in $\bbH$ for the Sieve Algorithm.
Finally we remark that we faced challenges in the proof of well--posedness of the system~\eqref{eq:NS_v} for the Nudging algorithm. 
Various versions of these equations are conceivable, differing in how the observed low modes $P_N \bu$ are re--inserted into the nonlinear term.
We were unable to prove the well--posedness result of Theorem~\ref{thm:NS_nudg_wellpsd} for other versions of the equations except those stated in~\eqref{eq:NS_v}, as the analysis relies on a delicate cancellation of trilinear terms between the energy estimates of low and high modes.
\subsection{Future work}
Two particular interesting directions to study further are: 1)~computational experiments probing the efficacy of the above algorithms and the sharpness of the theoretical limitations, and 2)~addressing the issue of higher-order convergence in the Nudging Algorithm for both the Transport-Diffusion and Navier-Stokes Equations. 
Regarding the first issue, it would be very interesting to investigate the stability of the proposed algorithms with respect to misspecification of the models, numerical approximations, and observations corrupted by noise.
Furthermore, the sharpness of the scaling between the Lipschitz constant of the enslaving map and the observational cut-off~$N$ could be investigated numerically as well.
Regarding the second issue, we in fact anticipate few problems with the transport--diffusion equations but we have not explored this further, not least for the sake of conciseness.
For the Navier--Stokes equations however, the problem will certainly be considerably more complicated. 
To begin with, one would need to prove a well--posedness result for the nudging system analogous to \cref{thm:NS_nudg_wellpsd} but for higher--order regularity. 
That theorem proved difficult already in the  present situation and required a strategic insertion of the observations into the nonlinearity.
We will explore this in future work.
\subsection*{Acknowledgments} The authors would like to thank the Isaac Newton Institute for Mathematical Sciences and their first satellite programme ``Geophysical fluid dynamics; from mathematical theory to operational prediction" in September 2022, where this collaboration was initiated.  V.R.M.\ was in part supported by the National Science Foundation through DMS 2213363 and DMS 2206491, as well as the Dolciani Halloran Foundation. G.C.\ and T.K.\ are members of INdAM-GNAMPA.
G.C., T.K., and V.R.M.\ acknowledge financial support and hospitality from the Centre for the Mathematics of Planet Earth at the University of Reading while this work was carried out. T.K.~gratefully acknowledge the partial support by 
 PRIN2022-PNRR--Project N.P20225SP98 ``Some mathematical approaches to climate change and its impacts.''
\appendix

\section{Background results}\label{sect:appendix}

In this section, we supply the relevant details for the well-posedness result stated in \cref{thm:td:exist}. 
The solutions asserted in \cref{thm:td:exist} are strongly continuous in time, as $\phi \in L^2(0,T;V)$ and $\frac{d}{dt} \phi \in L^2(0,T; V^*)$ implies that there exists a version of $C([0,T]; H)$, see Theorem 7.2 in \cite{RobinsonBook}.  We note that this fails for 3D Navier-Stokes because $\frac{d}{dt} \phi \in L^{4/3}(0,T; V^*)\not \subset  L^2(0,T; V^*)$. 
Note that below, we proceed without assuming that $\bv$ is divergence-free.
 \begin{proof}[Proof of \cref{thm:td:exist}]
    Let us first do the $L^2$ estimate
     \[
     \frac{1}{2}\frac{d}{dt} |\phi |^2 + \kappa \| \phi\|^2 = (g,\phi) - (\bv\cdotp\nabla \phi , \phi).
     \]
     By \eqref{est:trilinear:Holder} with $p\geq2=q$ and $r=\frac{2p}{p-2}$, we have
    \[
        (\bv\cdotp \nabla \phi , \phi) \leq |\bv |_p \| \phi\| | \phi |_r 
    \]
    We then interpolate with $p >d$ and $r< \infty$, that is $p >2$ and $-\frac{d}{r} \leq  \frac{d}{p} - \frac{d}{2} = \frac{d}{p} (1 - \frac{d}{2} ) + (1 - \frac{d}{p}) ( -\frac{d}{2})$, to obtain
     \[
    | (\bv\cdotp\nabla \phi , \phi) | \leq |\bv|_p  \|\phi\|^{1+d/p} | \phi|^{1-d/p}.
     \]
     This shows that the transport term is well-defined when $\bv\in L^p(\mathbb{T}^d)$ and $p>d$.

     Proceeding further, we apply Young's inequality with $p'= \frac{2}{1+d/p}=\frac{2p}{p+d}$ and $q' = \frac{2p}{p-d}$ to obtain
     \[
     |(\bv\cdotp \nabla \phi , \phi)| \leq \frac{p-d}{2p}c^{-\frac{2p}{p-d}}|\bv|_p^{\frac{2p}{p-d}} |\phi|^2 + \frac{p+d}{2p} c^{{\frac{2p}{p+d}}}\| \phi\|^2.
     \]
     Note that we require $p>d$ for this. Now choose $c$ such that $\frac{p+d}{2p} c^{{\frac{2p}{p+d}}} =\frac{\kappa}{4}$  then $c = \left( \frac{ p \kappa}{2 (p+d)} \right)^{\frac{p+d}{2p}}$ to obtain that
       \[
     (\bv\cdotp\nabla \phi , \phi) \leq \frac{p-d}{2p}\left( \frac{ p \kappa}{2 (p+d)} \right)^{-\frac{p+d}{p-d}}|\bv|_p^{\frac{2p}{p-d}} |\phi|^2 + \frac{\kappa}{4} \| \phi\|^2.
     \]
     Thus, for $p> d$, upon returning to the energy balance we deduce
     \[
     \frac{1}{2}\frac{d}{dt} |\phi|^2 + \kappa \| \phi\|^2 \leq  \frac{1}{\kappa}\|g\|_*^2 + \frac{\kappa}{4} \|\phi\|^2 + c |\bv|_p^{\frac{2p}{p-d}}  | \phi |^{2} + \frac{\kappa}{4} \| \phi\|^2 
     \]
     Gr\"onwall's inequality gives a.s. on $[0,T]$
     \[
     |\phi(t)|^2 + \int_0^t \| \phi(s)\|^2 ds \leq  e^{c \int_0^T |\bv(r)|_p^{\frac{2p}{p-d}}dr } \left(  | \phi(0)|   + 2 \int_0^t  \|g(s)\|_*^2 ds \right).
     \]
     Hence, for existence of weak solutions one needs $\bv \in L^{\frac{2p}{p-d}}(0,T; L^{p}(\T^d)^d)$, and $g \in L^2(0,T;V^*)$. Under these conditions one has weak solutions belonging to $L^\infty(0,T;L^2(\T^d)) \cap L^2(0,T;V)$.

     Let us now consider the time derivative
     \[
     \frac{\dd\phi}{\dd t} = - \bv\cdotp\nabla \phi + \kappa \Delta \phi + g.
     \]
     Since $g, \Delta \phi \in L^2(0,T;V^*)$ it remains to consider the transport term. For $\varphi \in V$, we have
     \[
     |(\bv\cdotp\nabla \phi, \varphi) | \leq |\bv|_d \| \phi \|\  | \varphi|_{\frac{2d}{d-2}}  
     \]
     as $\phi \in L^\infty(0,T;L^2(\T^d)) \cap L^2(0,T; H^1(\T^d))$ we need that $\bv \in L^\infty(0,T;L^d(\T^d)^d)$.
     On the other hand, we may also argue as follows: 
     \[
     |(\bv\cdotp\nabla \phi, \varphi) | \leq |((\nabla\cdotp \bv)\phi , \varphi) | + |(\bv\phi,  \nabla \varphi) |
     \]
     Then for the second summand,  we get
     \[ 
     |(\bv \phi,  \nabla \varphi) | \leq  |v|_p |\phi|_{\frac{2p}{p-2}}  \| \varphi \|
     \leq |\bv|_p \|\phi\|^{d/p}\ | \phi|^{1-d/p}
     \leq  \frac{d}{p} \| \phi\| + \frac{p-d}{p} |\bv|_p^{\frac{p}{p-d}} \ | \phi| 
     \]
     Using $\phi \in L^\infty(0,T;L^2(\T^d)) \cap L^2(0,T;H^1(\T^d))$ it is sufficient that $ \bv\in L^{\frac{2p}{p-d}}(0,T;L^p(\T^d)^d)$ in order for $\bv \phi\in L^2(0,T;H^{-1}(\T^d)^d)$. Note that the first summand, is zero when $\nabla\cdotp \bv=0$. Hence, the above considerations would be sufficient in this case.  
     On the other hand, if $\nabla\cdotp\bv \neq 0$, then we need an extra condition
     \[
     |((\nabla \cdotp\bv) \phi, \varphi) | \leq | \nabla\cdotp\bv |_p | \phi |_q \| \varphi\| 
     \]
     for $- d \leq -  \frac{d}{p}- \frac{d}{q} + 1 - \frac{d}{2} $ with strict inequality in the case $d=2$, that is $\frac{d}{2} + 1 \geq \frac{d}{p}+ \frac{d}{q}$. Then we interpolate with $-\frac{d}{q} = \alpha - \frac{d}{2}$ and use Young's inequality for $p'$ and $q'$ to get 
     \[
     |((\nabla \cdotp\bv)\phi , \varphi) |\leq | \nabla\cdotp \bv |_p \ | \phi |^{1-\alpha} \| \phi\|^\alpha\  \| \varphi\| \leq \left(\frac{1}{p'} | \nabla\cdotp \bv |^{p'}_p +\frac{1}{q'} | \phi |^{q'(1-\alpha)} \| \phi\|^{q'\alpha} \right)  \| \varphi\| 
     \]
     As  $\phi \in L^\infty(0,T; L^2(\T^d)) \cap L^2(0,T; H^1(\T^d))$, we have that  $| \phi |^{q'(1-\alpha)}$ is bounded by a constant and hence we need that $q'\alpha =1$. Hence $q' = 1/\alpha$ and we get that
     \[
     |((\nabla \cdotp\bv) \phi, \varphi) |\leq \left((1-\alpha) | \nabla\cdotp \bv |^{\frac{1}{1-\alpha}}_p +\alpha | \phi |^{\frac{1-\alpha}{\alpha}} \| \phi\| \right)  \| \varphi\| 
     \]
     and $\alpha = \frac{d}{2} - \frac{d}{q}\geq \frac{d}{p} -1 $ and so
     \[
     \frac{1}{1-\alpha} \geq \frac{p}{2p-d},
     \]
     with strict inequality in the case $d=2$. Hence we require $\nabla\cdotp \bv \in L^{\frac{2p}{2p-d}}(0,T; L^p(\T^d))$ or in $ \bigcap_{p' > \frac{2p}{2p-d}}$ $L^{p'}(0,T ;L^p(\T^d))$ for $d=2$, in order that  $(\nabla \cdotp\bv) \phi\in L^2(0,T; V^*)$.

     The a priori bound gives a subsequence $\phi_n$ which solves
     \[
     \frac{d}{dt} \phi_n + P_n(\bv\cdotp\nabla  \phi_n)- \kappa \Delta \phi_n = P_n g
     \]
     and converges weakly to $\phi$ in $  L^\infty(0,T;L^2(\T^d))$   and in $L^2(0,T;H^1(\T^d))$ and $\frac{d}{dt}\phi_n$ converges weakly to $\frac{d}{dt}\phi$ in $L^2(0,T; V^*)$. We can choose a further subsequence that converges a.s. in time and in $H^1$ in space. Hence in particular, by \cite[Theorem 7.2]{RobinsonBook} we can choose a version of $\phi\in C([0,T]; L^2(\T^d))$.
     Then $\Delta \phi_n$ and $P_n g$ converge weakly to $\Delta \phi$ and $g$ in $L^2(0,T;V^*)$. Recall that the set of all functions which have only finite many modes unequal to zero is dense in $V$. Let $\varphi$ such a function, then  we have that 
     \[
     ( P_n \bv\cdotp \nabla \phi_n, \varphi) = -(\phi_n, \bv \cdotp\nabla \varphi)
     \]
     We have that $\nabla P_n \varphi \in {C}^\infty(\T^d)\cap L^\infty(\T^d)$ and hence $\bv\cdotp\nabla  \varphi \in L^{\frac{2p}{p+2}}(\T^d) \subset V^*$, because $ - \frac{d(p+2)}{2p} = -\frac{d}{2} - \frac{d}{p} < -1 - \frac{d}{2}$.  Hence we obtain by the weak convergence in $V^*$ that
     \[
      (   \phi_n, \bv\cdotp \nabla P_n \varphi)  =   (   \phi_n, \bv\cdotp\nabla  \varphi)  \rightarrow  (   \phi, \bv\cdotp \nabla  \varphi) .
      \]
     Hence
     \[
     \frac{d}{dt} \phi + \bv\cdotp\nabla  \phi- \kappa \Delta \phi =  g
     \]
     holds a.s. in $t$ as an equation in $V^*$.
     
     In order to show that the we have the correct initial condition we note that for any smooth cylinder function $\varphi$ with values in $V$ holds for a.e. $t$ that for large enough $n$ (such that $P_n \varphi = \varphi$)
     \[
     ( \phi_n(t) , \varphi(t) ) - (\phi_n(0), \varphi(0) ) + \int_0^t (\bv(s)\cdotp\nabla  \phi_n(s), \varphi(s))- \kappa ( \Delta \phi_n(s), \varphi(s))  ds = \int_0^t  ( g(s), \varphi(s)) ds
     \]
     All terms converges in $n$, the only thing we have to check that the transport has dominating integrable bound:
     \[
     \left| (\bv(s)\cdotp\nabla  \phi_n(s), \varphi(s))  \right| \leq \| \bv(s) \| \| \varphi(s) \nabla \phi(s) \|_* \leq c   \| \bv(s) \|  \|\nabla \phi(s) \|_*.
     \]
     using that $\varphi \in {C}^\infty_c([0,T] \times \T^d)$.
 
 Next, in order to prove uniqueness, let $\phi$ and $\tilde{\phi}$ be to solutions, then
 \[
 \frac{d}{dt}(\phi - \tilde{\phi}) = - (\bv\cdotp\nabla)(\phi - \tilde{\phi}) - \kappa (\phi - \tilde{\phi})+( g - \tilde{g}).
 \]
As  $ \frac{\dd}{\dd t}(\phi - \tilde{\phi}) \in L^2(0,T, V^*)$ and $ \phi - \tilde{\phi} \in L^2(0,T, V)$ we get 
  \[
 \frac{1}{2}\frac{d}{dt}|\phi - \tilde{\phi}|^2 + \kappa \| \phi - \tilde{\phi}\| = -( (\bv\cdotp\nabla)(\phi - \tilde{\phi}), \phi - \tilde{\phi}) +( g - \tilde{g}, \phi - \tilde{\phi}) \
 \]   
 If $\nabla \cdotp\bv= 0$, then the transport term vanishes. 
 Otherwise, in general, we estimating as before we obtain that
 \[
 | ( (\bv\cdotp\nabla)(\phi - \tilde{\phi}), \phi - \tilde{\phi}) | \leq  c |\bv|_p^{\frac{2p}{p-d}}  \ | \phi - \tilde{\phi} |^{2} + \frac{\kappa}{4} \| \phi - \tilde{\phi} \|^2
 \]
 By the Cauchy-Schwarz inequality, we get that
  \[
 \frac{1}{2}\frac{d}{dt}|\phi - \tilde{\phi}|^2 + \kappa \| \phi - \tilde{\phi}\| = c |\bv|_p^{\frac{2p}{p-d}}  \ | \phi - \tilde{\phi} |^{2} + \frac{\kappa}{4} \| \phi - \tilde{\phi} \|^2
 \]   
which implies that
 \[
\frac{d}{dt}|\phi - \tilde{\phi}|^2 + \left(\lambda_1 \kappa - c |\bv|_p^{\frac{2p}{p-d}}  \right)  | \phi - \tilde{\phi}|  \leq  \frac{2}{\kappa} \| g - \tilde{g}\|^2_*   
 \]   
 Recall that when $\nabla\cdotp \bv=0$, then $c=0$. Thus
 \[
 | \phi(t) - \tilde{\phi}(t) |^2 \leq e^{\left(  c\int_0^t |\bv(s)|_p^{\frac{2p}{p-d}} ds  - \lambda_1 \kappa t \right)} \left(  | \phi(0) - \tilde{\phi}(0) |^2 + \frac{2T}{\kappa} \esssup_{t \in [0,T]}\| g(s) - \tilde{g}(s)\|^2_*  \right)
  \]
 Note that since $\T^d=[0,2\pi]^d$, we have $\lam_1=1$.
 
 Strong solutions are given by multiplying with $-P_n \Delta \phi$ to obtain
 \[
 \frac{1}{2} \frac{d}{dt} \| \phi_n\| + \kappa | \Delta \phi_n | = - (\bv\cdotp\nabla \phi_n , \Delta \phi_n) + (g, \Delta \phi_n) 
 \]
 The transport term can be estimated as follows: by integration by parts we obtain
 \[
- (\bv\cdotp\nabla \phi_n , \Delta \phi_n)= \sum_{k} ( \partial_k \bv\cdotp \nabla \phi_n, \partial_k \phi_n ) + \frac{1}{2}( \bv\cdotp\nabla ( \partial_k \phi_n)^2)
 \]
 If $\nabla\cdotp \bv =0$, then upon interpolating and applying Young's inequality we get
  \[
\left| \sum_{k} ( \partial_k \bv \cdotp\nabla \phi_n, \partial_k \phi_n )  \right| \leq c \|\bv\| | \nabla \phi_n|_4^2\leq c \|\bv\| \| \phi_n\|_2\, \| \phi_n \| \leq \frac{\kappa}{4} \| \phi_n\|_2^2 + \frac{c}{\kappa}  \| \bv\|^2 \| \phi_n\|^2
 \]
 So all together we get 
  \[
  \frac{d}{dt} \| \phi_n\| + \kappa | \Delta \phi_n | =  \frac{c}{\kappa}  \| \bv\|^2 \| \phi_n\|^2  + \frac{2}{\kappa} |g|^2 
 \]
 So if $\bv \in L^2(0,T; \bbV)$, we get an a priori bound and weak convergence for a subsequence in $L^\infty(0,T; H)\cap L^2(0,T; V)$.
 
 For $\varphi \in H$ we get for $-\frac{d}{p} < 1 - \frac{d}{2}$, which is $p > \frac{2d}{d-2}$ and $\frac{1}{p}+ \frac{1}{q} = \frac{1}{2}$ that
 \[
 |  \bv\cdotp\nabla \phi|   = |\bv|_p \, |\nabla \phi |_q.
 \]
 Upon interpolating with $ - \frac{d}{q} = \alpha (1 - \frac{d}{2}) + (1- \alpha) (- \frac{d}{2}) = \alpha - \frac{d}{2}$ , that is $\alpha = \frac{d}{2} - \frac{d}{q} =  \frac{d}{p}$, we get
 \[
 |  \bv\cdotp\nabla \phi| \leq c \| \bv\|\, \|\phi\|^{\frac{p-d}{p}} \| \phi\|_2^{\frac{d}{p}} 
 \]
 Using that $\phi \in L^\infty(0,T; V)\cap L^2(0,T; V)$, we then apply Young's inequality with $p = \frac{p}{d}$ and $q = \frac{p}{p-d}$ to get that
 \[
 \leq c \| \bv\|^{\frac{p}{p-d}} \| \phi\| + c \| \phi\|_2.
 \]
 The right-hand side is square-integrable over $[0,T)$ for all $T>0$, provided that $\bv\in L^{\frac{2p}{p-d}}(0,T; \bbV)$ 
 Hence we get that 
 \[
 \frac{d}{dt} \phi_n - P_n \bv\cdotp\nabla \phi_n - \kappa \Delta \phi_n = P_n g
 \]
 is uniformly bounded in $L^2(0,T;H)$, we get also the weak convergence of $\frac{d}{dt} \phi_n$ to $\frac{d}{dt} \phi$. So we get in particular that $\phi \in C([0,T]; V)$ and  for the integrated equation we get that
 \[
 (\phi_n(t), \varphi(t)) -( \phi_n(0), \varphi(0)) + \kappa \int_0^t ( \Delta \phi_n(s), \varphi(s)) ds - \int_0^t ( \phi_n(s), \bv(s)\cdotp\nabla \varphi(s) ) ds = \int_0^t (g(s), \varphi(s))ds
 \]
 where all terms converge in the limit $n \rightarrow \infty$ and hence we have that the  strong solution is also a weak solution.
       \end{proof}
       
\begin{Rmk}
Note that if $p=d >2$ then 
      \[
     (\bv\cdotp\nabla \phi , \phi) \leq |\bv|_p \| \phi\|^2.
     \]
In the energy balance, this would then imply
     \[
     \frac{1}{2}\frac{d}{dt} |\phi |^2 + \kappa \| \phi\|^2 \leq \|g\|_* \|\phi\| + c |\bv|_p \| \phi\|^2.
     \]
Hence, in this case, existence can only be guaranteed for $c|\bv|_p$ small enough relative to $\kap$. Thus, $p > d$ is necessary to accommodate large velocity fields within this framework.
\end{Rmk}

%


\begin{multicols}{2}
\noindent Jochen Broecker$^1$\\
{\footnotesize
School of Mathematical, Physical, and Computational Sciences and
Centre for the Mathematics of Planet Earth\\
University of Reading \\
Web: \url{https://www.met.reading.ac.uk/~pt904209/}\\
Email: \url{j.broecker@reading.ac.uk}\\
}

\noindent Giulia Carigi$^2$\\
{\footnotesize
Department of Statistics\\
Indiana University Bloomington \\
Web: \url{https://sites.google.com/view/giuliacarigi/home}\\
Email: \url{gcarigi@iu.edu} \\
}

\columnbreak

\noindent Tobias Kuna$^3$\\
{\footnotesize
Dipartimento di Ingegneria e Scienze dell'Informazione e Matematica \\
Universit\`a degli Studi Dell'Aquila \\
Web: \url{https://smaq.mat-univaq.com/people/tobias-kuna}\\
Email: \url{tobias.kuna@univaq.it}\\
}

\noindent Vincent R. Martinez$^4$\\
{\footnotesize
Department of Mathematics \& Statistics\\
CUNY Hunter College \\
Department of Mathematics \\
CUNY Graduate Center \\
Web: \url{http://math.hunter.cuny.edu/vmartine/}\\
Email: \url{vrmartinez@hunter.cuny.edu}\\
}

\end{multicols}

\end{document}